\documentclass[a4paper,12pt]{amsart}

\usepackage[margin=4cm]{geometry}
\usepackage[T1]{fontenc}

\usepackage{amsthm, amsmath, amssymb, bbm}
\usepackage{mathrsfs}

\usepackage{tikz}
\usetikzlibrary{calc}
\usetikzlibrary{intersections}
\usepackage{tikz-3dplot}

\usepackage{enumerate}
\usepackage{graphics}

\usepackage[all]{xy}
\newdir{ >}{{}*!/-10pt/@{>}}

\usepackage{color}

\usepackage[colorlinks=true, pdfstartview=FitV, linkcolor=blue, %
citecolor=blue, urlcolor=blue]{hyperref}

\usepackage{marginnote}
\usepackage{xspace}
\usepackage{ulem}
\makeatletter

\newcommand{\nc}{\newcommand}
\newenvironment{rouge}
{\relax\color{red}}
{\hspace*{.3ex}\relax}
\newcommand{\ber}{\begin{rouge}{}\marginnote{\mbox{$\bullet$}}{}}
\newcommand{\er}{\end{rouge}}
\newcommand{\bera}{\begin{rouge}{}\marginnote{\fbox{\scshape\lowercase{A}}}{}}
\newcommand{\berm}{\begin{rouge}{}\marginnote{\fbox{\scshape\lowercase{M}}}{}}

\newcommand{\erm}{\end{rouge}}

\newenvironment{bleu}
{\relax\color{blue}}
{\hspace*{.3ex}\relax}
\newcommand{\beb}{\begin{bleu}}
\newcommand{\bebm}{\begin{bleu}{}\marginnote{\fbox{\scshape\lowercase{M}}}{}}
\newcommand{\beba}{\begin{bleu}{}\marginnote{\fbox{\scshape\lowercase{A}}}{}}
\newcommand{\eb}{\end{bleu}}

\usepackage{units}

\usepackage{accents}

\usepackage{tensor}

\theoremstyle{plain}
\newtheorem{introthm}{Theorem}[section]
\newtheorem{theorem}{Theorem}[subsection]
\newtheorem*{theorem*}{Theorem}
\newtheorem{corollary}[theorem]{Corollary}
\newtheorem{proposition}[theorem]{Proposition}
\newtheorem{lemma}[theorem]{Lemma}

\newtheorem{sublemma}[theorem]{Sublemma}
\newtheorem{lemmadefinition}[theorem]{Lemma-Definition}

\theoremstyle{definition}
\newtheorem{definition}[theorem]{Definition}

\newtheorem*{example*}{Example}
\newtheorem{notation}[theorem]{Notation}
\newtheorem{remark}[theorem]{Remark}
\newtheorem{convention}[theorem]{Convention}
\nc{\Lemma}{\begin{lemma}}
\nc{\enlemma}{\end{lemma}}
\nc{\Prop}{\begin{proposition}}
\nc{\enprop}{\end{proposition}}
\nc{\Def}{\begin{definition}}
\nc{\edf}{\end{definition}}

\renewcommand{\emptyset}{\varnothing}

\nc{\scup}{\mathop{\scalebox{.8}{$\displaystyle\bigcup$}}\mspace{1mu}\limits}
\nc{\scap}{\mathop{\scalebox{.8}{$\displaystyle\bigcap$}}\limits}
\nc{\ssqcup}{\mathop{\scalebox{.8}{$\displaystyle\bigsqcup$}}\limits}
\newcommand{\DUnion}{\bigsqcup\limits}

\newcommand{\union}{\cup}
\newcommand{\Union}{\bigcup\limits}

\newcommand{\C}{\mathbb{C}}

\newcommand{\R}{\mathbb{R}}
\newcommand{\Q}{\mathbb{Q}}
\newcommand{\Z}{\mathbb{Z}}

\DeclareMathOperator{\id}{id}

\newcommand{\derived}[1]{\mathrm{#1}}

\newcommand{\derd}{\derived{D}}
\newcommand{\dere}{\derived{E}}
\newcommand{\dereb}{\dere^{\mathrm{b}}}
\newcommand{\derr}{\derived{R}}
\newcommand{\derl}{\derived{L}}
\nc{\derb}{\derd^{\mathrm{b}}}
\newcommand{\BDC}{\derd^{\mathrm{b}}}

\newcommand{\TDC}{\derd}

\nc{\soplus}{\scalebox{.65}{\raisebox{.2ex}{$\displaystyle\bigoplus$}}}

\newcommand{\DSum}{\mathop{\bigoplus}}
\newcommand{\dsum}[1][]{\mathbin{\oplus_{#1}}}

\newcommand{\ilim}[1][]{\mathop{\varinjlim}\limits_{#1}}

\renewcommand{\to}[1][]{\xrightarrow{#1}}
\newcommand{\from}[1][]{\xleftarrow{#1}}

\newcommand{\isoto}[1][]{\xrightarrow[#1]{%
{\raisebox{-.6ex}[0ex][0ex]{$\mspace{1mu}\sim\mspace{2mu}$}}}}

\newcommand{\Endo}[1][]{\mathrm{End}_{\raise1.5ex\hbox to.1em{}#1}}

\newcommand{\Hom}[1][]{\mathrm{Hom}_{\raise1.5ex\hbox to.1em{}#1}}

\newcommand{\RHom}[1][]{\derr\mathrm{Hom}_{\raise1.5ex\hbox to.1em{}#1}}

\newcommand{\Ext}[2][]{\mathrm{Ext}_{\raise1.5ex\hbox to.1em{}#1}^{#2}}

\newcommand{\Mod}{\mathrm{Mod}}

\newcommand{\Tens}[1][]{\mathbin{\otimes_{\raise1.5ex\hbox to-.1em{}#1}}}

\newcommand{\LTens}[1][]{\mathbin{\otimes_{\raise1.5ex\hbox to-.1em{}#1}^{\derl}}}

\newcommand{\Tor}[2][]{\mathrm{Tor}^{\raise1.5ex\hbox to.1em{}#1}_{#2}}

\newcommand{\sheaffont}[1]{\mathcal{#1}}

\def\she{\sheaffont{E}}

\def\shi{\sheaffont{I}}

\def\shl{\sheaffont{L}}
\def\shm{\sheaffont{M}}
\def\shn{\sheaffont{N}}

\def\shp{\sheaffont{P}}

\def\shr{\sheaffont{R}}

\newcommand{\sect}{\varGamma}
\newcommand{\rsect}{\derr\varGamma}

\newcommand{\shendo}[1][]{{\sheaffont{E}nd}_{\raise1.5ex\hbox to.1em{}#1}}

\renewcommand{\hom}[1][]{{\sheaffont{H}om}_{\raise1.5ex\hbox to.1em{}#1}}
\newcommand{\aut}[1][]{{\sheaffont{A}ut}_{\raise1.5ex\hbox to.1em{}#1}}
\newcommand{\inn}[1][]{{\sheaffont{I}nn}_{\raise1.5ex\hbox to.1em{}#1}}

\newcommand{\rhom}[1][]{{\derr\sheaffont{H}om}_{\raise1.5ex\hbox to.1em{}#1}}

\newcommand{\ext}[2][]{{\sheaffont{E}xt}_{\raise1.5ex\hbox to.1em{}#1}^{#2}}

\newcommand{\thom}[1][]{{\sheaffont{T}hom}_{\raise1.5ex\hbox to.1em{}#1}}

\newcommand{\tens}[1][]{\mathbin{\otimes_{\raise1.5ex\hbox to-.1em{}#1}}}

\newcommand{\ltens}[1][]{\mathbin{\otimes_{\raise1.5ex\hbox to-.1em{}#1}^{\derl}}}

\newcommand{\tor}[2][]{{\sheaffont{T}or}^{\raise1.5ex\hbox to.1em{}#1}_{#2}}

\newcommand{\etens}[1][]{\mathbin{\boxtimes_{\raise1.5ex\hbox to-.1em{}#1}}}

\DeclareMathOperator{\supp}{supp}

\newcommand{\oim}[1]{#1_*}

\newcommand{\roim}[1]{\derr#1_*}

\newcommand{\reim}[1]{\derr#1_{\mspace{.5mu}!}\mspace{2mu}}

\newcommand{\reeim}[1]{\derr#1_{\mspace{1mu}!!}\mspace{1mu}}

\newcommand{\opb}[1]{#1^{-1}}

\newcommand{\epb}[1]{#1^{\mspace{1.5mu}!}\mspace{2mu}}

\newcommand{\Gr}{\mathop{\mathrm{Gr}}\nolimits}

\newcommand{\tenstop}[1][]{\mathbin{\hat{\otimes}_{\raise1.5ex\hbox to-.1em{}#1}}}

\newcommand{\homtop}[1][]{\sheaffont{L}_{\raise1.5ex\hbox to.1em{}#1}}

\newcommand{\Homtop}[1][]{\mathrm{L}_{\raise1.5ex\hbox to.1em{}#1}}

\DeclareMathOperator{\ord}{ord}

\newcommand{\D}{\sheaffont{D}}

\renewcommand{\O}{\sheaffont{O}}

\DeclareMathOperator{\chv}{char}

\newcommand{\detens}[1][]%
{\mathbin{\boxtimes_{\raise1.5ex\hbox to-.1em{}#1}^{\mspace{2mu}\mathsf{D}}}}

\newcommand{\doim}[1]{{\mathsf{D}#1}_*\mspace{1mu}}

\newcommand{\dopb}[1]{{\mathsf{D}#1}^{\mspace{1mu}*}}

\newcommand{\dtens}[1][]{\mathbin{\otimes_{\raise1.5ex\hbox to-.1em{}#1}^{\mathsf{D}}}}

\newcommand{\coh}{\mathrm{coh}}

\newcommand{\hol}{\mathrm{hol}}

\renewcommand{\leq}{\leqslant}
\renewcommand{\geq}{\geqslant}

\newcommand{\field}[1][]{\mathbf{k}}
\newcommand{\ind}{\mathrm{I}\mspace{2mu}}
\newcommand{\ifield}{\ind\field}

\newcommand{\Rc}{{\R\text-\mathrm{c}}}

\newcommand{\cconv}{\mathbin{\mathop\star\limits^+}}
\newcommand{\ctens}{\mathbin{\mathop\otimes\limits^+}}
\newcommand{\cetens}{\mathbin{\mathop\boxtimes\limits^+}}
\newcommand{\cihom}{{\derr\shi hom}^+}

\renewcommand{\Re}{\operatorname{Re}}
\renewcommand{\Im}{\operatorname{Im}}

\newcommand{\ihom}[1][]{{\shi hom}_{\raise1.5ex\hbox to.1em{}#1}}
\newcommand{\rihom}[1][]{{\derr\mspace{2mu}\shi hom}_{\raise1.5ex\hbox to.1em{}#1}}
\newcommand{\ii}[1][]{{\sheaffont{I}h}_{\raise1.5ex\hbox to.1em{}#1}}

\newcommand{\indlim}[1][]{\mathop{\text{\rm``$\varinjlim$''}}\limits_{#1}}

\newcommand{\dcomp}[1][]{\mathbin{\circ_{\raise1.5ex\hbox to-.1em{}#1}^{\mathsf{D}}}}

\newcommand{\enh}{\derived{E}}

\newcommand{\solE}[1][X]{\mathcal{S}ol^{\mspace{1mu}\enh}_{#1}}

\newcommand{\fhom}{\rhom^\enh}
\newcommand{\FHom}{\RHom^\enh}
\newcommand{\fihom}{\rihom^\enh}
\newcommand{\nuhom}[1][]{\nu{hom}^\enh_{#1}}

\newcommand{\Eoim}[1]{{\enh#1}_*}

\newcommand{\Eeeim}[1]{{\enh#1}_{!!}}

\newcommand{\Eopb}[1]{{\enh#1}^{-1}}

\newcommand{\Eepb}[1]{{\enh\mspace{1mu}#1}^{\mspace{1.5mu}!}}

\newcommand{\LE}{\operatorname{L^\enh}}
\newcommand{\RE}{\operatorname{R^\enh}}

\newcommand{\tot}{{\operatorname{tot}}}

\newcommand{\semicolon}{\nobreak \mskip2mu\mathpunct{}\nonscript\mkern-\thinmuskip{;}\mskip6mu plus1mu\relax}

\newcommand{\defeq}{\mathbin{:=}}
\newcommand{\eqdef}{\mathbin{=:}}
\newcommand{\bl}{\bigl(}
\newcommand{\br}{\bigr)}
\newcommand{\To}[1][]{\xrightarrow[]{\mspace{10mu}{#1}\mspace{10mu}}}

\newenvironment{myarray}[1]{\relax\setlength{\arraycolsep}{1pt}

\begin{array}{#1}}{\end{array}\relax}

\newcommand{\ba}{\begin{myarray}}
\newcommand{\ea}{\end{myarray}}

\newcommand{\be}{\begin{enumerate}}
\newcommand{\ee}{\end{enumerate}}
\newcommand{\bnum}{\be[{\rm(i)}]}

\nc{\bwr}{\mbox{\large{$\wr$}}}
\nc{\vphi}{\varphi}
\nc{\seteq}{\mathbin{:=}}
\nc{\noi}{\noindent}
\nc{\ro}{{\rm(}}
\nc{\rf}{{\rm)}\xspace}
\nc{\ms}{\mspace}
\nc{\sbcup}{\mathop{\scalebox{0.75}{$\displaystyle\bigcup$}}}
\nc{\ol}{\overline}
\nc{\scbul}{{\,\raise1pt\hbox{$\scriptscriptstyle\bullet$}\,}}
\nc{\set}[2]{\left\{#1\;\semicolon\; #2 \right\}}
\nc{\extp}{\mathop{\raisebox{.3ex}{\scalebox{0.8}{$\displaystyle\bigwedge$}}}\limits}

\newenvironment{myequation}
{\relax\setlength{\arraycolsep}{1pt}\begin{eqnarray}}
{\end{eqnarray}}
\newenvironment{myequationn}
{\relax\setlength{\arraycolsep}{1pt}\begin{eqnarray*}}
{\end{eqnarray*}}

\newenvironment{myalign}
{\relax\begin{align}}
{\end{align}}
\newenvironment{myalignn}
{\relax\begin{align*}}
{\relax\end{align*}}

\nc{\eq}{\begin{myequation}}
\nc{\eneq}{\end{myequation}}
\nc{\eqn}{\begin{myequationn}}
\nc{\eneqn}{\end{myequationn}}

\nc{\eqa}{\begin{myalign}}
\nc{\eneqa}{\end{myalign}}
\nc{\eqan}{\begin{myalignn}}
\nc{\eneqan}{\end{myalignn}}

\nc{\on}{\operatorname}
\nc{\Ind}{\on{Ind}}
\nc{\Proof}{\begin{proof}}
\nc{\QED}{\end{proof}}
\nc{\cor}{\field}
\nc{\tone}{\To[+1]}
\renewcommand{\ge}{\geq}
\renewcommand{\le}{\leq}

\newcommand{\chom}{\rhom^+}
\newcommand{\Eeim}[1]{{\enh#1}_{!}}

\renewcommand{\tor}{\mathrm{tor}}

\newcommand{\bclose}[1]{{\accentset{\vee}{#1}}}
\newcommand{\unbordered}[1]{{\accentset{\circ}{#1}}}

\newcommand{\cR}{{\overline\R}}

\nc{\unb}{\unbordered}
\nc{\eps}{\varepsilon}
\nc{\inb}{\inbordered}
\nc{\colim}{\varinjlim\limits}
\nc{\ssubset}{\subset\ms{-3mu}\subset}
\nc{\al}{\alpha}
\nc{\qtq}[1][and]{\quad\text{#1}\quad}
\nc{\qt}[1]{\quad\text{#1}}
\nc{\olG}[1][f]{{\overset{\ms{4mu}\rule[-.05ex]{1.6ex}{.115ex}}{\Gamma}}_{%
\ms{-3mu}#1}}

\newcommand{\cM}{\bclose{M}}

\newcommand{\cN}{\bclose{N}}
\nc{\cf}{\bclose{f}}

\newcommand{\inbordered}[1]{{#1_\infty}}

\newcommand{\quot}{\derived Q}

\newcommand{\Efield}{\field^\enh}

\newcommand{\V}{\mathbb{V}}
\newcommand{\W}{\V^*}
\newcommand{\PP}{\mathbb{P}}

\newcommand{\bb}{\PP^*}

\newcommand{\Ex}{\mathbb{E}}
\newcommand{\ex}{\mathsf{E}}

\newcommand{\Lap}{\mathsf{L}}
\newcommand{\lap}{{}^\Lap}
\newcommand{\Lapa}{\Lap^r}
\newcommand{\lapa}{{}^{\Lapa}}

\newcommand{\SSE}{SS^\enh}
\newcommand{\SSEo}{\SSE_\rho}

\newcommand{\F}{\mathsf{F}}
\renewcommand{\Gr}{\mathsf{Gr}}
\newcommand{\G}{\mathsf{G}}

\newcommand{\cX}{\bclose{X}}

\newcommand{\cY}{\bclose{Y}}

\newcommand{\etale}[1]{\operatorname{\acute{e}t}\bigl(#1\bigr)}

\newcommand{\ram}{p}
\nc{\tM}{\widetilde{M}}
\nc{\tX}{\widetilde{X}}
\nc{\ti}{{\tilde\imath}}
\nc{\tj}{{\tilde\jmath}}

\newcommand{\eprec}{\preccurlyeq}
\newcommand{\aleq}[2][]{\mathrel{\preccurlyeq^{#1}_{#2}}}
\newcommand{\aless}[2][]{\mathrel{\prec^{#1}_{#2}}}
\newcommand{\asim}[1]{\underset{#1}\sim}

\newcommand{\St}{\operatorname{St}}

\newcommand{\range}[2]{#1 \mathbin{\rhd} #2}

\newcommand{\dotin}{\mathbin{\dot\in}}
\newcommand{\dotowns}{\mathbin{\dot\owns}}

\newcommand{\oshp}{\smash{\overline\shp}}

\newcommand{\ens}[1]{\{#1\}}

\makeatother

\begin{document}
\title[Enhanced Fourier transform]{A microlocal approach to 
the enhanced Fourier-Sato transform in dimension one}

\author[A.~D'Agnolo]{Andrea D'Agnolo}
\address[Andrea D'Agnolo]{Dipartimento di Matematica\\
Universit{\`a} di Padova\\
via Trieste 63, 35121 Padova, Italy}
\thanks{The research of A.D'A.\
was supported by GNAMPA/INdAM and by grant CPDA159224 of
Padova University.}
\email{dagnolo@math.unipd.it}

\author[M.~Kashiwara]{Masaki Kashiwara}
\thanks{The research of M.K.\
was supported by Grant-in-Aid for Scientific Research (B)
15H03608, Japan Society for the Promotion of Science.}
\address[Masaki Kashiwara]{Research Institute for Mathematical Sciences, Kyoto University,
Kyoto 606-8502, Japan \& Korea Institute for Advanced Study, Seoul 02455, Korea}
\email{masaki@kurims.kyoto-u.ac.jp}

\keywords{Fourier transform, holonomic D-modules, Riemann-Hilbert
correspondence, enhanced ind-sheaves, stationary phase formula, Stokes filtered local systems, microlocal study of sheaves}
\subjclass[2010]{Primary 34M35, 32S40, 32C38, 30E15}
\date{September, 1, 2017}

\maketitle

\begin{abstract}
Let $\shm$ be a holonomic algebraic $\D$-module on the affine line.
Its exponential factors are Puiseux germs describing the growth 
of holomorphic solutions to $\shm$ at irregular points.
The stationary phase formula states that the exponential factors
of the Fourier transform of $\shm$ are obtained by Legendre transform
from the exponential factors of $\shm$. We give a microlocal proof
of this fact, by translating it in terms of enhanced ind-sheaves through the Riemann-Hilbert correspondence.
\end{abstract}

\tableofcontents

\section{Introduction}
\addtocontents{toc}{\protect\setcounter{tocdepth}{1}}
\numberwithin{equation}{section}

\subsection{}\label{sse:begin}
Let $\V$ be a one-dimensional complex vector space, with coordinate $z$, and
$\W$ its dual, with dual coordinate $w$.
The Fourier transform, originally introduced as an integral transform with kernel associated to $e^{-zw}$, has various realizations.

At the Weyl algebra level, it is the isomorphism $P\mapsto\lap P$ given by $z\mapsto -\partial_w$, $\partial_z\mapsto w$.
This induces an equivalence between holonomic algebraic $\D$-modules on $\V$ and on $\W$, that we still denote by $\shm\mapsto\lap\shm$.

At a microlocal level, the Fourier transform is attached to the symplectic transformation
\begin{equation}
\label{eq:introchi}
\chi_\rho\colon T^*\V \to T^*\W, \quad (z,w)\mapsto (w,-z),
\end{equation}
where we used the identifications $T^*\V = \V\times\W$ and $T^*\W = \W\times\V$ of the cotangent bundles. This induces the Legendre transform from Puiseux germs on $\V$ (i.e., holomorphic functions on small sectors, which admit a Puiseux series expansion) to Puiseux germs on $\W$.

The exponential factors of a holonomic $\D_\V$-module $\shm$ are Puiseux germs on $\V$ describing the growth 
of its holomorphic solutions at irregular points.
We are interested here in the stationary phase formula, 
which states that the exponential factors
of $\lap\shm$ are obtained by Legendre transform
from the exponential factors of $\shm$.

The Riemann-Hilbert correspondence allows a restatement of this fact in terms
of yet another realization of the Fourier transform, that is the Fourier-Sato transform for enhanced ind-sheaves. 
In this setting, we provide a microlocal proof of the stationary phase formula.

\smallskip
Let us explain all of the above in more details.

\subsection{}\label{sse:LT}
Let $a$ be a singular point of $\shm$ (this includes $a=\infty$, where $\shm$ is naturally extended as a meromorphic connection).
Let $z_a$ be the local coordinate given by $z_a=z-a$ if $a\in\V$, and $z_\infty=z^{-1}$.
Denote by $S_a\V$ the circle of tangent directions at $a$.
Let $\shp_{S_a\V}$ be the sheaf on $S_a\V$ whose stalk
at $\theta\in S_a\V$ is the set of holomorphic functions on small sectors around $\theta$, which admit a Puiseux series expansion at $a$. Its sections are called \emph{Puiseux germs}.
The sheaf $\oshp_{S_a\V}$ is the quotient of $\shp_{S_a\V}$ modulo bounded functions.

The Hukuhara-Levelt-Turrittin theorem describes both the formal and the asymptotic structure of $\shm$ at $a$, as follows. 

At the formal level, after a ramification $u^\ram=z_a$ and formal completion by $\C(\!(u)\!)$, $\shm$ decomposes as a finite direct sum of modules of the form $\she^f\dtens\shr_f$.  Here, $f\in \C\{u\}[u^{-1}]$ is a meromorphic germ, the $\D$-module $\she^f$ corresponds to the meromorphic connection $d+d f$, we denote by $\dtens$ the tensor product for $\D$-modules, and $\shr_f$ is a regular holonomic $\D$-module.

For a chosen determination of $z_a^{1/\ram}$ at $\theta\in S_a\V$, let us still denote by $f$ the Puiseux germ with expansion $f\in \C\{z_a^{1/\ram}\}[z_a^{-1/\ram}]$ as above. (We write $(a,\theta,f)$ instead of $f$ if we need more precision.)
As the isomorphism class of $\she^f$ only depends on the equivalence class $[f]\in\oshp_{S_a\V}$, we can assume that different summands correspond to different equivalence classes.
Then, the set $N_\theta^{>0}$ of those $f$'s with $\shr_f\neq0$ is called a system of exponential factors of $\shm$ at $\theta$, and the rank $N(f)$ of $\shr_f$ is called the multiplicity of $f$.

At the asymptotic level, the Hukuhara-Levelt-Turrittin theorem 
states that for any $\theta\in S_a\V$ there is a basis $\{u_{f,i}\}_{f\in N^{>0}_\theta,i=1,\dots,N(f)}$ of holomorphic solutions to $\shm$ on a small sector around $\theta$, such that $e^{-f}u_{f,i}$ has moderate asymptotic growth at $a$.

\subsection{}\label{sse:Airy}
Let us illustrate the stationary phase formula by an example.

\smallskip
Let $\shn=\D_{\W}/\D_{\W}Q$ be the $\D_{\W}$-module associated with the Airy operator $Q=\partial_w^2-w$. It is regular everywhere except at 
$w=\infty$. Recall that the Airy equation $Q\psi=0$ has two entire solutions
$\psi_\pm$ with the following asymptotics at $\eta=1\cdot\infty\in S_\infty\W$:
\eqn
\psi_\pm(w)=e^{\pm\frac{2}{3}w^{3/2}}w^{-1/4}\bl 1+\mathrm{O}(w^{-1})\br
\qt{on $|\arg(w)|<\frac\pi3$,}
\eneqn
where we chose the determination of $w^{1/4}$ with $1^{1/4} = 1$.
Then, $g_\pm(w) = \pm\frac{2}{3}w^{3/2}$ are the exponential factors of $\shn$ at $\eta$. 

At the level of the Weyl algebra, one has $Q=\lap P$, for $P=z^2-\partial_z$. Hence $\shn\simeq\lap\shm$, where $\shm=\D_{\V}/\D_{\V} P$ is regular everywhere except at $z=\infty$.
The equation $P\vphi=0$ has the entire solution $\vphi(z) = e^{z^3/3}$.
Here there is no ramification, and $f(z)=z^3/3$ is the only exponential factor of $\shm$ at any $\theta\in S_\infty\V$.

Let us show how to compute the exponential factors $g_\pm$ of $\shn$ from the exponential factor $f$ of $\shm$.

Recall the symplectic transformation $\chi_\rho$ in \eqref{eq:introchi}.
To $f$, we associate the Lagrangian
\[
C_f \defeq (\text{graph of }df \text{ in } T^*\V)
=\{(z,w)\semicolon w=z^2\}.
\]
As pictured in Figure~\ref{fig:lagrangians}, one has
\[
\chi_\rho(C_f)=\{(w,z)\semicolon z^2 = w\}\subset T^*\W,
\]
which is a ramified double cover of $\W$. 

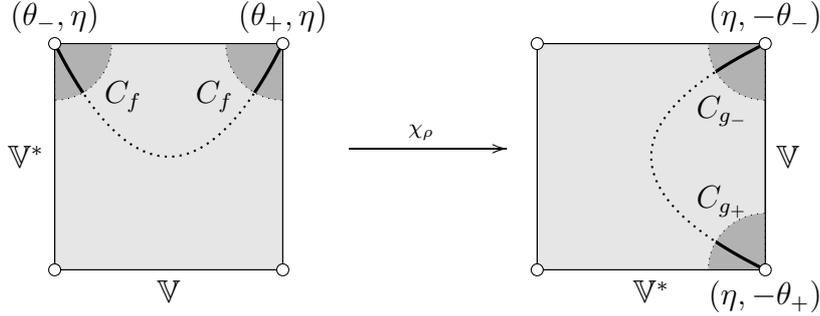
\begin{figure}
\xymatrix@C=5em{
\begin{tikzpicture}[scale=1.5,
	baseline=(O.base)]
\draw (0,0) coordinate (O) ;
\draw[fill=black!10] (-1,-1) -- (-1,1) -- (1,1) -- (1,-1) -- cycle ;
\filldraw[fill=black!30,draw,dotted] (-1,1) -- +(-90:.5) arc (-90:0:.5) ; 
\filldraw[fill=black!30,draw,dotted] (1,1) -- +(-90:.5) arc (-90:-180:.5) ; 
\draw (-1,-1) -- (-1,1) -- (1,1) -- (1,-1) -- cycle ;
\draw[domain=-1:1,smooth,variable=\x,thick,dotted] plot ({\x},{\x*\x});
\begin{scope}
    \clip (-1,1) circle (.5);
    \draw[domain=-1:1,smooth,variable=\x,very thick] plot ({\x},{\x*\x});
\end{scope}
\begin{scope}
    \clip (1,1) circle (.5);
    \draw[domain=-1:1,smooth,variable=\x,very thick] plot ({\x},{\x*\x});
\end{scope}
\foreach \i in {-1,1} {
\foreach \j in {-1,1} {
\draw[fill=white,draw=black] (\i,\j) circle (1.5pt) ;
}}
\draw (-1,0) node[left]{$\W$} ;
\draw (0,-1) node[below]{$\V$} ;
\draw (-1,1) node[above]{$(\theta_-,\eta)$} ;
\draw (1,1) node[above]{$(\theta_+,\eta)$} ;
\draw (.4,.3) node[above]{$C_f$} ;
\draw (-.4,.3) node[above]{$C_f$} ;
\end{tikzpicture}
\ar[r]^{\chi_\rho}
& 
\begin{tikzpicture}[scale=1.5,
	baseline=(O.base)]
\draw (0,0) coordinate (O) ;
\draw[fill=black!10] (-1,-1) -- (-1,1) -- (1,1) -- (1,-1) -- cycle ;
\filldraw[fill=black!30,draw,dotted] (1,-1) -- +(90:.5) arc (90:180:.5) ; 
\filldraw[fill=black!30,draw,dotted] (1,1) -- +(-90:.5) arc (-90:-180:.5) ; 
\draw (-1,-1) -- (-1,1) -- (1,1) -- (1,-1) -- cycle ;
\draw[domain=-1:1,smooth,variable=\x,thick,dotted] plot ({\x*\x},{\x});
\begin{scope}
    \clip (1,1) circle (.5);
    \draw[domain=-1:1,smooth,variable=\x,very thick] plot ({\x*\x},{\x});
\end{scope}
\begin{scope}
    \clip (1,-1) circle (.5);
    \draw[domain=-1:1,smooth,variable=\x,very thick] plot ({\x*\x},{\x});
\end{scope}
\foreach \i in {-1,1} {
\foreach \j in {-1,1} {
\draw[fill=white,draw=black] (\i,\j) circle (1.5pt) ;
}}
\draw (-1,0) node[left]{} ;
\draw (1,0) node[right]{$\V$} ;
\draw (0,-1) node[below]{$\W$} ;
\draw (1,1) node[above]{$(\eta,-\theta_-)$} ;
\draw (1,-1) node[below]{$(\eta,-\theta_+)$} ;
\draw (.3,.4) node[right]{$C_{g_-}$} ;
\draw (.3,-.4) node[right]{$C_{g_+}$} ;
\end{tikzpicture}
} 
\caption{The microlocal Fourier transform.}\label{fig:lagrangians}
\end{figure}

For $(z,w)\in C_f$, one has
$w\to\eta$ if and only if either $z\to\theta_+$ or $z\to\theta_-$,
where $\theta_\pm=\pm1\cdot \infty\in S_\infty\V$. 
Thus, denoting by $C_{(\infty,\theta_\pm,f)}$ the restriction of $C_f$ for $z$ in a small sector around $\theta_\pm$, one has
\[
\chi_\rho(C_{(\infty,\theta_\pm,f)})
= \{z = \pm w^{1/2}\}
= C_{(\infty,\eta,g_\pm)},
\]
since $g_\pm(w) = \pm \frac23 w^{3/2}$ are primitives of $\pm w^{1/2}$.

The above microlocal construction describes the Legendre transform
at the level of Puiseux germs, that we denote by $\Lap$ (see \S\ref{sse:Legendre}, where it is also explained how to fix the choice of a primitive). 
In this example, as pictured in Figure~\ref{fig:Puiseux}, it gives $\Lap(\infty,\theta_\pm,f)=(\infty,\eta,g_\pm)$.

\begin{figure}
\xymatrix@C=5em{
\begin{tikzpicture}[scale=.5,
	baseline=(O.base)]
\draw (0,0) coordinate (O) ;
\filldraw[fill=black!10] (0,0) circle (3) ;
\filldraw[fill=black!30, draw=black,dotted] (-15:2) -- (-15:3) arc (-15:15:3) -- (15:2) arc (15:-15:2) ; 
\draw[fill] (3,0) circle (3pt) node[right]{$\theta_+$} ;
\draw (2,0) node[left]{$f$} ;
\filldraw[fill=black!30, draw=black,dotted] (165:2) -- (165:3) arc (165:195:3) -- (195:2) arc (195:165:2) ; 
\draw (0,0) circle (3) ;
\draw[fill] (-3,0) circle (3pt) node[left]{$\theta_-$} ;
\draw (-2,0) node[right]{$f$} ;
\draw (45:3) node[above right] {$S_\infty\V$} ;
\draw (0,2) node {$\V$} ;
\end{tikzpicture}
\ar[r]^\Lap
& 
\begin{tikzpicture}[scale=.5,
	baseline=(O.base)]
\draw (0,0) coordinate (O) ;
\filldraw[fill=black!10] (0,0) circle (3) ;
\filldraw[fill=black!30, draw=black,dotted] (-30:2) -- (-30:3) arc (-30:30:3) -- (30:2) arc (30:-30:2) ; 
\draw[fill] (3,0) circle (3pt) node[right]{$\eta$};
\draw (0,0) circle (3) ;
\draw (-3,0) node[left]{\empty} ;
\draw (2,0) node[left]{$g_\pm$} ;
\draw (0,2) node {$\W$} ;
\draw (45:3) node[above right] {$S_\infty\W$} ;
\end{tikzpicture}
} 
\caption{The Legendre transform.}\label{fig:Puiseux}
\end{figure}
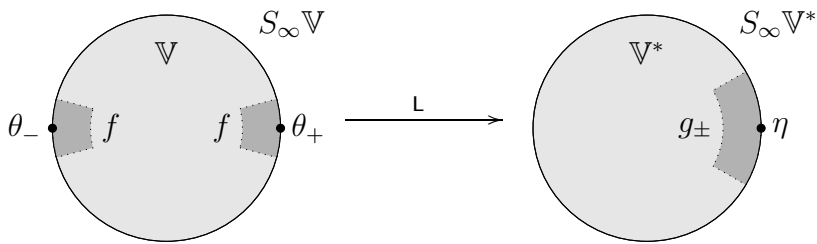

\subsection{}\label{sse:FouM}
Note that the above microlocal construction fails if the Puiseux germ
$f$ is linear, that is, if $f(z)= bz$ for some $b\in\W$. In fact, in this case $C_f=\{w=b\}$, so that $\chi_\rho(C_f)$ is not of graph type.

Let $a\in\V\union\{\infty\}$ and $\theta\in S_a\V$. We say that a Puiseux germ $(a,\theta,f)$ is \emph{admissible} if $f$ is unbounded at $a$ and (if $a=\infty$) not linear modulo bounded functions.
The Legendre transform $\Lap$ establishes a one-to-one correspondence from admissible Puiseux germs on $\V$ to admissible Puiseux germs on $\W$.

We can now state the stationary phase formula.

\begin{introthm}\label{thm:introstatio}
Let $(a,\theta,f)$ be an admissible Puiseux germ on $\V$,
and $\shm$ a holonomic algebraic $\D_\V$-module.
Then, $(a,\theta,f)$ is an exponential factor of $\shm$ if and only if $\Lap(a,\theta,f)$ is an exponential factor of $\lap\shm$. 
Moreover, they have the same multiplicity.	
\end{introthm}

This result is classical (see \S\ref{sse:refs} for some references).
Let us translate it through the Riemann-Hilbert correspondence.

\subsection{}\label{sse:introRH}
In order to keep track of the behavior at $\infty$, instead of $\V$ consider the pair
$\V_\infty = (\V,\PP)$, where $\PP=\V\union\{\infty\}$. Such a pair is called a \emph{bordered space}.
By definition, a holonomic $\D_{\V_\infty}$-module is a holonomic analytic $\D_\PP$-module such that $\shm\simeq\shm(*\infty)$. Then, the category of algebraic holonomic $\D$-modules on $\V$ is equivalent to that of holonomic $\D_{\V_\infty}$-modules.

The Riemann-Hilbert correspondence of \cite{DK16} provides an embedding (i.e., a fully faithful functor)
\[
\xymatrix{
\solE[\V_\infty]\colon\BDC_\hol(\D_{\V_\infty}) \ar@{ >->}[r] & 
\dereb_\Rc(\ind\C_{\V_\infty})
}
\]
from the triangulated category of holonomic $\D_{\V_\infty}$-modules to that of $\R$-constructible enhanced ind-sheaves on $\V_\infty$. Let us briefly recall the construction of the latter category, and the behavior of the ``enhanced solution functor'' $\solE[\V_\infty]$.

\smallskip
\emph{Note that, to simplify the presentation, some of the definitions that we give in \S\ref{sse:introERc} and \S\ref{sse:introEIRc} are different, but equivalent, to those that we recall in \S\ref{sec:not}.}

\subsection{}\label{sse:introERc}
Denote by $\BDC_\Rc(\C_{\V_\infty})$ the triangulated category of $\R$-constructible sheaves on $\V_\infty$, that is, the restrictions to $\V$ of $\R$-constructible sheaves (in the derived sense) on $\PP$. 

Consider the bordered space $\R_\infty=(\R,\cR)$, where $\cR=\R\cup\{-\infty,+\infty\}$, and denote by $t\in\R$ the coordinate.

The category $\dereb_\Rc(\C_{\V_\infty})$ of $\R$-constructible enhanced sheaves on $\V_\infty$ is the full triangulated subcategory of $\BDC_\Rc(\C_{\V_\infty\times\R_\infty})$ whose objects $F$ satisfy 
$\C_{\{t\geq 0\}}\ctens F \isoto F$,
where $\ctens$ denotes the convolution in the $t$ variable (see \eqref{eq:ctens}).
Denote by $\quot(F)=\C_{\{t\geq 0\}}\ctens F$ the projector from $\BDC_\Rc(\C_{\V_\infty\times\R_\infty})$ to $\dereb_\Rc(\C_{\V_\infty})$.

Let $U$ be an open subset of $\V$, subanalytic in $\PP$.
Let $\varphi\colon U\to \R$ be a globally subanalytic function, that is, a function
whose graph is subanalytic in $\PP\times\cR$. Set
\[
\ex^\varphi_{U|\V} \defeq \quot\C_{\{(z,t)\in\V\times\R\semicolon z\in U,\ t+\varphi(z)\geq 0\}}.
\]
It belongs to the heart of the natural $t$-structure on $\dereb_\Rc(\C_{\V_\infty})$. If 
$\varphi^+,\varphi^-$ are globally subanalytic functions on $U$ with $\varphi^-\le\varphi^+$, define $\ex^{\range{\varphi^+}{\varphi^-}}_{U|\V}$ by the short exact sequence
\[
0 \to \ex^{\range{\varphi^+}{\varphi^-}}_{U|\V} \to
\ex^{\varphi^+}_{U|\V} \to
\ex^{\varphi^-}_{U|\V} \to 0.
\]

A structure theorem asserts that for any $F\in\dereb_\Rc(\C_{\V_\infty})$ there is a subanalytic stratification of $\V$ such that $F$ decomposes on each stratum $U$ as a finite direct sum  of shifts of objects of the form $\ex^\varphi_{U|\V}$ or $\ex^{\range{\varphi^+}{\varphi^-}}_{U|\V}$.

\subsection{}\label{sse:introEIRc}
The category $\dereb_\Rc(\ind\C_{\V_\infty})$ of $\R$-constructible enhanced ind-sheaves on $\V_\infty$ is the  triangulated category with the same objects as $\dereb_\Rc(\C_{\V_\infty})$, and morphisms
\[
\Hom[\dereb_\Rc(\ind\C_{\V_\infty})](F_1,F_2) =
\ilim[c\to+\infty]\Hom[\dereb_\Rc(\C_{\V_\infty})](F_1,\C_{\{t\geq c\}}\ctens F_2).
\]
In order to avoid confusion, denote by $\C_{\V_\infty}^\enh\ctens F$ the enhanced ind-sheaf corresponding to the enhanced sheaf $F$.

As for enhanced sheaves, any $K\in\dereb_\Rc(\ind\C_{\V_\infty})$ locally decomposes as a finite direct sum of shifts of objects of the form
\[
\Ex^\varphi_{U|\V_\infty} \defeq \C_{\V_\infty}^\enh\ctens\ex^\varphi_{U|\V}, \quad
\Ex^{\range{\varphi^+}{\varphi^-}}_{U|\V_\infty} \defeq \C_{\V_\infty}^\enh\ctens\ex^{\range{\varphi^+}{\varphi^-}}_{U|\V}.
\]

There is a natural embedding
\[
\xymatrix{e \colon \BDC_\Rc(\C_{\V_\infty}) \ar@{ >->}[r] & \dereb_\Rc(\ind\C_{\V_\infty})}, \quad
L \mapsto \C_{\V_\infty}^\enh\ctens \quot(\C_{\{t= 0\}}\tens\opb\pi L),
\]
where $\pi\colon\V\times\R\to\V$ is the projection.

\subsection{}\label{sse:intronormal}

Without recalling its definition, let us illustrate the behavior of the functor $\solE[\V_\infty]$ by two examples.

The classical Riemann-Hilbert correspondence of \cite{Kas84} associates to
a regular holonomic $\D_{\V_\infty}$-module $\shl$, the $\C$-constructible complex of its holomorphic solutions $L$. One has
\[
\solE[\V_\infty](\shl) \simeq e(L).
\]

Let $f$ be a meromorphic function on $\PP$ with possible poles at $a\in\PP$ and $\infty$.
Denote by $\she^f_{\V\setminus\{a\}|\V_\infty}$ the holonomic $\D_{\V_\infty}$-module associated with the meromorphic connection $d+d f$.
Then, one has
\[
\solE[\V_\infty](\she^f_{\V\setminus\{a\}|\V_\infty}) \simeq \Ex^{\Re f}_{\V\setminus\{a\}|\V_\infty}.
\]

\subsection{}\label{sse:enhHLT}
Let us translate the Hukuhara-Levelt-Turrittin theorem in terms of enhanced ind-sheaves.

\smallskip
Let $a\in\PP$. A multiplicity at $a$ is a morphism of sheaves of sets 
\[
N\colon \shp_{S_a\V } \to (\Z_{\geq0})_{S_a\V }
\]
such that $N^{>0}_\theta \defeq N_\theta^{-1}(\Z_{>0})$ is finite for any $\theta\in S_a\V $, and $f,g\in N^{>0}_\theta$ with $f\neq g$ implies $[f]\neq[g]$ in $\oshp_{S_a\V }$ (see \S\ref{sse:Normenh} for the precise definition).
Consider the induced multiplicity class
\[
\overline N\colon \oshp_{S_a\V } \to (\Z_{\geq0})_{S_a\V }
\]
defined by $\overline N([f]) = N(g)$ if $[f]=[g]$ for some $g\in N^{>0}_\theta$, and $\overline N([f]) =0$ otherwise.

Let us say that $F\in \dereb_\Rc(\C_{\V_\infty})$ has normal form at $a$
if there exists a multiplicity $N$ such that any $\theta\in S_a \V $ has an open sectorial neighborhood $V_\theta$ where
\begin{equation}
\label{eq:introCVF}
\opb\pi\C_{V_\theta}\tens F \simeq 
\DSum_{f\in N_\theta^{>0}} \bl\ex^{\Re f}_{V_\theta|\V_\infty}\br^{N(f)}.
\end{equation}
Note that the multiplicity $N$ is uniquely determined by $F$.

Let us say that $K\in \dereb_\Rc(\ind\C_{\V_\infty})$ has normal form at $a$
if there exists $F\in \dereb_\Rc(\C_{\V_\infty})$ with normal form at $a$
such that
\[
K \simeq \C^\enh_{\V_\infty} \ctens F.
\]
(This definition is different, but equivalent, to the one we give in \S\,\ref{sse:normenhind}.)
Note that, if $N$ is the multiplicity of $F$, the multiplicity class $\overline N$ is uniquely determined by $K$.

Let $\shm$ be a holonomic $\D_{\V_\infty}$-module.
In terms of enhanced solutions,
the Hukuhara-Levelt-Turrittin theorem states that  $K=\solE[\V_\infty](\shm)$ has normal form at any singular point of $\shm$.

More precisely, as it was observed in \cite{Moc16}, there is an equivalence of categories between germs of meromorphic connections $\shm$ with poles at $a$, and germs of enhanced ind-sheaves $K$ with normal form at $a$, such that $\opb\pi\C_{\V\setminus\{a\}}\tens K \simeq K$.

\subsection{}
As a technical tool, to any admissible Puiseux germ $(a,\theta,f)$ on $\V$ we associate the \emph{multiplicity test functor} (see \S\ref{sse:ilim})
\[
\G_{(a,\theta,f)} \colon \dereb_\Rc(\ind\C_{\V_\infty}) \to \BDC(\C),
\]
with values in the bounded derived category of $\C$-vector spaces.

If $V_\theta$ is a sectorial neighborhood of $\theta$ and $K\in \dereb_\Rc(\ind\C_{\V_\infty})$, one has
\begin{equation}
\label{eq:introGEV}
\G_{(a,\theta,f)}(K) \simeq \G_{(a,\theta,f)}(\opb\pi\C_{V_\theta}\tens K).
\end{equation}
Moreover, if $(a,\theta,h)$ is another Puiseux germ,
\begin{equation}
\label{eq:introGEf}
\G_{(a,\theta,f)}(\Ex^{\Re h}_{V_\theta|\V_\infty}) \simeq
\begin{cases}
\C &\text{if }[f]=[h], \\
0 &\text{otherwise.}
\end{cases}
\end{equation}
In particular, if $K$ has normal form at $a$ with multiplicity class $\overline N$, then
\[
\G_{(a,\theta,f)}(K) \simeq
\C^{\overline N([f])}.
\]

\subsection{}\label{sse:FouK}
Consider the correspondence
\[
\xymatrix{\V_\infty & \V_\infty\times \W_\infty \ar[l]_-p \ar[r]^-q & \W_\infty,}
\]
where $\W_\infty = (\W,\bb)$ for $\bb=\W\cup\{\infty\}$.

The Fourier-Laplace transform for $\D$-modules and the Fourier-Sato transform for enhanced ind-sheaves are the integral transforms with kernel associated to $e^{-zw}$. More precisely,
for $\shm\in\BDC_\hol(\D_{\V_\infty})$ and $K\in\dereb_\Rc(\ind\C_{\V_\infty})$, one sets
\begin{align*}
\lap \shm &= \doim q(\she^{-zw}_{\V\times\W|\V_\infty\times\W_\infty}\dtens\dopb p \shm), \\
\lap K &= \Eeeim q(\ex^{-\Re zw}_{\V\times\W}[1]\ctens\Eopb p K),
\end{align*}
where $\dtens$, $\doim q$, $\dopb p$ and $\ctens$, $\Eeeim q$, $\Eopb p$ denote the operations for $\D$-modules and for enhanced ind-sheaves, respectively.

Since the enhanced solution functor is compatible with these operations,  one has
\[
\lap\solE[\V_\infty](\shm) \simeq \solE[\W_\infty](\lap\shm).
\]
We can thus deduce the stationary phase formula of Theorem~\ref{thm:introstatio} from the following analogue for enhanced ind-sheaves.

\begin{introthm}\label{thm:introstatio2}
Let $(a,\theta,f)$ be an admissible Puiseux germ on $\V$, 
and set $(b,\eta,g) =\Lap(a,\theta,f)$.
Let $K\in\dereb_\Rc(\ind\C_\V )$ have normal form at $a$
with multiplicity class $\overline N$.
Then, for generic $\eta$, one has
\[
\G_{(b,\eta,g)}(\lap K) \simeq \C^{\overline N([f])}.
\]
\end{introthm}

In fact, the statement still holds when replacing $\dereb_\Rc(\ind\C_\V )$ with $\dereb_\Rc(\ifield_\V )$, for an arbitrary base field $\field$.

As we now briefly explain, our proof of this result is based on microlocal arguments.

\subsection{}
The microlocal theory of sheaves of \cite{KS90} associates to an object $H$ of $\BDC(\C_{\V\times\R})$ its microsupport $SS(H)\subset T^*(\V\times\R)$,
a closed conic involutive subset of the cotangent bundle. 
Denote by $(t,t^*)\in T^*\R$ the natural homogeneous symplectic coordinates.
Recall that there is an equivalence
\[
\dereb_\Rc(\C_\V) \simeq \BDC_\Rc(\C_{\V_\infty\times\R_\infty})/\{F\semicolon SS(F)\subset\{t^*\leq 0\}\}.
\]
In particular, considering the map
\[
\xymatrix@R=0ex{
T^*(\V\times\R)\supset\{t^*>0\} \ar[r]^-{\rho'} & T^*\V,
}\quad
(z,t;z^*,t^*) \mapsto (z,z^*/t^*),
\]
there is a well defined microsupport for enhanced sheaves given by
\[
\SSEo(\quot (H)) \defeq \overline{\rho'(SS(H)\cap\{t^*>0\})} \subset T^*\V,
\]
for $H\in\BDC_\Rc(\C_{\V_\infty\times\R_\infty})$.

To be more precise, the microsupport is a subset of the cotangent bundle to the real affine plane $\V^\R$ underlying $\V$. Here, for $\varphi$ a real valued smooth function, we use the identification $(T^*\V)^\R\isoto T^*(\V^\R)$, $\partial\varphi\mapsto \frac12 d\varphi$, where $d=\partial+\overline\partial$ is the exterior differential. Then, we have $\partial f\mapsto d(\Re f)$ if $f$ is holomorphic.

\subsection{}
Using the same definitions as for enhanced ind-sheaves, 
there is a Fourier-Sato transform for enhanced sheaves which, for $F\in\dereb_\Rc(\C_{\V_\infty})$, satisfies
\[
\C_{\W_\infty}^\enh\ctens \lap F \simeq \lap(\C_{\V_\infty}^\enh\ctens F) .
\]
It was proved in \cite{Tam08} that
\[
\SSEo(\lap F) = \chi_\rho(\SSEo(F)).
\]

Note that, if $\varphi$ is a real valued smooth function defined on an open subset $U\subset \V$, one has
\[
\SSEo(\ex_{U|\V}^\varphi)\cap T^*U = C_\varphi \defeq (\text{graph of }d\varphi).
\]

In particular, for $(a,\theta,f)$ a Puiseux germ on $\V$, one has
\[
\SSEo(\ex_{V_\theta|\V}^{\Re f})\cap T^*V_\theta = C_{(a,\theta,f)},
\]
for $V_\theta$ a sectorial neighborhood of $\theta$ in $\V$.
Thus, one has the following  link between the Fourier-Sato transform and the Legendre transform:
\[
\SSEo(\lap\ex_{V_\theta|\V}^{\Re f})\cap T^*W_\eta = C_{\Lap(a,\theta,f)},
\]
for $W_\eta$ a sectorial neighborhood of $\eta$ in $\W$.

\subsection{}
Our proof of Theorem~\ref{thm:introstatio2} proceeds by the following arguments.
Assume for example that $a=\infty$ and $b=\infty$ as in \S\ref{sse:Airy}.
(The other cases are treated in a similar way.)

Since $K$ has normal form at $\infty$, we can write $K=\C^\enh_{\V_\infty}\ctens F'$ for $F'\in\dereb_\Rc(\C_{\V_\infty})$ with normal form at $\infty$.
Then, since $\C_{\W_\infty}^\enh\ctens\lap F' \simeq \lap K$, it is enough to show that
\[
\G_{(\infty,\eta,g)}(\lap F') \simeq \C^{\overline N(f)}, 
\]
where the above multiplicity test functor for enhanced sheaves has the same
definition as the one for enhanced ind-sheaves.

For $R>0$, set $F = \opb\pi\C_{\{|z|>R\}}\tens F'$. A microlocal argument similar to the one sketched below shows that
\[
\G_{(\infty,\eta,g)}(\lap F') \simeq \G_{(\infty,\eta,g)}(\lap F).
\]

We can take $R$ big enough so that $\{|z|>R\}$ is covered by sectors $V_\theta$ where $F$ decomposes as in \eqref{eq:introCVF}. 

Moreover, since $\lap F$ is $\R$-constructible, a generic $\eta\in S_\infty\W$ has a sectorial neighborhood $W_\eta$ where there is a decomposition
\begin{equation}
\label{eq:introDec}
\opb\pi\C_{W_\eta} \tens \lap F \simeq
\bl \DSum_{i\in I} \ex_{W_\eta|\W}^{\varphi_i}[d_i] \br
\oplus
\bl \DSum_{j\in J} \ex_{W_\eta|\W}^{\range{\varphi_j^+}{\varphi_j^-}}[d_j] \br,
\end{equation}
with $d_i,d_j\in\Z$, and
$\varphi_i,\varphi^\pm_j$ analytic and globally subanalytic functions such that
$\varphi^-_j < \varphi^+_j$.

By \eqref{eq:introGEV}, we are left to show
\[
\G_{(\infty,\eta,g)}(\opb\pi\C_{W_\eta}\tens \lap F) \simeq \C^{\overline N(f)}.
\]

On one hand, \eqref{eq:introCVF} implies
\[
\SSEo( \lap F) = \chi_\rho( \SSEo(F) ) \subset Z\union \Union_{\tilde\theta\in S_\infty\V,\ \tilde f\in N^{>0}_{\tilde \theta}}C_{\Lap(\infty,\tilde\theta,\tilde f)},
\]
where $Z$ is due to spurious contributions from $\{|z|=R\}$.

On the other hand, \eqref{eq:introDec} gives
\[
\SSEo( \lap F) \cap T^*W_\eta = \bl \Union_{i\in I} C_{\varphi_i} \br
\union
\bl \Union_{j\in J} (C_{\varphi_j^+}\union C_{\varphi_j^-}) \br.
\]

Let $\phi$ be either $\varphi_i$ or $\varphi_j^\pm$. We can show that 
\[
\G_{(\infty,\eta,g)}(\ex_{W_\eta|\W}^\phi)\simeq 0
\]
if $C_\phi \subset Z$. 
By \eqref{eq:introGEf}, the same holds for $C_\phi=C_{\Lap(\infty,\tilde\theta,\tilde f)}$
unless $\Lap(\infty,\tilde\theta,\tilde f) = (\infty,\eta,g)$.
It follows that
\[
\G_{(\infty,\eta,g)}(\opb\pi\C_{W_\eta}\tens \lap F) \simeq \DSum_{k\in I_g\cup J_g^+ \cup J_g^-}\C[d'_k],
\]
where $I_g = \{ i\in I \semicolon \varphi_i=\Re g\}$, $J^\pm_g = \{ j\in J \semicolon \varphi^\pm_j=\Re g\}$, and
\[
d'_k=
\begin{cases}
d_k &\text{if }k\in I_g\cup J^+_g, \\
d_k-1 &\text{if }k\in J^-_g.
\end{cases}
\]

It remains to show that $d'_k=0$ for any $k$, and $\overline N(f) = \# (I_g\cup J^+_g \cup J_g^-)$. For this, we use a result of \cite{KS90} which allows to keep track microlocally of multiplicities and shifts, by viewing the Fourier-Sato transform as a quantization of the symplectic transformation $\chi_\rho$.

\subsection{}\label{sse:refs}
Let us mention some related literature.

\smallskip
The Fourier-Laplace transform for holonomic $\D$-modules in dimension one has 
been studied in \cite{Mal91}, and more  systematically in \cite{Moc10}, where the
Stokes phenomenon is also considered. See \cite{Sab16,HS15} 
for explicit computations in some special cases.

Classically, the stationary phase formula is stated in terms of the so-called
local Fourier-Laplace transform for formal holonomic $\D$-modules.
This was introduced in \cite{BE04} (see also \cite{Gar04,Ari10}), 
by analogy with the $\ell$-adic case treated in~\cite{Lau87}. 
(For related results in the $p$-adic case, see e.g.\ \cite{Ram12} for analytic
\'etale sheaves and \cite{AM15} for arithmetic $\D$-modules.)

An explicit stationary phase formula was obtained in \cite{Sab08,Fan11} 
(see also \cite{Gra13}) for $\D$-modules, and in \cite{Fu10,AS10} for $\ell$-adic sheaves.

The fact that the Riemann-Hilbert correspondence of \cite{DK16} intertwines the
Fourier-Laplace transform of holonomic $\D$-modules with the enhanced Fourier-Sato
transform was observed in \cite{KS16L}, where the non-holonomic case is also discussed.
In dimension one, the enhanced Fourier-Sato transform of perverse sheaves has been studied in \cite{DHMS17}, where the Stokes phenomenon is also considered.

In the present paper we do not discuss linear exponential factors.
Note that a point $a\in\V$ gives a linear function $w\mapsto aw$ on $\W$. 
Let $\shm$ be a holonomic $\D_{\V_\infty}$-module.
In \cite{DK17} we relate, in the framework of enhanced ind-sheaves,
the vanishing cycles of $\shm$ at $a$ with the graded component 
$\Gr_{aw}\Psi_\infty(\lap \shm)$ of its Fourier transform.

\subsection{}
The contents of this paper is as follows.

\smallskip
After recalling some notations in Section~\ref{sec:not}, we study in Section~\ref{sec:exp} the objects of the form $\ex^\varphi_{U|\V}$ or $\ex^{\range{\varphi^+}{\varphi^-}}_{U|\V}$ (resp. $\Ex^\varphi_{U|\V_\infty}$ or $\Ex^{\range{\varphi^+}{\varphi^-}}_{U|\V_\infty}$), which are the building blocks of $\R$-constructible enhanced sheaves (resp.\ ind-sheaves).

We study their behavior on sectorial neighborhoods in Section~\ref{sec:blow}.
This is used in Section~\ref{sec:normal} to discuss the notion of enhanced (ind-)sheaves with normal form at a point. In Section~\ref{sec:FStokes}, to an enhanced ind-sheaf with normal form we attach a filtered Stokes local system. In particular, we get a notion of exponential factor. We also introduce the multiplicity test functor, which detects the multiplicities of exponential factors.

In Section~\ref{sec:statement} we recall the Legendre transform for Puiseux germs, highlighting its microlocal nature. We can then state the stationary phase formula in terms of enhanced ind-sheaves. Our proof of this formula uses techniques of the  microlocal study of sheaves, which are detailed in Section~\ref{sec:micro}.

\subsection*{Acknowledgments} 
The first author acknowledges the kind hospitality at RIMS, 
Kyoto University, during the preparation of this paper.

\addtocontents{toc}{\protect\setcounter{tocdepth}{2}}
\numberwithin{equation}{subsection}

\section{Notations and complements}\label{sec:not}

In this paper, $\field$ denotes a base field.

We say that a topological space is good if it is Hausdorff, locally compact, countable at
infinity, and has finite soft dimension.

\medskip

We recall here some notions and results, mainly to fix notations, referring to the literature for details. In particular, we refer to \cite{KS90} for sheaves and their microsupport, to \cite{Tam08} (see also \cite{GS14,DK16bis}) for enhanced sheaves and their microsupport, to \cite{KS01} for ind-sheaves, to \cite{DK16} for bordered spaces and enhanced ind-sheaves (see also \cite{KS16L,KS16D,Kas16,DK16bis}), and to \cite{Kas03} for $\D$-modules.

\subsection{A remark on inductive limits}\label{sse:ilim}
Denote by $\Mod(\field)$ the Grothendieck category of $\field$-vector spaces,
by $\TDC(\field)$ its derived category, and by $\BDC(\field)\subset\TDC(\field)$ its bounded derived category, whose objects $K$ satisfy $H^nK=0$ for $|n|\gg 0$. 

Since $\Mod(\field)$ is semisimple, there is an equivalence of additive categories
\begin{equation}
\label{eq:Dk}
\TDC(\field)\isoto\Mod(\field)^\Z,\quad
K\mapsto \bl H^nK[-n]\br_{n\in\Z}.
\end{equation}
Note that a triangle $X\to[f]Y\to[g]Z\to[h]X[1]$ in $\TDC(\field)$ is distinguished if and only if $H^nX\to[H^nf]H^nY\to[H^ng]H^nZ\to[H^nh]H^{n+1}X\to[H^{n+1}f]H^{n+1}Y$ is exact in $\Mod(\field)$ for any $n\in\Z$.

It follows from \eqref{eq:Dk} that small filtrant inductive limits exist in $\TDC(\field)$.

Let us say that a small filtrant inductive system $u\colon I\to\TDC(\field)$ is uniformly bounded if there exist integers $c\leq d$ such that $H^nu(i)=0$ for any $i\in I$ and $n\notin[c,d]$.

Clearly, if $u\colon I\to\TDC(\field)$ is uniformly bounded, then
$\ilim u \in \BDC(\field)$.

\subsection{Enhanced sheaves}\label{sse:enhs}
Let $M$ be a good topological space. Denote by $\BDC(\field_M)$ the bounded derived category of sheaves of $\field$-vector spaces on $M$, and by $\tens$, $\reim f$, $\opb f$, $\rhom$, $\roim f$, $\epb f$ the six operations (here $f\colon M\to N$ is a continuous map). Write $\RHom(L_1,L_2)=\rsect(M;\rhom(L_1,L_2))$.

Denote by $t$ the coordinate in $\R$, and consider the maps
\[
\xymatrix@C=2em{
M\times\R^2 \ar[rr]^{p_1,p_2,\mu} && M\times \R \ar[r]^-\pi & M,
}
\]
where $p_1$, $p_2$, $\pi$ are the projections, and $\mu(x,t_1,t_2)=(x,t_1+t_2)$.
The convolution functors with respect to the $t$ variable in $\BDC(\field_{M\times\R})$ are defined by 
\begin{align}
\label{eq:ctens}
F_1\ctens F_2 &= \reim\mu(\opb p_1 F_1 \tens \opb p_2 F_2), \\ \notag
\chom(F_1,F_2) &= \roim{{p_1}}\rhom(\opb{p_2} F_1, \epb\mu F_2).
\end{align}

Note that the object $\field_{\{t\geq 0\}}$ is idempotent for $\ctens$.
The triangulated category of enhanced sheaves is defined by
\[
\dereb_+(\field_M) = \BDC(\field_{M\times\R})/\{F\semicolon \field_{\{t\geq 0\}}\ctens F \simeq 0 \}.
\]
The quotient functor
\begin{equation}
\label{eq:QE}
\quot\colon \BDC(\field_{M\times\R})\to\dereb_+(\field_M)
\end{equation}
has fully faithful left and right adjoints, respectively given by
\[
\LE(\quot F) = \field_{\{t\geq 0\}}\ctens F, \quad
\RE(\quot F) = \chom(\field_{\{t\geq 0\}}, F).
\]

Enhanced sheaves are endowed with the six operations $\ctens$, $\Eeim f$, $\Eopb f$, $\chom$, $\Eoim f$, $\Eepb f$.
Here, $\ctens$ and $\chom$ descend from $\BDC(\field_{M\times\R})$, and the exterior operations are defined by
\[
\Eeim f (\quot F) = \quot (\derr f_{\R!} F),\quad \Eopb f (\quot G) = \quot (\opb{f_\R} G),\quad \dots,
\]
where $f_\R=f\times\id_\R$.

The natural $t$-structure of $\BDC(\field_{M\times\R})$ induces by $\LE$ a $t$-structure for enhanced sheaves, and we consider its heart
\[
\dere^0_+(\field_M) = \{F\in\dereb_+(\field_M) \semicolon H^j\LE(F)=0 \text{ for any }j\neq 0\}.
\]
Then $U\mapsto\dere^0_+(\field_U)$ is a stack on $M$.

\subsection{Ind-sheaves on bordered spaces}\label{sse:bddspaces}
Let $M$ be a good topological space. Denote by $\BDC(\ind\field_M)$ the bounded derived category of ind-sheaves of $\field$-vector spaces on $M$, that is, ind-objects with values in the category of sheaves with compact support. 
There is a natural exact embedding $\xymatrix{\BDC(\field_M) \ar@{ >->}[r] & \BDC(\ind\field_M)}$,
which has an exact left adjoint $\alpha$, that has in turn an exact fully faithful left adjoint $\beta$.

Denote by $\tens$, $\reeim f$, $\opb f$, $\rihom$, $\roim f$, $\epb f$ the six operations for ind-sheaves. Ind-sheaves on $M$ form a stack, whose sheaf-valued hom functor has $\rhom=\alpha\rihom$ as its right derived functor.

A \emph{bordered space} is a pair $M_\infty = (M, \cM)$ of a good topological space $\cM$ and an open subset $M\subset\cM$. A morphism $f\colon M_\infty\to (N, \cN)$ is a continuous map $f\colon M\to N$ such that the first projection $\cM\times\cN\to\cM$ is proper on the closure $\overline\Gamma_f$ of the graph $\Gamma_f$ of $f$. 
If also the second projection $\overline\Gamma_f\to\cN$ is proper, $f$ is called \emph{semiproper}.
The category of topological spaces embeds into that of bordered spaces by the identification $M=(M,M)$.

The triangulated category of ind-sheaves on $M_\infty$ is defined by
\[
\BDC(\ind\field_{M_\infty}) = \BDC(\ind\field_{\cM})/\BDC(\ind\field_{\cM\setminus M}).
\]
The quotient functor
\[
\mathsf q\colon  \BDC(\ind\field_{\cM}) \to \BDC(\ind\field_{M_\infty})
\]
has a left adjoint $\mathsf l$ and a right adjoints $\mathsf r$, both fully faithful, given by
\[
\mathsf l(\mathsf q F) = \field_M\tens F, \quad
\mathsf r(\mathsf q F) = \rihom(\field_M, F).
\]

For $f\colon M_\infty\to (N, \cN)$ a morphism of bordered spaces, the six operations for ind-sheaves on bordered spaces are defined by
\begin{align*}
&\mathsf q F_1 \tens \mathsf q F_2 = \mathsf q (F_1 \tens F_2), &
&\rihom(\mathsf q F_1, \mathsf q F_2) = \mathsf q \rihom(F_1, F_2), \\
&\reeim f (\mathsf q F) = \mathsf q \reeim q (\field_{\Gamma_f} \tens \opb p F), & 
&\roim f (\mathsf q F) = \mathsf q \roim q \rihom(\field_{\Gamma_f}, \epb p F), \\
&\opb f (\mathsf q G) = \mathsf q \reeim p (\field_{\Gamma_f} \tens \opb q G), & 
&\epb f (\mathsf q G) = \mathsf q \roim p \rihom(\field_{\Gamma_f}, \epb q G).
\end{align*}
where $p\colon \cM\times\cN\to\cM$ and $q\colon \cM\times\cN\to\cN$ are the projections.

There is a natural exact embedding $\xymatrix{\BDC(\field_M) \ar@{ >->}[r] & \BDC(\ind\field_{M_\infty})}$, which has an exact left adjoint $\alpha$.

\subsection{Enhanced ind-sheaves}
Let $M_\infty=(M,\cM)$ be a bordered space. 
Set $\R_\infty= (\R,\cR)$, with $\cR=\R\union\{-\infty,+\infty\}$, and recall that $t\in\R$ denotes the coordinate.
The triangulated category of enhanced ind-sheaves on $M_\infty$ is defined by
\[
\dereb_+(\ind\field_{M_\infty}) = \BDC(\ind\field_{M_\infty\times\R_\infty})/\{F\semicolon \field_{\{t\geq 0\}}\ctens F \simeq 0 \},
\]
where the convolution functor $\ctens$ is defined as in \S\ref{sse:enhs}, replacing $\R$ with $\R_\infty$ and $\reim\mu$ with $\reeim\mu$.
The quotient functor
\begin{equation}
\label{eq:QEI}
\quot\colon \BDC(\ind\field_{M_\infty\times\R_\infty}) \to \dereb_+(\ind\field_{M_\infty})
\end{equation}
has a fully faithful left and right adjoint $\LE$ and $\RE$, respectively, defined as in \S\ref{sse:enhs}.

The six operations for enhanced ind-sheaves are denoted\footnote{In \cite{DK16}, $\cihom$ is denoted by ${\shi hom}^+$.} by $\ctens$, $\Eeeim f$, $\Eopb f$, $\cihom$, $\Eoim f$, $\Eepb f$ (here $f\colon M_\infty\to N_\infty$ is a morphism of bordered spaces). As in \S\ref{sse:enhs}, the exterior operations are defined via the associated morphism $f_\R=f\times\id_{\R_\infty}$.
Denote by $\cetens$ the external tensor product.
There are outer hom functors\footnote{In \cite{DK16}, $\fhom$ is denoted by $\hom^\enh$.}
\begin{align*}
\fihom(K_1,K_2) &\defeq \roim\pi\rihom(\LE K_1,\RE K_2) \\
&\simeq \roim\pi\rihom(\LE K_1,\LE K_2), \\
\fhom(K_1,K_2) &\defeq \alpha\fihom(K_1,K_2),
\end{align*}
with values in $\BDC(\ind\field_{M_\infty})$ and $\BDC(\field_M)$, respectively.  We write 
\[
\FHom(K_1,K_2)=\rsect(M;\fhom(K_1,K_2))\ \in\BDC(\field).
\]

The triangulated category $\dereb_+(\ind\field_{M_\infty})$ has a natural $t$-structure, and we denote by $\dere^0_+(\ind\field_{M_\infty})$ its heart. Then $U\mapsto\dere^0_+(\ind\field_U)$ is a stack on $M$.

There are natural embeddings  $\xymatrix{\dereb_+(\field_M) \ar@{ >->}[r] & \dereb_+(\ind\field_{M_\infty})}$ and
\begin{align*}
\xymatrix{\epsilon\colon\BDC(\ind\field_{M_\infty}) \ar@{ >->}[r] & \dereb_+(\ind\field_{M_\infty})}, \quad L\mapsto &\quot(\field_{\{t= 0\}} \tens \opb\pi L) \\
&\simeq \quot(\field_{\{t\geq 0\}} \tens \opb\pi L).
\end{align*}
Note that $\epsilon(L)\in\dereb_+(\field_M)$ if $L\in\BDC(\field_M)$.

For $L\in\BDC(\ind\field_{M_\infty})$ and $K\in\dereb_+(\ind\field_{M_\infty})$, the operations
\[
\opb\pi L \tens K , \quad
\rihom(\opb\pi L,K) 
\]
are well defined.  In fact, one has
\begin{align*}
\opb\pi L \tens K &\simeq \epsilon(L) \ctens K , \\
\rihom(\opb\pi L,K) &\simeq \cihom(\epsilon(L),K).
\end{align*}

For $K\in\dereb_+(\ind\field_{M_\infty})$ one has
$\quot\field_{\{t\geq 0\}} \ctens K \simeq K$, and more generally for $c\in\R$
\[
\quot\field_{\{t\geq c\}} \ctens K \simeq \Eoim{{\mu_c}}K,
\]
where $\mu_c(x,t)=(x,t+c)$ is the translation.

The following lemma will be of use later.

\begin{lemma}\label{lem:stalk}
For $i\in\Z_{\ge0}$, let $U_i\subset M$ be a family of open subsets such that $U_{i+1}\subset\subset U_i$. Let $K_1,K_2\in\dereb_+(\ifield_{M_\infty})$. Then, one has an isomorphism in $\BDC(\field)$
\[
\ilim[i] \rsect(U_i;\fhom(K_1,K_2)) \simeq
\ilim[i] \FHom(\opb\pi\field_{U_i}\tens K_1,\opb\pi\field_{U_i}\tens K_2).
\]
\end{lemma}

\begin{proof}
One has
\begin{align}\label{eq:homenh}
\rsect(U_i;\fhom(K_1,K_2)) 
&\simeq \RHom(\field_{U_i},\alpha\roim\pi\rihom(\LE K_1,\RE K_2)) \\
&\simeq \RHom(\opb\pi\beta\field_{U_i},\rihom(\LE K_1,\RE K_2)). \notag
\end{align}
Recall that 
\[
\beta\field_{U_i}\simeq\indlim[V\subset\subset U_i]\field_V.
\]
Since there is a natural commutative diagram
\[
\xymatrix{
\beta\field_{U_{i+1}} \ar[r] \ar@/^2ex/[rr] & \field_{U_{i+1}} \ar[r] \ar@/_2ex/[rr]
& \beta\field_{U_i} \ar[r] & \field_{U_i} ,
}
\]
we can replace $\beta\field_{U_i}$ with $\field_{U_i}$ in the limit, so that
\[
\ilim[i] \rsect(U_i;\fhom(K_1,K_2)) \simeq
\ilim[i] \RHom(\opb\pi\field_{U_i},\rihom(\LE K_1,\RE K_2)).
\]
Finally, one has
\begin{align*}
\RHom(\opb\pi\field_{U_i},{}&\rihom(\LE K_1,\RE K_2)) \\
&\simeq \RHom(\opb\pi\field_{U_i}\tens \LE K_1,\RE K_2) \\
&\simeq \RHom(\LE(\opb\pi\field_{U_i}\tens K_1),\RE K_2) \\
&\simeq \FHom(\opb\pi\field_{U_i}\tens K_1, K_2) \\
&\simeq \FHom(\opb\pi\field_{U_i}\tens K_1, \opb\pi\field_{U_i}\tens K_2).
\end{align*}
\end{proof}

\subsection{Constructible enhanced ind-sheaves}\label{sse:EIRc}
Let $M$ be a subanalytic space. Denote by $\BDC_\Rc(\field_M)\subset\BDC(\field_M)$ the full triangulated subcategory whose objects have $\R$-constructible cohomologies.

A subanalytic bordered space is a pair $M_\infty = (M, \cM)$ of a subanalytic space $\cM$ and an open subanalytic subset $M\subset\cM$. A morphism $f\colon M_\infty\to (N, \cN)$ of subanalytic bordered spaces is a morphism of bordered spaces whose graph is subanalytic in $\cM\times\cN$.

Let $\BDC_\Rc(\field_{M_\infty})\subset\BDC(\field_M)$ be the full triangulated subcategory of $\R$-constructible sheaves on $M_\infty$, that is, objects which are restriction to $M$ of $\R$-constructible objects on $\cM$. 

The triangulated category $\dereb_\Rc(\field_{M_\infty})\subset\dereb_+(\field_M)$ of $\R$-constructible enhanced sheaves on $M_\infty$ is  the essential image of $\BDC_\Rc(\field_{M_\infty\times\R_\infty})$ by the quotient map \eqref{eq:QE}. The family of $\R$-constructible enhanced sheaves is stable by the six operations, assuming semiproperness for direct images.

One sets
\begin{align*}
\field^\enh_M &\defeq \quot\bl\indlim[c\to+\infty] \field_{\{t\geq c\}}\br\ \in\dereb(\ifield_M), \\
\field^\enh_{M_\infty} &\defeq \Eopb j\bl\field^\enh_\cM\br\ \in\dereb(\ifield_{M_\infty}),
\end{align*}
where $\quot$ is the quotient map \eqref{eq:QEI}, and $j\colon M_\infty\to\cM$ is the natural morphism.

The triangulated category of $\R$-constructible enhanced ind-sheaves is the full subcategory  $\dereb_\Rc(\ind\field_{M_\infty})\subset\dereb_+(\ind\field_{M_\infty})$ whose objects $K$ satisfy the following condition. There exists $F\in\BDC_\Rc(\field_{M_\infty\times\R_\infty})$ such that
\[
K \simeq \field^\enh_{M_\infty}\ctens \quot F.
\]
(The equivalence of this description with that in \S\ref{sse:introEIRc} follows from \cite[Proposition 4.7.9]{DK16}.)

The family of $\R$-constructible enhanced ind-sheaves is stable by the six operations, assuming semiproperness for direct images.

There is a natural embedding
\[
\xymatrix{e\colon\BDC_\Rc(\field_{M_\infty}) \ar@{ >->}[r] & \dereb_\Rc(\ifield_{M_\infty})}, \quad L\mapsto \opb\pi L \tens\field^\enh_{M_\infty} \simeq \field^\enh_{M_\infty} \ctens \epsilon(L).
\]

Note that the canonical functor
\[
\dereb_\Rc(\field_{M_\infty})\to\dereb_\Rc(\ifield_{M_\infty}), \quad
F \mapsto \field_{M_\infty}^\enh\ctens F
\] 
is essentially surjective but not fully faithful (see \S\ref{sse:introEIRc}).

\subsection{Microsupport}\label{sse:SS}

Let $M$ be a real analytic manifold.
To $L\in\BDC(\field_M)$ one associates its microsupport $SS(L)\subset T^*M$, a closed conic involutive subset of the cotangent bundle. 

Denote by $(t,t^*)\in T^*\R$ the homogeneous symplectic coordinates.
There is an equivalence
\[
\dereb_+(\field_M) \simeq \BDC(\field_{M\times\R})/\{F\semicolon SS(F)\subset\{t^*\leq 0\}\}.
\]

Denote by $\omega_M$ the canonical 1-form on $T^*M$, that is $\omega_M=\sum x_i^*dx_i$ in local homogeneous symplectic coordinates $(x,x^*)\in T^*M$.
Then, the space $(T^*M)\times\R$ has a contact structure given by $dt+\omega_M$.

Consider the maps
\[
\xymatrix@R=0ex{
&\llap{$T^*(M\times\R)\supset{}$} \{t^*>0\} \ar[r]^-{\gamma} & (T^*M)\times\R \ar[r]^-\rho &
T^*M, \\ 
&(x,t;x^*,t^*) \ar@{|->}[r] & ((x;x^*/t^*),t),
}
\]
where $\rho$ is the projection.
For $F\in\dereb_+(\field_M)$, one sets
\begin{align*}
\SSE(F) &\defeq \gamma \bl SS(F')\cap\{t^*>0\} \br \subset (T^*M)\times\R, \\
\SSEo(F) &\defeq \overline{\rho(\SSE(F))} \subset T^*M,
\end{align*}
where $F'\in\derb(\cor_{M\times\R})$ satisfies $\quot(F')\simeq F$.
The definition of $\SSE(F)$ does not depend on the choice of $F'$.

Note that $\SSEo(\epsilon(L)) = SS(L)$ for $L\in\BDC(\field_M)$.

\medskip 
For a complex manifold $X$, denote by $X^\R$ the underlying real analytic manifold.
For $F\in\dereb_+(\field_X)$, its microsupport $\SSEo(F)$ is a subset of the cotangent bundle $T^*(X^\R)$. In this paper, we use the identification\footnote{This choice differs from the usual one (see e.g.\ \cite[\S11.1]{KS90}), where the identification is given by $\theta\mapsto\theta+\overline\theta$.}
\[
(T^*X)^\R\isoto T^*(X^\R), \quad
\theta\mapsto\Re\theta=\tfrac12(\theta+\overline\theta).
\] 
Hence, denoting by $d=\partial+\overline\partial$ the exterior differential, one has $\partial\varphi\mapsto\frac12 d\varphi$
for $\varphi$ a real valued smooth function. Then, $d f\mapsto d(\Re f)$ if $f$ is holomorphic.

\subsection{$\D$-modules}

Let $X$ be a complex manifold, and denote by $\O_X$ and $\D_X$ the sheaves of holomorphic functions and of differential operators, respectively. 
Let $\BDC(\D_X)$ be the bounded derived category of left $\D_X$-modules.
For $f \colon X \to Y$ a morphism of complex manifolds, denote by $\dtens$, $\doim f$, $\dopb f$ the operations for $\D$-modules.

Denote by $\BDC_\coh(\D_X)\subset\BDC(\D_X)$ the full triangulated
subcategory of  objects with coherent cohomologies.
To $\shm\in\BDC_\coh(\D_X)$ one associates its characteristic variety $\chv(\shm)\subset T^*X$, a closed conic involutive subset of the cotangent bundle. If $\chv(\shm)$ is Lagrangian, $\shm$ is called holonomic. 
Denote by $\BDC_{\hol}(\D_X)\subset\BDC_\coh(\D_X)$ the full triangulated
subcategory whose objects have holonomic cohomologies.

If $Y\subset X$ is a closed hypersurface, denote by $\O_X(*Y)$ the sheaf of meromorphic functions with poles at $Y$.  For $\shm\in\BDC(\D_X)$, set
\[
\shm(*Y) = \shm \dtens \O_X(*Y).
\]
For $f\in\O_X(*Y)$ and $U=X\setminus Y$, set
\[
\D_X e^f = \D_X/\{P\semicolon P e^f=0 \text{ on } U\}, \quad
\she^f_{U|X} = \D_X e^f(*Y).
\]
These are holonomic $\D_X$-modules.

A complex bordered manifold is a pair $X_\infty=(X,\cX)$ of a complex manifold $\cX$ and an open subset $X\subset \cX$ such that $\cX\setminus X$ is a complex analytic subset of $\cX$.
A morphism $f\colon X_\infty\to(Y,\cY)$ of complex bordered manifolds is a morphism of bordered spaces such that the closure of its graph is a complex analytic subset of $\cX\times\cY$.

The triangulated category $\BDC_\hol(\D_{X_\infty})$ is the full triangulated subcategories of $\BDC_\hol(\D_{\cX})$ of objects $\shm$ satisfying $\shm\simeq\shm(*Z)$ for $Z=\cX\setminus X$.

The operations for $\D$-modules extend to bordered spaces. 
The family of holonomic $\D$-modules is stable by the operations, assuming semiproperness for direct images.

\section{Enhanced exponentials}\label{sec:exp}

As explained in \cite{DK16,DK16bis}, constant sheaves on the epigraphs of subanalytic functions are the building blocks of $\R$-constructible enhanced \mbox{(ind-)}sheaves. 
Here we consider their analogues on topological spaces, and state some of their properties.
(See \cite[\S3]{Moc16} for similar results.)

\subsection{Exponential enhanced sheaves}\label{sse:exp}

Let $M$ be a good topological space, and $U\subset M$ an open subset.
Let $\varphi,\varphi^+,\varphi^-\colon U\to\R$ be continuous
functions with $\varphi^-(x)\le\varphi^+(x)$ for any $x\in U$.
The associated exponential enhanced sheaves are defined by
\[
\ex^\varphi_{U|M} \defeq \quot\field_{\{t+ \varphi\geq 0\}}, \quad
\ex^{\range{\varphi^+}{\varphi^-}}_{U|M} \defeq \quot\field_{\{-\varphi^+ \leq t <-\varphi^-\}}.
\]
Here, $\{t+ \varphi\geq 0\}$ is short for $\{(x,t)\in U\times\R\semicolon t+ \varphi(x)\geq 0\}$, and similarly for $\{-\varphi^+ \leq t <-\varphi^-\}$.

By the definitions, one has
\[
\opb\pi\field_U\tens\ex^\varphi_{U|M}\simeq \ex^\varphi_{U|M}, \quad
\opb\pi\field_U\tens\ex^{\range{\varphi^+}{\varphi^-}}_{U|M}\simeq \ex^{\range{\varphi^+}{\varphi^-}}_{U|M}.
\]
Moreover,
$$\LE(\ex^\varphi_{U|M})\simeq \field_{\{t+ \varphi\geq 0\}}, \quad
\LE(\ex^{\range{\varphi^+}{\varphi^-}}_{U|M})
\simeq\field_{\{-\varphi^+ \leq t <-\varphi^-\}},
$$
so that in particular
\[
\ex^\varphi_{U|M},\ \ex^{\range{\varphi^+}{\varphi^-}}_{U|M} \in \dere^0_+(\field_M).
\]
Then, the natural exact sequence
\begin{equation}
\label{eq:SESex}
0\to
\field_{\{-\varphi^+ \leq t <-\varphi^-\}} \to \field_{\{t+ \varphi^+\geq 0\}}
\to \field_{\{t+ \varphi^-\geq 0\}} \to0
\end{equation}
induces an exact sequence in $\dere^0_+(\field_M)$
\begin{equation}
\label{eq:exphipsi}
0 \to \ex^{\range{\varphi^+}{\varphi^-}}_{U|M} \to
\ex^{\varphi^+}_{U|M} \to
\ex^{\varphi^-}_{U|M} \to 0.
\end{equation}

\begin{lemma}\label{lem:fgex}
Let $\varphi,\psi\colon U\to\R$ be continuous
functions.
\bnum
\item One has
$$\fhom(\ex^{\varphi}_{U|M},\ex^{\psi}_{U|M})
\simeq\rhom(\cor_{\{\psi\le\vphi\}},\cor_M),$$
where $\{\psi\le\vphi\} \defeq \{x\in U\semicolon\psi(x)\le\vphi(x)\}$.
\item
Assume that $\hom(\cor_U,\cor_M)\simeq\cor_{\ol{U}}$.
Then one has
$$
H^0\fhom(\ex^{\varphi}_{U|M},\ex^{\psi}_{U|M})
\simeq\cor_S,$$
where
\begin{align*}
S \defeq \{ x\in \ol{U} \semicolon{}&
\psi(y) \le\varphi(y) \text{ for any $y\in U\cap \Omega$ } \\
&\text{for some open neighborhood $\Omega$ of $x$} \}.
\end{align*}
\ee
\end{lemma}

\begin{proof}
(i)\ We have
\eqn
\fhom(\ex^{\varphi}_{U|M}\ex^{\psi}_{U|M})
&\simeq&\roim{\pi}\rhom(\cor_{\ens{t+ \vphi\ge0}},
\cor_{\ens{t+ \psi\ge0}})\\
&\simeq&\roim{\pi}\rhom\bl\cor_{\ens{t+ \vphi\ge0}},
\rhom(\cor_{\ens{t+ \psi>0}},\cor_{M\times\R})\br\\
&\simeq&\roim{\pi}\rhom\bl\cor_{\ens{t+ \vphi\ge0}}\tens
\cor_{\ens{t+ \psi>0}},\cor_{M\times\R}\br\\
&\underset{(*)}\simeq&\rhom\bl\reim{\pi}(\cor_{\ens{t+ \vphi\ge0}}\tens
\cor_{\ens{t+ \psi>0}}[1]),\cor_{M}\br\\
&\simeq&\rhom\bl\cor_{\ens{\psi\le\vphi}},\cor_{M}\br,
\eneqn
where $(*)$ follows from $\field_{M\times\R}[1]\simeq\epb\pi\field_M$.

\medskip\noi
(ii) Let $Z$ be the closure of $\ens{\psi\le\vphi}$ in $M$.
Note that $\ens{\psi\le\vphi} = Z\cap U$.
Then, by (i), one has
\begin{align*}
H^0\fhom(\ex^{\varphi}_{U|M},\ex^{\psi}_{U|M})
&\simeq \hom\bl\cor_{\ens{\psi\le\vphi}},\cor_{M}\br \\
&\simeq \hom(\cor_Z \tens \cor_U,\cor_M) \\
&\simeq \hom(\cor_Z,\hom(\cor_U,\cor_M)) \\
&\underset{(*)}\simeq \hom(\cor_Z,\cor_{\ol{U}}) \simeq \cor_S,
\end{align*}
where $(*)$ follows from the assumption, and the last isomorphism follows from the fact that $Z\subset\ol{U}$ and $S=\operatorname{Int}_{\ol{U}}(Z)$, the interior of $Z$ relative to $\ol{U}$. 
\end{proof}

\begin{corollary}\label{cor:fgex}
Let $\varphi,\psi\colon U\to\R$ be continuous
functions. Then
\begin{itemize}
\item[(i)] $\psi\leq\varphi$ on $U$ if and only if
there is an epimorphism
$\xymatrix@C=1.5em{\ex^{\varphi}_{U|M} \ar@{->>}[r] &
\ex^{\psi}_{U|M}}$ in $\dere^0(\field_M)$.
\item[(ii)] $\psi=\varphi$ on $U$ if and only if
there is an isomorphism $\ex^{\varphi}_{U|M} \simeq
\ex^{\psi}_{U|M}$ in $\dere^0(\field_M)$.
\end{itemize}
\end{corollary}

\begin{proof}
(i)
The only non trivial implication is the ``if'' part. 
Assume that there exists an epimorphism
$\xymatrix@C=1.5em{\ex^{\varphi}_{U|M} \ar@{->>}[r] &
\ex^{\psi}_{U|M}}$.
Then, the sheaf $$H^0\fhom(\ex^{\varphi}_{U|M},\ex^{\psi}_{U|M})\simeq
\hom(\cor_{\{\psi\le\vphi\}},\cor_M)$$
has non-zero stalk at every point of $U$.
Hence we have $\psi(x)\le\vphi(x)$ for any $x\in U$.

\medskip\noi
(ii) follows from (i).
\end{proof}

\subsection{Exponential enhanced ind-sheaves}\label{sse:iexp}

Let $M_\infty=(M,\cM)$ be a bordered space. 
We say that a subset $A$ of $M$ is a relatively compact subset of $M_\infty$ if the closure of $A$ in $\cM$ is compact.

Let $U\subset M$ be an open subset, and $\varphi,\varphi^+,\varphi^-\colon U\to\R$ be continuous functions with $\varphi^-\le\varphi^+$, as in the previous subsection. 
The associated exponential enhanced ind-sheaves are defined by
\[
\Ex^\varphi_{U|M_\infty} \defeq \Efield_{M_\infty}\ctens\ex^\varphi_{U|M}, \quad
\Ex^{\range{\varphi^+}{\varphi^-}}_{U|M_\infty} \defeq \Efield_{M_\infty}\ctens\ex^{\range{\varphi^+}{\varphi^-}}_{U|M},
\]
where we used the natural embedding $\dereb_+(\field_M)\subset\dereb_+(\ifield_{M_\infty})$.
By the definitions, one has
\[
\opb\pi\field_U\tens\Ex^\varphi_{U|M_\infty}\simeq \Ex^\varphi_{U|M_\infty}, \quad
\opb\pi\field_U\tens\Ex^{\range{\varphi^+}{\varphi^-}}_{U|M_\infty}\simeq \Ex^{\range{\varphi^+}{\varphi^-}}_{U|M_\infty}.
\]

Since the functor $\Efield_{M_\infty}\ctens\ast$ is exact by \cite[Lemma 2.8.2]{DK16}, one has
\[
\Ex^\varphi_{U|M_\infty},\
\Ex^{\range{\varphi^+}{\varphi^-}}_{U|M_\infty} \in\dere^0_+(\ifield_{M_\infty}).
\]
Moreover, \eqref{eq:exphipsi}
induces an exact sequence in $\dere^0_+(\ifield_{M_\infty})$
\[
0 \to \Ex^{\range{\varphi^+}{\varphi^-}}_{U|M_\infty} \to
\Ex^{\varphi^+}_{U|M_\infty} \to
\Ex^{\varphi^-}_{U|M_\infty} \to 0.
\]
One has
\begin{align*}
\Ex^\varphi_{U|M} 
&\simeq \quot \bl \indlim[c\to+\infty] \field_{\{t+\varphi\geq c\}} \br, \\
\Ex^{\range{\varphi^+}{\varphi^-}}_{U|M} 
&\simeq \quot \bl \indlim[c,d\to+\infty] \field_{\{-\varphi^++c\leq t < -\varphi^-+d\}} \br.
\end{align*}

\begin{notation}
Let  $\varphi,\psi\colon U\to\R$ be continuous functions.
\begin{itemize}
\item[(i)]
$\psi\aleq{U}\varphi$ means that $\psi-\varphi$ is bounded from above on 
$K\cap U$ for any relatively compact subset $K$ of $M_\infty$.
\item[(ii)]
$\psi\asim U\varphi$ means that $\psi\aleq{U}\varphi$ and $\varphi\aleq{U}\psi$,
i.e.\ that $\psi-\varphi$ is bounded on
$K\cap U$ for any relatively compact subset $K$ of $M_\infty$.
\end{itemize}
\end{notation}

\begin{lemma}\label{lem:fgEx}
Let $\varphi,\psi\colon U\to\R$ be continuous
functions.
\bnum
\item For any $n\in\Z$, one has
$$H^n\fhom(\Ex^{\varphi}_{U|M},\Ex^{\psi}_{U|M})
\simeq\ilim[{c\to+\infty}]H^n\rhom(\cor_{\{\psi\le\vphi+c\}},\cor_M),$$
where $\{\psi\le\vphi+c\} \defeq \{x\in U\semicolon\psi(x)\le\vphi(x)+c\}$.
\item
Assume that $\hom(\cor_U,\cor_M)\simeq\cor_{\ol{U}}$.
Then one has
$$
H^0\fhom(\Ex^{\varphi}_{U|M},\Ex^{\psi}_{U|M})
\simeq\cor_T,$$
where
\[
T \defeq \{ x\in \ol{U} \semicolon
\psi\aleq{U\cap \Omega}\varphi \text{ for some open neighborhood $\Omega$ of $x$} \}.
\]
\ee
\end{lemma}

\begin{proof}
(i) We have
\begin{align*}
H^n\fhom{}&(\Ex^{\varphi}_{U|M},\Ex^{\psi}_{U|M}) \\
& \simeq H^n\fhom(\quot \field_{\{t+\varphi\geq 0\}} , \quot \bl \indlim[c\to+\infty] \field_{\{t+\psi\geq c\}} \br) \\
&\underset{(*)}\simeq \ilim[{c\to+\infty}]H^n\fhom(\quot\field_{\{t+\varphi\geq 0\}}, \quot\field_{\{t+\psi\geq c\}}) \\
&\underset{(**)}\simeq \ilim[{c\to+\infty}]H^n\rhom(\field_{\{\psi\le\vphi+c\}},\field_M),
\end{align*}
where $(*)$ follows from \cite[Proposition 4.7.9]{DK16} and $(**)$ from Lemma~\ref{lem:fgex}.

\medskip\noi
(ii) By (i) and Lemma~\ref{lem:fgex}~(ii), one has
\[
H^0\fhom(\Ex^{\varphi}_{U|M},\Ex^{\psi}_{U|M})
\simeq \ilim[{c\to+\infty}]\field_{S_c},
\]
where
\begin{align*}
S_c \defeq \{ x\in \ol{U} \semicolon{}&
\psi(y)\le\varphi(y)+c \text{ for any $y\in U\cap \Omega$ } \\
&\text{for some open neighborhood $\Omega$ of $x$} \}.
\end{align*}
The statement follows.
\end{proof}

The following result was stated in \cite[\S3.3]{DK16bis} in the subanalytic case.

\begin{corollary}\label{cor:fgEx}
Let $\varphi^+,\varphi^-\colon U\to\R$ be continuous functions such that $\varphi^-\le\varphi^+$. Then
$\Ex^{\range{\varphi^+}{\varphi^-}}_{U|M_\infty}\simeq 0$ if and only if 
$\varphi^+ \asim {U} \varphi^-$.
\end{corollary}

\begin{proof}
By Lemma~\ref{lem:fgEx}, $\varphi^+ \asim {U} \varphi^-$ is equivalent to $\Ex^{\range{\varphi^+}{\varphi^-}}_{U|\cM}\simeq 0$. 

Let $j\colon M_\infty\to \cM$ be the natural morphism.
Then, one concludes by noticing that $\Ex^{\range{\varphi^+}{\varphi^-}}_{U|\cM}\simeq\Eeeim j(\Ex^{\range{\varphi^+}{\varphi^-}}_{U|M_\infty})$ and $\Ex^{\range{\varphi^+}{\varphi^-}}_{U|M_\infty}\simeq\Eopb j(\Ex^{\range{\varphi^+}{\varphi^-}}_{U|\cM})$.
\end{proof}

\subsection{$\R$-constructibility}

Assume now that $M_\infty=(M,\cM)$ is a subanalytic bordered space.
We say that a subset $A$ of $M$ is subanalytic in $M_\infty$ if $A$ is a subanalytic subset of $\cM$.

Let $U$ be an open subanalytic subset of $M_\infty$.
One says that a function $\varphi\colon U \to \R$ is globally subanalytic if its 
graph is subanalytic in $M_\infty\times\R_\infty$.

Let $\varphi,\varphi^+,\varphi^-\colon U\to\R$ be continuous and globally subanalytic functions with $\varphi^-\le\varphi^+$. Then
\begin{align*}
\ex^\varphi_{U|M},\ \ex^{\range{\varphi^+}{\varphi^-}}_{U|M} &\in \dere^0_\Rc(\field_{M_\infty}), \\
\Ex^\varphi_{U|M_\infty},\ \Ex^{\range{\varphi^+}{\varphi^-}}_{U|M_\infty} &\in \dere^0_\Rc(\ifield_{M_\infty}).
\end{align*}

As explained in \cite[Lemma 4.2.9]{DK16} and \cite[\S3.3]{DK16bis}, these are in fact the building blocks of $\R$-constructible enhanced (ind-)sheaves.

\section{Real blow-up}\label{sec:blow}

In this section $M$ is a smooth manifold, except in \S\ref{sse:RcSect} where it is a real analytic manifold.

\medskip 
We use here the real blow-up of a point on a manifold to describe morphisms of exponential enhanced \mbox{(ind-)}sheaves on small sectors.

\subsection{Notations}\label{sse:NotBlow}

Let $M$ be a smooth manifold of dimension $n\geq 1$, and let $a\in M$.
The total real blow-up 
\[
\varpi_a^\tot\colon \widetilde M_a^\tot \to M
\]
of $M$ along $a$ is the map of smooth manifolds locally defined as follows,
in local coordinates $(x_1,\dots,x_n)$ with $a=(0,\dots,0)$,
\begin{align*}
&\widetilde M_a^\tot = \{(\rho,\xi)\in \R\times\R^n \semicolon |\xi|=1 \}, \\
&\varpi_a^\tot\colon \widetilde M_a^\tot \to M, \quad (\rho,\xi) \mapsto \rho\xi.
\end{align*}
The real blow up $\tM_a$ of $M$ at $a$, is the closed subset $\{\rho\geq 0\}$ of $\widetilde M_a^\tot$.
Setting $\varpi_a = \varpi_a^\tot|_{\widetilde M_a}$, consider the commutative diagram
\begin{equation}
\label{eq:blow}
\xymatrix@R=1ex{
S_a M \ar@{^(->}[r]^-{\ti_a} & \widetilde M_a \ar[dd]_{\varpi_a} & \\
&& M\setminus\{a\}, \ar@{_(->}[dl]^-{j_a} \ar@{_(->}[ul]_-{\tj_a} \\
& M
}
\end{equation}
where $S_a M \defeq \varpi_a^{-1}(a)\simeq S^{n-1}$ is the sphere of tangent directions at $a$.

Let $\theta\in S_aM$ and $V\subset M$. One says that $V$ is a \emph{sectorial neighborhood} of $\theta$ if $V\subset M\setminus\{a\}$ and $S_a M\cup \tj_a (V)$ is a
neighborhood of $\theta$ in $\tM_a$. This is equivalent to saying that $V=\opb\tj_a(U)$ for some neighborhood $U$ of $\theta$ in
$\widetilde M_a$. 

For $\theta\in S_a X$, we will:
\begin{itemize}
\item
write for short $x\to\theta$ instead of $\tj_a(x)\to\theta$,
\item
write $\theta\dotin V$ to indicate that $V$ is a sectorial neighborhood of $\theta$,
\item
say that a property $P(x)$ holds at $\theta$, if there exists $V\dotowns\theta$ such that $P(x)$ holds for any $x\in V$.
\end{itemize}

One says that $U\subset M\setminus\{a\}$ is a sectorial neighborhood of $I\subset S_a X$ if $U\dotowns\theta$ for any $\theta\in I$.

\subsection{Nearby morphisms}

Consider the maps \eqref{eq:blow}.

\begin{lemmadefinition}\label{def:nuhom}
For $K_1,K_2\in\dereb_+(\ind\field_M)$, consider the object of $\BDC(\field_{S_a M})$
\begin{align*}
\nuhom[a](K_1,K_2) 
&\defeq \opb{\ti_a}\fhom(\Eopb\varpi_a(\opb\pi\field_{M\setminus\{a\}}\tens K_1),\Eepb\varpi_a K_2) \\
&\simeq \opb{\ti_a}\fhom(\Eopb\varpi_a(\opb\pi\field_{M\setminus\{a\}}\tens K_1),\Eopb\varpi_a K_2) .
\end{align*}
\end{lemmadefinition}

\begin{proof}
One has
\begin{align*}
\fhom(&\Eopb\varpi_a(\opb\pi\field_{M\setminus\{a\}}\tens K_1),\Eepb\varpi_a K_2) \\
&\simeq \fhom(\opb\pi\field_{\tj_a(M\setminus\{a\})}\tens \Eopb\varpi_a K_1,\Eepb\varpi_a K_2) \\
&\simeq \fhom(\Eopb\varpi_a K_1,\rihom(\opb\pi\field_{\tj_a(M\setminus\{a\})},\Eepb\varpi_a K_2)),
\end{align*}
and similarly for $\Eepb\varpi_a K_2$ replaced with $\Eopb\varpi_a K_2$.
It is then enough to remark that
\[
\rihom(\opb\pi\field_{\tj_a(M\setminus\{a\})},\Eepb\varpi_a K_2)
\simeq
\rihom(\opb\pi\field_{\tj_a(M\setminus\{a\})},\Eopb\varpi_a K_2).
\]
\end{proof}

With the above notations, 
Lemma~\ref{lem:stalk} implies

\begin{lemma}
\label{lem:nustalk}
For $K_1,K_2\in\dereb_+(\ind\field_M)$ and $\theta\in S_aX$, one has
\[
\nuhom[a](K_1,K_2)_\theta \simeq
\ilim[V\dotowns\theta] \FHom(\opb\pi\field_V\tens K_1,\opb\pi\field_V\tens K_2).
\]
\end{lemma}

\subsection{Comparing exponentials on small sectors}\label{sse:sectmorph}

Let $U\subset M$ be an open subset, and $\varphi\colon U \to \R$ a continuous function. 
For $V\subset U$ open, let us set for short
\[
\ex^{\varphi}_{V|M} \defeq \ex^{\varphi|_V}_{V|M} \simeq  \opb\pi\field_V\tens\ex^{\varphi}_{U|M},
\]
and similarly for $\Ex^{\varphi}_{V|M}$.

\begin{lemma}\label{lem:Ephi12}
Let $a\in M$, $\theta\in S_a M$, and $U\subset M$ an open sectorial neighborhood of $\theta$.
Let $\varphi, \psi\colon U \to \R$ continuous functions. 
\begin{itemize}
\item[(i)]
If $\psi-\varphi$ is bounded from above at $\theta$, then
\[
\nuhom[a]\bl\Ex^{\varphi}_{U|M}, \Ex^{\psi}_{U|M}\br_\theta \simeq\field.
\]
\item[(ii)]
If $\psi \leq \varphi$ at $\theta$, then
\[
\field \simeq \nuhom[a]\bl\ex^{\varphi}_{U|M}, \ex^{\psi}_{U|M}\br_\theta
\isoto \nuhom[a]\bl\Ex^{\varphi}_{U|M}, \Ex^{\psi}_{U|M}\br_\theta,
\]
where the last isomorphism is induced by $\field^\enh_M\ctens\ast$.
\item[(iii)]
If $\lim\limits_{x\to \theta}\bl\psi(x)-\varphi(x)\br = +\infty$, then
\[
\nuhom[a]\bl\ex^{\varphi}_{U|M}, \ex^{\psi}_{U|M}\br_\theta \simeq 0 \simeq \nuhom[a]\bl\Ex^{\varphi}_{U|M}, \Ex^{\psi}_{U|M}\br_\theta.
\]
\end{itemize}
\end{lemma}

\begin{proof}
(a) Let us first prove the statements concerning $\Ex^{\varphi}_{U|M}$ and $\Ex^{\psi}_{U|M}$, namely
\[
\ilim[V\dotowns\theta]\FHom(\Ex^{\varphi}_{V|M}, \Ex^{\psi}_{V|M}) \simeq
\begin{cases}
\field & \text{if {\rm (i)} holds,} \\
0 & \text{if {\rm (iii)} holds.}
\end{cases}
\]

By Lemma~\ref{lem:fgEx}, one has
\[
\ilim[V\dotowns\theta]\FHom(\Ex^{\varphi}_{V|M}, \Ex^{\psi}_{V|M}) \simeq
\ilim[V\dotowns\theta,\ {c\to+\infty}]\RHom(\cor_{V\cap\{\psi-\vphi\le c\}},\cor_M).
\]

\medskip\noindent(a-1) Under hypothesis (i), there exist $C>0$ and an open sectorial neighborhood $V_\circ\subset U$ 
of $\theta$ such
that $\psi(x)-\varphi(x)\leq C$ for any $x\in V_\circ$. Then
\[
V\cap\{\psi-\varphi\leq c \} = V
\]
if $V\subset V_\circ$ and $c\geq C$. 
Up to shrinking, we may further assume that $V$ is contractible.
Under these conditions, one has
\[
\RHom(\cor_{V\cap\{\psi-\vphi\le c\}},\cor_M)\simeq
\RHom(\field_V,\field_M)\simeq
\rsect(V; \field_V)\simeq\field.
\]

\medskip\noindent(a-2) Assume the condition in (iii).
Then, for any $c>0$ there exists an open subset $V_c\subset U$ with $\theta\dotin V_c$ such that $V_c\cap\{\psi-\varphi\leq c\}=\emptyset$. Hence, for $V\subset V_c$ one has
\[
\RHom(\cor_{V\cap\{\psi-\vphi\le c\}},\cor_M)\simeq0.
\]

\medskip\noindent(b) The statements concerning $\ex^{\varphi}_{U|M}$ and $\ex^{\psi}_{U|M}$, namely
\[
\ilim[V\dotowns\theta,\ V\subset U]\FHom(\ex^{\varphi}_{V|M}, \ex^{\psi}_{V|M}) \simeq
\begin{cases}
\field & \text{if {\rm (ii)} holds,} \\
0 & \text{if {\rm (iii)} holds,}
\end{cases}
\] 
are proved as in (a) above, using Lemma~\ref{lem:fgex} instead of Lemma~\ref{lem:fgEx}.
\end{proof}

\subsection{$\R$-constructible enhanced sheaves on generic sectors}\label{sse:RcSect}

Let $M$ be a real analytic smooth surface, and recall the diagram \eqref{eq:blow}.

We say that a statement $P(\theta)$ on $\theta\in S_aM$ holds for \emph{generic} $\theta$ if it holds for $\theta$ outside a finite subset of $S_aM$.

\begin{lemma}\label{lem:RcStruct}
Let  $M$ be a real analytic smooth surface, and let $F\in\dereb_\Rc(\field_M)$. Then, for generic $\theta\in S_aM$,  there exists a subanalytic open subset $V\subset M$ such that $\theta\dotin V$ and
\begin{equation}
\label{eq:VG}
\opb\pi\field_V \tens F \simeq
\bl \DSum_{i\in I} \ex_{V|M}^{\varphi_i}[d_i] \br
\oplus
\bl \DSum_{j\in J} \ex_{V|M}^{\range{\varphi_j^+}{\varphi_j^-}}[d_j] \br.
\end{equation}
Here, $I, J$ are finite sets,
$d_i,d_j\in\Z$, and
$\varphi_i,\varphi^\pm_j\colon V\to\R$ are
analytic and globally subanalytic functions,
such that $\varphi^-_j(x) < \varphi^+_j(x)$ for any $x\in V$.
\end{lemma}

A similar statement holds for $K\in\dereb_\Rc(\ifield_M)$, replacing $\ex$ with $\Ex$.

\begin{proof}
By \cite[Lemma 4.9.9]{DK16} (see also \cite[Remark 3.5.12]{DK16bis}), there exists an open subanalytic subset $U\subset M\setminus\{a\}$ such that $Z\defeq M\setminus U$ has dimension $\leq1$ and the isomorphism \eqref{eq:VG} holds for any connected component $V$ of $U$.
However, in loc.\ cit.\ the functions $\varphi_i,\varphi^\pm_j$ are not supposed to be analytic, and only satisfy the weak inequalities $\varphi^-_j \leq \varphi^+_j$. 

We can assume that the functions $\varphi_i,\varphi^\pm_j$ are analytic, after removing from $V$ their singular loci. Then, we can remove those indices $j\in J$ for which $\varphi^-_j - \varphi^+_j\equiv 0$ on $V$. Moreover, we can remove from $V$ the zero loci of $\varphi^-_j - \varphi^+_j$ for the remaining indices $j\in J$.
This shrinking of the connected components of $U$ leaves $Z$ of dimension $\leq1$.

To conclude, note that $\theta\dotin U$ if $\theta\in S_aM$ is outside the finite set
$S_aM\cap\overline{\bl\opb\varpi_a Z\setminus S_aM\br}$.
Hence, $\theta\dotin V$ for some connected component $V$ of $U$.
\end{proof}

\section{Enhanced normal form}\label{sec:normal}

In this section, $X$ denotes a smooth complex analytic curve.

\medskip

As recalled in \S\ref{sse:LT}, the Hukuhara-Levelt-Turrittin theorem describes the formal and asymptotic structure of holonomic $\D$-modules.
We discuss here an analogous condition for enhanced ind-sheaves.

\subsection{Puiseux germs}

Let $X$ be a smooth complex analytic curve, and take $a\in X$.
Recall from Section~\ref{sse:NotBlow} the notations relative to the real blow up $\tX_a$ of $X$ at $a$, i.e.\ the blow-up of the smooth real analytic surface underlying $X$:
\begin{equation}
\label{eq:IStX}
\xymatrix@R=1ex{
S_a X \ar@{^(->}[r]^-{\ti_a} & \widetilde X_a \ar[dd]_{\varpi_a} & \\
&& X\setminus\{a\}. \ar@{_(->}[dl]^-{j_a} \ar@{_(->}[ul]_-{\tj_a} \\
& X
}
\end{equation}
Here, $S_a X\simeq S^1$ is the circle of tangent directions at $a$.

A local coordinate $z_a$ at $a$ is a holomorphic function $z_a$ defined on a neighborhood of $a$ such that $z_a(a)=0$ and $(d z_a)(a)\neq0$.

\begin{definition}\label{def:Puis} \hfill
\begin{itemize}
\item[(i)]
Let $\theta\in S_aX$ and $U\dotowns\theta$. We say that $f\in\O_X(U)$ \emph{admits a
Puiseux expansion} at $\theta$ if there exist $\ram\in\Z_{> 0}$, a local coordinate $z_a$ at $a$, an open subset
$V\subset U$ with $\theta\dotin V$, and a determination of $z_a^{1/\ram}$ on $V$, such
that $f(x)=h(z_a(x)^{1/\ram})$ for $x\in V$, where $h$ is a section
of $\O_\C(*0)$ in a neighborhood of $0$.
\item[(ii)]
We denote by $\shp_{\widetilde X_a}$ the subsheaf of $\oim{{\tj_a}}\opb j_a\O_X$ whose
sections on $U\subset \widetilde X_a$ are holomorphic functions on
$\opb{\tj_a} U$ admitting a Puiseux expansion at any point of
$U\cap S_aX$.
\item[(iii)]
The sheaf $\shp_{S_aX}\defeq\opb{\ti_a}\shp_{\widetilde X_a}$ is called the
\emph{sheaf of Puiseux germs} on $S_aX$. It is a locally constant sheaf.
\item[(iv)]
For $\lambda\in\Q$, let $\shp^{\le\lambda}_{a|X}\subset\shp_{a|X}$ be the subsheaf of sections of \emph{pole order} $\le\lambda$, i.e.\ of sections that locally belong to
\[
\Union_{\ram\in\Z_{\geq 1}}z_a^{-\lambda}\C\{z_a^{1/\ram}\}
\]
for some (hence, any)  local coordinate $z_a$ at $a$, and some (hence, any) determination of $z_a^{1/\ram}$ at $\theta$.
\item[(v)]
Set $\oshp_{S_aX} \defeq \shp_{S_aX}/\shp^{\leq 0}_{S_aX}$
and, for $f\in\shp_{S_aX}$, denote by $[f]$ its image in $\oshp_{S_aX}$.
\item[(vi)]
For $\lambda\in\Q$, let $\shp^{\lambda}_{S_aX}\subset\shp_{S_aX}$ be the
subsheaf of sections of \emph{pole order} $\lambda$, i.e.\ of sections that
locally at any $\theta\in S_aX$ belong to
\[
\Union_{\ram\in\Z_{\geq 1}}z_a^{-\lambda}\bigl(\C^\times+z_a^{1/\ram}\C\{z_a^{1/\ram}\}\bigr),
\]
for some (hence, any)  local coordinate $z_a$ at $a$, and some (hence, any) determination of $z_a^{1/\ram}$ at $\theta$.
\item[(vii)]
Let $I\subset S_aX$ be an open connected subset.
For non-zero $f\in\shp_{S_aX}(I)$ we set $\ord_a(f) \defeq \lambda$, where $\lambda$ is the unique rational such that $f\in\shp^{\lambda}_{S_aX,\theta}$ for some (hence, any) $\theta\in I$. It is called the \emph{pole order} of $f$. We set $\ord_a(0)=-\infty$.
\item[(viii)]
For $K\subset\R$ an interval, set $\shp_{S_aX}^K \defeq \coprod_{\lambda\in
K\cap\Q}\shp_{S_aX}^\lambda$.
\end{itemize}
\end{definition}

The sheaves $\shp_{S_aX}$, and $\shp^{\leq\lambda}_{S_aX}$ are locally constant sheaves on $S_aX$,
and one has
\[
\sect\bl S_aX ;\shp_{S_aX}\br \simeq \bigl(\O_X(*a)\bigr)_a,
\quad
\sect\bl S_aX ;\shp_{S_aX}^{\leq 0}\br \simeq (\O_X)_a.
\]
The sheaf $\shp^\lambda_{S_aX}$ is a locally constant subsheaf of $\shp_{S_aX}$. 

\begin{definition}\label{def:wellsep}
Let $\theta\in S_aX$ and $\Phi\subset\shp_{S_aX,\theta}$. One says that $\Phi$ is well separated if for any $f,h\in\Phi$
\begin{enumerate}
\item $[f]=0$ 
implies $f = 0$,
\item $[f] = [h]$
implies $f = h$.
\end{enumerate}
In particular, this implies that there is a natural bijection between $\Phi$ and its class $\overline\Phi\subset\oshp_{S_aX,\theta}$.
\end{definition}

\begin{definition}\label{def:wellrep}
A subsheaf $\shp_{S_aX}'\subset\shp_{S_aX}$ of $\C$-vector spaces is called representative if
the quotient map $f\mapsto[f]$ induces an isomorphism  $\shp_{S_aX}'\isoto\overline\shp_{S_aX}$.
\end{definition}

For example, for a choice of local coordinate $z_a$ at $a$, let $\shp_{S_aX}'\subset\shp_{S_aX}$ be the subsheaf of sections that belong to
\[
\Union_{\ram\in\Z_{\geq 1}}z_a^{-1/p}\C[z^{-1/p}]
\]
for some (hence, any) local determination of $z_a^{1/\ram}$.
Then, $\shp_{S_aX}'$ is representative.

Note that if $\shp_{S_aX}'\subset\shp_{S_aX}$ is a representative subsheaf, then it is a locally constant sheaf, and its stalks are well separated.

\begin{notation}\label{not:aleq}
Let $\theta\in S_aX$ and $f,h\in\shp_{S_aX,\theta}$.
\begin{itemize}
\item[(i)]
$h \asim\theta f$ means that $\Re(h-f)$ is bounded at $\theta$.
\item[(ii)]
$h \aleq\theta f$ means that $\Re(h-f)$ is bounded from above at $\theta$,
\item[(iii)]
$h \aless\theta f$ means that $h \aleq\theta f$ and $h
\not\asim\theta f$.
\end{itemize}
\end{notation}

Note that all of these conditions are open in $\theta$.

\begin{lemma}
Let $f,h\in\shp_{S_aX}(I)$, with $I\subset S_aX$ an open connected subset.
Then the following are equivalent
\begin{itemize}
\item[(i)] $f \asim\theta h$ for some $\theta\in I$,
\item[(ii)] $f \asim\theta h$ for any $\theta\in I$,
\item[(iii)] $h-f\in\shp_{S_aX}^{\le0}(I)$.
\end{itemize}
In other words, $[f]=[h]$ in $\oshp_{S_aX}(I)$ if and only if $h-f$ is bounded at $\theta$ for some {\rm(}hence, any{\rm)} $\theta\in I$.
\end{lemma}

\begin{definition}
\label{def:Stokes} 
Let $I\subset S_aX$ be a connected open subset.
The \emph{Stokes directions} of a pair $f,h\in\shp_{S_aX}(I)$ are the points
of $I$ where $f$ and $h$ are not comparable, i.e.
\[
\St_a(f,h) \defeq \{ \theta\in I \semicolon f \not\aleq\theta h,\ h
\not\aleq\theta f \}.
\]
\end{definition}

The set $\St_a(f,h)$ is a finite subset of $I$.

Note that $\St_a(f,h)\neq\emptyset$ implies $h-f\notin\shp_{S_aX}^{\leq0}(I)$.

For $\Phi_1,\Phi_2\subset\shp_{S_aX}(I)$, one sets
\[
\St_a(\Phi_1,\Phi_2) \defeq \Union_{f\in\Phi_1,\ h\in\Phi_2}\St_a(f,h).
\]

Assume that $h-f\notin\shp_{S_aX}^{\le0}(I)$, and let $U$ be a sectorial open neighborhood of $I$ where $f$ and $h$ are defined (i.e.\ $f,h\in\O_X(U)$).
The Stokes curves of $f,h$ are the real analytic arcs $\Gamma_s\subset U$ defined by
\[
\{x\in U\semicolon \Re\bl h(x)-f(x)\br=0 \} = \Union_{s\in S}\Gamma_s.
\]
After shrinking $U$, we may assume that the $\Gamma_s$'s do not cross, and that
they are non singular with $|z_a|$ as parameter, for $z_a$ a  local coordinate at $a$. Then $\St_a(f,h)$ is the set of tangent directions at $a$ of
the Stokes curves.

\begin{lemma}
\label{lem:Stepsilon}
Let $f,h\in\shp_{S_aX,\theta}$ and assume that $\theta\in\St_a(f,h)$.
Then, $\theta\in\St_a(f+k,h)$ for any $k\in\shp_{S_aX,\theta}$ such that $\ord_a(k)<\ord_a(h-f)$.
\end{lemma}

\subsection{A lemma on nearby morphisms}

Let $a\in X$ and  $\theta\in S_aX$. 
Recall the maps \eqref{eq:IStX}.

\begin{lemma}
\label{lem:homEfgtheta}
Let $f,h\in\shp_{S_aX,\theta}$, and let $U$ be an open sectorial neighborhood of $\theta$ where $f$ and $h$ are defined.
\begin{itemize}
\item[(i)] 
One has
\[
\nuhom[a]\bl\Ex^{\Re f}_{U| X},\Ex^{\Re h}_{U| X}\br_\theta \simeq
\begin{cases}
\field, & \text{if } h \aleq\theta f,\\
0,&\text{otherwise}
\end{cases}
\]
\item[(ii)] 
Assume that either $\Re h\leq\Re f$ at $\theta$, or $f \not\asim \theta h$.
Then, one has a natural isomorphism
\[
\nuhom[a]\bl\ex^{\Re f}_{U| X},\ex^{\Re h}_{U| X}\br_\theta
\isoto 
\nuhom[a]\bl\Ex^{\Re f}_{U| X},\Ex^{\Re h}_{U| X}\br_\theta.
\]
\end{itemize}
\end{lemma}

\begin{proof}
(i) By the definition, one has
\[
\nuhom[a]\bl\Ex^{\Re f}_{U| X},\Ex^{\Re h}_{U| X}\br_\theta
\simeq
\ilim[V\dotowns\theta] \FHom(\Ex^{\Re f}_{V|X},\Ex^{\Re h}_{V|X}).
\]
\medskip\noindent
(i-a) If $h \aleq\theta f$, the statement follows from Lemma~\ref{lem:Ephi12}~(i). 

\medskip\noindent
(i-b) Otherwise $h-f \not\aleq\theta 0$, so that in particular $h-f\in\shp_{S_aX,\theta}^\lambda$ for some $\lambda\in\Q_{>0}$. After a ramification and the choice of a local coordinate $z_a$ at $a$, we can assume that $h-f=z_a^{-1}$ and $\theta$ is a direction in the closed half-space $\Re z_a\geq 0$.
By Lemma~\ref{lem:fgEx} (i), one has
\[
\nuhom[a]\bl \Ex^{\Re f}_{U| X},\Ex^{\Re h}_{U| X} \br_\theta \simeq 
\ilim[V\dotowns\theta,\ c\to+\infty] \RHom(\field_{V\cap\{\Re(z_a^{-1})\leq c\}},\field_X).
\]
For $c>0$, $\{\Re(z_a^{-1})> c\}$ is an open disc with center $1/2c$ and radius $1/2c$. Hence there is a cofinal system of sectorial neighborhoods $V$ of $\theta$ such that for any $c\gg0$ one has
\[
\RHom(\field_{V\cap\{\Re(z_a^{-1})\leq c\}},\field_X)
\simeq 
\RHom(\rsect_c(X;\field_{V\cap\{\Re(z_a^{-1})\leq c\}}) ,\field)[-2]
\simeq 0.
\]
Indeed, we can take $V=\{z_a\in\gamma\semicolon |z_a|<\varepsilon\}$ for $\varepsilon>0$ and $\gamma$ an open convex proper cone intersecting $\Re z_a>0$ (see Figure~\ref{fig:cone}).

\begin{figure}
\def\rrss{0.684}
\def\rrdd{1}
\def\lll{115}
\def\rrr{65}
\begin{tikzpicture}[scale=1.5,
	baseline=(O.base)]
\draw (0,0) coordinate (O) ;
\begin{scope}
  \clip(-1,-.5) rectangle (1.2,1.2);
\filldraw[fill=black!20,draw=white,thick] (0,0) -- (\lll:\rrss) arc (\lll:\rrr:\rrss) -- (0,0);
\filldraw[fill=white,draw=black,thick] (\rrdd,0)circle[radius=\rrdd] ;
\draw[thick,densely dotted] (0,0) -- (\lll:\rrss) arc (\lll:\rrr:\rrss) -- (0,0);
\draw[->,thin] (-1,0)--(1.2,0) node[below left]{$\Re z_a$};
\draw[->,thin] (0,-.5)--(0,1.2) node[below left]{$\Im z_a$};
\draw[fill=white,draw=black] (O) circle (1pt) ;
\draw[fill=white,draw=black] (70:\rrss) circle (1pt) ;
\end{scope}
\end{tikzpicture}
\caption{In gray, the intersection $V\cap\{\Re(z_a^{-1})\leq c\}$ when $\theta$ is the direction of the imaginary axis..}\label{fig:cone}
\end{figure}
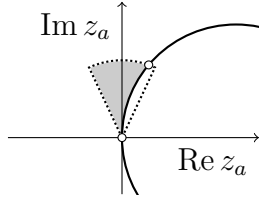

\medskip\noindent
(ii) The proof is similar to that of (i).

\medskip\noindent
(ii-a) If $\Re h\leq\Re f$ at $\theta$, the statement follows from Lemma~\ref{lem:Ephi12}~(ii). 

\medskip\noindent
(ii-b) Otherwise $\Re (h- f)>0$ at $\theta$, and $f \not\asim \theta h$. Hence,
a similar argument as in (i-b) holds, for $\Ex$ replaced by $\ex$, using Lemma~\ref{lem:fgex} instead of Lemma~\ref{lem:fgEx}.
Indeed, with the previous notations one has
\[
\nuhom[a]\bl \ex^{\Re f}_{U| X},\ex^{\Re h}_{U| X} \br_\theta \simeq 
\ilim[V\dotowns\theta] \RHom(\field_{V\cap\{\Re(z_a)\leq 0\}},\field_X).
\]
\end{proof}

Let $I\subset S_aX$ be a connected open subset, and denote by $j_I\colon I\to S_aX$ the embedding. Then, Lemma~\ref{lem:homEfgtheta} immediately implies the following lemma.

\begin{lemma}
\label{lem:homEfgI}
Let $f,h\in\shp_{S_aX}(I)$, and let $U$ be an open sectorial neighborhood of $I$ where $f$ and $h$ are defined.
Then, there is an isomorphism
\[
\opb{j_I}\nuhom[a](\Ex^{\Re f}_{U| X},\Ex^{\Re h}_{U| X})
\simeq \field_{\{\theta\in I\semicolon h \aleq\theta f\}}.
\]
\end{lemma}

\begin{corollary}[{cf \cite[Lemma 3.15]{Moc16}}]\label{cor:homEfgI}
Let $I\subset S_aX$ be a connected open subset, and let $\{f_a\}_{a\in A}$ and $\{h_b\}_{b\in B}$ be finite families of elements of $\shp_{S_aX}(I)$.
If there is an isomorphism
\[
\DSum_{a\in A}\Ex_{U|X}^{f_a}\simeq\DSum_{b\in B}\Ex_{U|X}^{g_b}, 
\]
then there exists a bijection $u\colon A\to B$ such that $[f_a]=[h_{u(a)}]$ for any $a\in A$.
\end{corollary}

\subsection{Normal form for enhanced sheaves}\label{sse:Normenh}
Recall Definition~\ref{def:wellsep}.

A \emph{multiplicity} at $a\in X$ is a morphism of sheaves of sets
\[
N\colon \shp_{S_aX} \to (\Z_{\geq0})_{S_aX}
\]
such that $N^{>0}_\theta \defeq N_\theta^{-1}(\Z_{>0}) \subset \shp_{S_aX,\theta}$ is well separated and finite for some (hence, any) $\theta\in S_aX$.

A Puiseux germ $f \in N^{>0}_\theta$ is called an exponential factor of $N$ at $\theta$, and the positive integer $N(f)$ is called its multiplicity.

\begin{definition}\label{def:eform}
Let $a\in X$ and $F\in\dereb_\Rc(\field_X)$.
One says that $F$ has a normal form at $a$ if there exists a multiplicity $N\colon \shp_{S_aX} \to (\Z_{\geq0})_{S_aX}$ such that any $\theta\in S_a X$ has an open sectorial neighborhood $V_\theta$ such that
\[
\opb\pi\field_{V_\theta}\tens F \simeq 
\DSum_{f\in N^{>0}_\theta} \bl\ex^{\Re f}_{V_\theta|X}\br^{N(f)}.
\]
\end{definition}

Note that the multiplicity $N$ is uniquely determined by $F$.

Note also that if $F$ has a normal form at $a$, then there exists an open neighborhood $\Omega$ of $a$ such that
\[
\opb\pi\field_{\Omega\setminus\{a\}}\tens F\in\dere^0_\Rc(\field_X).
\]

\subsection{Normal form for enhanced ind-sheaves}\label{sse:normenhind}

A \emph{multiplicity class} at $a\in X$ is a morphism of sheaves of sets
\[
\overline N\colon \oshp_{S_aX} \to (\Z_{\geq0})_{S_aX}
\]
such that $\overline N_\theta^{-1}(\Z_{>0}) \subset \oshp_{S_aX,\theta}$ is a finite set for some (hence, any) $\theta\in S_aX$. For $f\in\shp_{S_aX}$, we write for short $\overline N(f) = \overline N([f])$.

\begin{notation}
If $N$ is a multiplicity, we denote by $\overline N$ its class, defined by setting $\overline N(f)=N(h)$ if there exists $h\in N_\theta^{>0}$ such that $[f]=[h]$, and $\overline N(f)=0$  otherwise.
\end{notation}

Let $\shp_{S_aX}'\subset\shp_{S_aX}$ be a representative subsheaf. Then, any multiplicity class $\overline N$ gives a multiplicity $N$ with class $\overline N$, by setting $N(f)=\overline N(f)$ if $f\in\shp_{S_aX}'$, and $N(f)=0$ otherwise.

\begin{definition}\label{def:Eform}
One says that $K\in\dereb_\Rc(\ifield_X)$ has a normal form at $a\in X$ if there exists a multiplicity $N$ such that 
any $\theta\in S_a X$ has an open sectorial neighborhood $V_\theta$ such that
\[
\opb\pi\field_{V_\theta}\tens K \simeq \DSum_{f\in N^{>0}_\theta} \bl\Ex^{\Re f}_{V_\theta|X}\br^{N(f)}.
\]
\end{definition}

Note that the class $\overline N$ of $N$ is uniquely determined by $K$. We call it the multiplicity class of $K$.

\begin{remark}
If $\field=\C$ and $K$ corresponds to a holonomic $\D_X$-module by the Riemann-Hilbert correspondence, this definition corresponds to the classical notion of quasi-normal form (see \cite[\S7.3]{DK16}).
As we deal with Puiseux germs, we do not distinguish here between normal and quasi-normal forms.
\end{remark}

\begin{lemma}
\label{lem:Klocdec}
An object $K\in\dereb_\Rc(\ind\field_X)$ has a normal form at $a\in X$ if and only if for any $\theta\in S_aX$ there exists a finite subset $\Phi_\theta\subset\shp_{S_aX,\theta}$ and integers $n_\theta(f)\in\Z_{> 0}$ for $f\in\Phi_\theta$ such that
\[
\opb\pi\field_{V_\theta}\tens K \simeq \DSum_{f\in\Phi_\theta} \bl\Ex^{\Re f}_{V_\theta|X}\br^{n_\theta(f)}
\]
for some open sectorial neighborhood $V_\theta$ of $\theta$.
\end{lemma}

\begin{proof}
As the ``only if'' part is clear, let us prove the ``if'' part.

Let $\shp_{S_aX}'\subset\shp_{S_aX}$ be a representative subsheaf. 
As the isomorphism class of $\Ex^{\Re f}_{V_\theta|X}$ only depends on $[f]\in\oshp_{S_aX,\theta}$, we can assume that $\Phi_\theta\subset\shp_{S_aX,\theta}'$.
Then, it follows from Corollary~\ref{cor:homEfgI} that $\{\Phi_\theta\}_\theta$ gives a local system $\Phi\subset\shp_{S_aX}'$, and that $\{n_\theta\}_\theta$ gives a morphism of sheaves $n\colon\Phi\to(\Z_{\geq 0})_{S_aX}$.

Consider the multiplicity given by $N(f)=n(f)$ if $f\in\Phi$, and $N(f)=0$ otherwise. Then the isomorphisms in the statement show that $K$ has normal form at $a$ with multiplicity class $\overline N$.
\end{proof}

\begin{proposition}\label{pro:KF}
Let $K\in\dereb_\Rc(\ifield_X)$ have normal form at $a\in X$ with multiplicity class $\overline N$. 
Let $N$ be a multiplicity with class $\overline N$.
Then there exist an open neighborhood $\Omega$ of $a$, and $F\in\dere^{0}_\Rc(\field_X)$ with normal form at $a$ and multiplicity $N$,
such that
\[ 
\opb\pi\field_{\Omega\setminus\{a\}}\tens K
\simeq
\field_X^\enh \ctens F.
\]
\end{proposition}

\begin{proof}
By definition, any $\theta\in S_a X$ has an open sectorial neighborhood $V_\theta$ such that
\begin{equation}
\label{eq:sumK}
\opb\pi\field_{V_\theta}\tens K \isoto \DSum_{f\in N^{>0}_\theta} \bl\Ex^{\Re f}_{V_\theta|X}\br^{N(f)}.
\end{equation}
Let $\Theta\subset S_aX$ be a cyclically ordered finite subset such that
\begin{itemize}
\item [(i)]
$\{a\}\union\Union\nolimits_{\theta\in\Theta}V_\theta$ is a neighborhood of $a$,
\item[(ii)]
for any $\theta,\theta'\in\Theta$ with $\theta\neq\theta'$ the intersection $V_\theta\cap V_{\theta'}$ is non empty if and only if $\theta$ and $\theta'$ are consecutive,
and in this case $V_\theta\cap V_{\theta'}$ is contractible.
\end{itemize}
Let $\theta_1<\cdots<\theta_d<\theta_{d+1}\defeq\theta_1$ be the cyclic ordering of $\Theta$.
Write for short $V_k=V_{\theta_k}$.
Set $\Omega= \{a\}\union \bl\Union\nolimits_{k=1}^d V_k\br$,
and denote by  $j_a\colon\Omega\setminus\{a\}\to X$ and $j_k\colon V_k\to X$ the embeddings.
Set
\begin{align*}
F'_k &= \DSum_{f\in N^{>0}_{\theta_k}} \bl\ex^{\Re f}_{V_k|V_k}\br^{N(f)}
&&\in\dere_\Rc^0(\field_{V_k}).
\end{align*}
The isomorphism \eqref{eq:sumK} induces an isomorphism
\[
u^\enh_k \colon \opb\pi\field_{V_k}\tens K \isoto \field_X^\enh \ctens \Eeim{{j_k}}F'.
\]
Set $u_{kl}^\enh = \opb\pi\field_{V_k\cap V_l} \tens (u^\enh_k\circ(u^\enh_l)^{-1})$.
By Lemma~\ref{lem:homEfgI}~(ii), each $u_{kl}^\enh$ induces an isomorphism $u_{kl}\colon F'_l|_{V_k\cap V_l} \isoto F'_k|_{V_k\cap V_l}$. The isomorphisms $u_{kl}$'s satisfy the usual cocycle condition $u_{kl} \circ u_{lm} = u_{km}$, since so do the isomorphisms $u^\enh_{kl}$'s. Hence, they patch the $F'_k$'s to an object $F'\in\dere_\Rc^0(\field_{\Omega\setminus\{a\}})$.

This proves the statement, with $F \defeq \Eeim{{j_a}}(F')$.
\end{proof}

\section{Stokes filtered local systems}\label{sec:FStokes}

In this section, $X$ denotes a complex analytic curve.

\medskip
To an enhanced ind-sheaf $K$ on $X$ with normal form at $a\in X$, we attach a Stokes filtered local system on $S_aX$. 
If $K=\solE[X](\shm)$ for $\shm$ a meromorphic connection, this coincides with the classical construction in \cite{Del07,Mal91} (see \cite[\S1 and \S8]{Kas16} for a detailed explanation).
Then, we introduce the multiplicity test functor, which allows to detect the multiplicities of exponential factors.

\subsection{Stokes filtrations}

Stokes filtered local systems were introduced in \cite{Del07,Mal91}, and we 
refer to \cite{Sab13} for an exposition.

\medskip

Let $a\in X$ and $L$ a local system of finite rank on $S_a X$.

\begin{definition}
\begin{itemize}
\item[(i)] A pre-Stokes filtration $\F_{\eprec \bullet}L$ on $L$ is the data of a
subsheaf $\F_{\eprec f}L\subset L|_I$ for any open subset $I\subset S_aX$ and any $f\in\shp_{S_aX}(I)$, such that for any $\theta\in S_aX$ and any $f,h\in\shp_{S_aX,\theta}$
with $f \aleq\theta h$, one has $(\F_{\eprec f}L)_\theta\subset(\F_{\eprec 
h}L)_\theta$.
\item[(ii)]
A pre-Stokes filtration $\F_{\eprec \bullet}L$ on $L$ is a Stokes filtration if
there exists a multiplicity $N\colon \shp_{S_aX} \to (\Z_{\geq0})_{S_aX}$ at $a$ such that for any $\theta\in S_aX$ there are an open neighborhood $I\subset S_aX$, and an isomorphism
\begin{equation}
\label{eq:LI}
L|_I \simeq \DSum_{h\in N^{>0}_\theta}\field_I^{N(h)}, 
\end{equation}
inducing for any $\eta\in I$ and $f\in \shp_{S_aX,\eta}$ an isomorphism
\[
(\F_{\eprec f} L)_\eta \simeq \DSum_{h\in N^{>0}_\theta,\ h\aleq\eta f} \field^{N(h)}.
\] 
\end{itemize}
\end{definition}

Note that $\F_{\eprec f}L$ only depends on the class of $f$ in $\oshp_{S_aX}(I)$.

\begin{notation}
Let $\F_{\eprec \bullet}L$ be a Stokes filtration  on $L$.
\begin{itemize}
\item[(i)] For $f\in\shp_{S_aX}(I)$, let $\F_{\prec f}L \subset \F_{\eprec f}L$ be the
subsheaf such that, for any $\theta\in I$, one has
\[
(\F_{\prec f}L)_\theta = \sum_{h\in \shp_{S_aX,\theta},\ h \aless\theta f}(\F_{\eprec h}L)_\theta.
\]
Note that this defines indeed a subsheaf since, under \eqref{eq:LI}, one has
\[
(\F_{\prec f}L)_\eta = \DSum_{h\in N_\theta^{>0},\ h \aless\eta f}\field^{N(h)}.
\]
\item[(ii)] For $f\in\shp_{S_aX}(I)$, set
\[
\Gr_fL \defeq \F_{\eprec f}L/\F_{\prec  f}L.
\]
\end{itemize}
\end{notation}

Note that $\F_{\prec f}L$ and $\Gr_fL$ only depend on the class of $f$ in $\oshp_{S_aX}(I)$.

Note also that \eqref{eq:LI} implies $\Gr_fL|_I\simeq\field_I^{\overline N(f)}$, and in particular $\Gr_fL$ is a locally constant sheaf.

\subsection{Enhanced nearby cycles}

Here, to an enhanced ind-sheaf $K$ on $X$ with normal form at $a\in X$, we associate a Stokes filtered local system on $S_aX$.

\medskip

For $a\in X$ and $I\subset S_a X$ an open subset, consider the commutative diagram
\[
\xymatrix@R=1ex{
I \ar@{^(->}[r]^-{j_I} &S_a X \ar@{^(->}[r]^-{\ti_a} & \widetilde X_a \ar[dd]_{\varpi_a} & \\
&&& X\setminus\{a\}. \ar@{_(->}[dl]^-{j_a} \ar@{_(->}[ul]_-{\tj_a} \\
&& X
}
\]
Recall Definition~\ref{def:nuhom}.

\begin{definition}
Let $a\in X$, $I\subset S_a X$ an open subset, $f\in\shp_{S_aX}(I)$, and $K\in\dereb_\Rc(\ifield_X)$. Set
\begin{align*}
\Psi_a(K) 
&\defeq \opb{\ti_a} \roim{{\tj_a}} \opb j_a \fhom(\field^\enh_X,K) &&\in \BDC(\field_{S_aX}), \\
\F_{\eprec f}\Psi_a(K) &\defeq
\opb{j_I}\nuhom[a](\Ex^{\Re f}_{U| X}, K)  &&\in \BDC(\field_I),
\end{align*}
with $U\subset X$ a sectorial neighborhood of $I$ where $f$ is defined.
\end{definition}

Note that
\[
\opb{\tj_a}\fhom(\Eopb\varpi_a\Ex^{\Re f}_{U| X},\Eepb\varpi_a K) \simeq
\opb{j_a} \fhom(\field^\enh_X,K).
\]
Hence, the adjunction morphism $\id\to\roim{{\tj_a}} \opb{\tj_a}$ induces a morphism
\begin{equation}\label{eq:FeprectoL}
\F_{\eprec f}\Psi_a(K) \to \opb{j_I}\Psi_a(K).
\end{equation}

\begin{proposition}\label{pro:Fpsi}
Let $K\in\dereb_\Rc(\ifield_X)$ have normal form at $a\in X$ with multiplicity class $\overline N$.
Then, with the above notations,
\begin{itemize}
\item [(i)] $\Psi_a(K)$ is concentrated in degree zero, and is a local system of finite rank on $S_aX$;
\item [(ii)] $\F_{\eprec f}\Psi_a(K)$ is concentrated in degree zero, and is an $\R$-constructible sheaf on $I$;
\item [(iii)] the morphism \eqref{eq:FeprectoL} is a monomorphism;
\item[(iv)] $\F_{\eprec \bullet}\Psi_a(K)$ is a Stokes filtration of $\Psi_a(K)$.
\end{itemize}
In particular, $\Gr_f\Psi_a(K)$ is a local system on $I$ of rank $\overline N(f)$.
\end{proposition}

By the Definition~\ref{def:Eform}, this follows from the next lemma.

\begin{lemma}
\label{lem:Fpsifg}
Let $f,h\in\shp_{S_aX}(I)$ for $I\subset S_a X$ open,
and let $U$ be an open sectorial neighborhood of $I$ where $f$ and $h$ are defined.
Then one has
\begin{itemize}
\item[(i)] $\Psi_a(\Ex^{\Re h}_{U|X})|_I \simeq \field_I$,
\item[(ii)] $\F_{\eprec f}\Psi_a(\Ex^{\Re h}_{U|X}) \simeq
\field_{\{\theta\in I\semicolon h \aleq\theta f\}}$,
\item[(iii)] $\F_{\prec f}\Psi_a(\Ex^{\Re h}_{U|X}) \simeq
\field_{\{\theta\in I\semicolon h \aless\theta f\}}$,
\item[(iv)] $\Gr_f\Psi_a(\Ex^{\Re h}_{U|X}) \simeq
\begin{cases}
\field_I, & \text{if }[f]=[h],\\
0,&\text{otherwise}.
\end{cases}$
\end{itemize}
\end{lemma}

\begin{proof}
(i) Since $h$ is locally bounded on $U$, one has
\[
\fhom(\field^\enh_X,\Ex^{\Re h}_{U|X})|_U
\simeq \fhom(\field^\enh_U,\field^\enh_U) \simeq \field_U.
\]

\medskip\noindent
(ii) follows from Lemma~\ref{lem:homEfgI} (i).

\medskip\noindent
(iii) and (iv) follow from (ii).
\end{proof}

It was shown in \cite{Del07,Mal91} that there is an equivalence between the category of germs of meromorphic connections with pole at $a$, and the category of finite rank local systems on $S_aX$ endowed with a Stokes filtration.
Through this equivalence, when $\field=\C$, Proposition~\ref{pro:eqStNorm} below corresponds to \cite[Proposition 3.28]{Moc16}.

Let $\Omega\subset X$ be a contractible open neighborhood of $a$.
Let us consider the following conditions for  $K\in\dere^0_\Rc(\ifield_\Omega)$:
\begin{itemize}
\item[(1)] $K|_{\Omega\setminus\{a\}}\simeq e(L)$, for $L$ a local system of finite rank on $\Omega\setminus \{a\}$,
\item[(2)] $K\simeq\opb\pi\field_{\Omega\setminus\{a\}}\tens K$,
\item[(3)] $K$ has normal form at $a$.
\end{itemize}

\begin{proposition} \label{pro:eqStNorm}
The functor $\Psi_a$ induces an equivalence between the full subcategory of $\dere^0_\Rc(\ifield_\Omega)$ whose objects satisfy conditions {\rm (1)--(3)} above, and the category of finite rank local systems on $S_aX$ endowed with a Stokes filtration.
\end{proposition}

\begin{proof}
Let $I\subset S_aX$.
Denote by $S_h$ the constant sheaf $\field_I$ endowed with the Stokes filtration
\[
(\F_{\eprec k} S_h)_\theta = 
\begin{cases}
\field, &\text{if } h\aleq\theta k, \\
0, &\text{otherwise}.
\end{cases}
\]
By Lemma~\ref{lem:Fpsifg} (ii), one has $S_h\simeq\Psi_a(\Ex^{\Re h}_{U| \Omega})$.
Then, for $f,h\in\shp_{S_aX}(I)$, one has
\[
\hom(S_f,S_h) \simeq \opb{j_I}\nuhom[a](\Ex^{\Re f}_{U| \Omega}, \Ex^{\Re h}_{U| \Omega}),
\]
where $U$ is an open sectorial neighborhood of $I$ where $f$ and $h$ are defined.
Hence, the statement is obtained by similar arguments to those in the proof of Lemma~\ref{pro:KF}.
\end{proof}

\subsection{Multiplicity test functor}
Recall the remark in \S\ref{sse:ilim}.

\begin{definition}
Let $(a,\theta,f)$ be a Puiseux germ on $X$.
The multiplicity test functor at $(a,\theta,f)$ is given by
\begin{align*}
\G_{(a,\theta,f)}\colon \dereb_+(\ifield_X) &\to \BDC(\field),\\
K &\mapsto \ilim[V,c,\delta,\varepsilon]
\FHom(\ex_{V|X}^{\range{\Re f+c}{\Re f -\delta|z_a|^{-\varepsilon}}}, K),
\end{align*}
where $z_a$ is a  local coordinate at $a$,
$V$ runs over the open sectorial neighborhoods of $\theta$, $c\to+\infty$, and
$\delta,\varepsilon\to0+$.

This does not depend on the choice of the local coordinate $z_a$.
\end{definition}

Note that one has
\[
\G_{(a,\theta,f)}(K) \simeq
\ilim[c,\delta,\varepsilon\to0+] \nuhom[a]\bl\ex_{U|X}^{\range{\Re f+c}{\Re f -\delta|z_a|^{-\varepsilon}}}, K\br_\theta,
\]
for $U$ an open sectorial neighborhood of $\theta$ where $f$ is defined.

\begin{lemma}\label{lem:Gzetabdd} 
Let $(a,\theta,f)$ be a Puiseux germ on $X$, and 
$K\in\dereb_+(\ifield_X)$.
\begin{itemize}
\item[(i)]
If $K \simeq \Efield_X\ctens K$, then
\begin{align*}
\G_{(a,\theta,f)}(K) 
&\simeq \ilim[V,\delta,\varepsilon]
\FHom(\ex_{V|X}^{\range{\Re f}{\Re f -\delta|z_a|^{-\varepsilon}}}, K) \\
&\simeq \ilim[V,\delta,\varepsilon]
\FHom(\Ex_{V|X}^{\range{\Re f}{\Re f -\delta|z_a|^{-\varepsilon}}}, K).
\end{align*}
\item[(ii)]
If  $h\in\shp_{S_aX,\theta}$ satisfies $h\asim \theta f$,  
then $\G_{(a,\theta,h)}(K) \simeq \G_{(a,\theta,f)}(K)$.
\item[(iii)] 
One has $\G_{(a,\theta,f)}(K) \simeq \G_{(a,\theta,f)}(\Efield_X\ctens K)$.
\end{itemize}
\end{lemma}

\begin{proof}
By \cite[Proposition 4.7.9]{DK16}, one has
\begin{align*}
\ilim[c\to+\infty] \FHom{}&(\ex_{V|X}^{\range{\Re f+c}{\Re f -\delta|z_a|^{-\varepsilon}}}, K) \\
&\simeq \FHom(\Efield_X\ctens \ex_{V|X}^{\range{\Re f}{\Re f -\delta|z_a|^{-\varepsilon}}}, \Efield_X\ctens K) \\
&\simeq \FHom(\ex_{V|X}^{\range{\Re f}{\Re f -\delta|z_a|^{-\varepsilon}}}, \Efield_X\ctens K),
\end{align*}
which implies (i) and (iii).  (ii) is obvious.
\end{proof}

\begin{proposition}\label{pro:GrK}
Let $(a,\theta,f)$ be a Puiseux germ on $X$.
Let $K\in\dereb_\Rc(\ifield_X)$ have normal form at $a$
with multiplicity class $\overline N$.
Then,
\[
\G_{(a,\theta,f)}(K) \simeq
\bigl(\Gr_f\Psi_a(K)\bigr)_\theta \simeq
\field^{\overline N(f)}.
\]
\end{proposition}

\begin{proof}
The second isomorphism follows from Proposition~\ref{pro:Fpsi}.

Then, decomposing $K$ as in Definition~\ref{def:Eform}, and using Lemma~\ref{lem:Gzetabdd}~(i),
the statement follows from Lemma~\ref{lem:GEf12} below.
\end{proof}

\begin{lemma}\label{lem:GEf12}
Let $f,h\in\shp_{S_aX,\theta}$, and let $U$ be an open sectorial neighborhood of $\theta$ where $h$ is defined. Then, one has
\[
\G_{(a,\theta,f)}(\ex^{\Re h}_{U|X}) \simeq
\G_{(a,\theta,f)}(\Ex^{\Re h}_{U|X}) \simeq
\begin{cases}
\field, &\text{if }f \asim\theta h, \\
0, &\text{otherwise.}
\end{cases}
\]
\end{lemma}

\begin{proof}
The first isomorphism follows from Lemma~\ref{lem:Gzetabdd} (iii).
Let us prove the second isomorphism.

Let us set
\begin{align*}
\G^+
&\defeq \nuhom[a]\bl\Ex_{U|X}^{\Re f} ,
\Ex^{\Re h}_{U|X}\br_\theta, \\
\G^-_{\delta,\varepsilon} 
&\defeq \nuhom[a]\bl\Ex_{U|X}^{\Re f -\delta|z_a|^{-\varepsilon}},
\Ex^{\Re h}_{U|X}\br_\theta, \\
\G^- &\defeq \ilim[\delta,\varepsilon\to0+] \G^-_{\delta,\varepsilon},
\end{align*}
so that there is a distinguished triangle in $\BDC(\field)$
\[
\G^-\to \G^+\to \G_{(a,\theta,f)}(\Ex^{\Re h}_{U|X}) \to[+1].
\]
Using a local coordinate $z_a$ at $a$, let $\beta\in\C^\times$ be such that
$\Re(\theta\beta)>0$. Then, after shrinking $U$, there is a constant $C>0$ such that $0 < \Re z_a\beta\leq |\beta|\,|z_a|
\leq C \Re z_a\beta$ on $U$. It follows that
\[
\G^- \simeq \ilim[\delta,\varepsilon\to0+] \G^{\prime-}_{\delta,\varepsilon},
\]
where we set $f_{\delta,\varepsilon}(z) \defeq f(z) -\delta(z_a\beta)^{-\varepsilon}$, and
\[
\G^{\prime-}_{\delta,\varepsilon} 
\defeq \nuhom[a]\bl\Ex_{U|X}^{\Re f_{\delta,\varepsilon}},
\Ex^{\Re h}_{U|X}\br_\theta.
\]

To get the statement, it is then enough to prove
\begin{itemize}
\item[(a)]
$\G^+\simeq\field$ and $\G^-\simeq 0$ if $f \asim\theta h$,
\item[(b)]
$\G^-\isoto \G^+ \simeq \field$ if $h \aless\theta f$,
\item[(c)]
$\G^-\simeq 0 \simeq \G^+$ if $h \not\aleq\theta f$.
\end{itemize}
For this, we are going to use Lemma~\ref{lem:homEfgtheta} (i).

\medskip\noindent
(a) If $f \asim\theta h$, then $f_{\delta,\varepsilon} \aless\theta h$ for any $\delta$ and $\varepsilon$. Thus, $\G^+\simeq \field$ and $\G^{\prime-}_{\delta,\varepsilon}\simeq 0$.

\medskip\noindent
(b) If $h \aless\theta f$, then $h \aless\theta f_{\delta,\varepsilon}$ for any $0<\delta,\varepsilon\ll1$. 
Thus, $\G^+\simeq \field \simeq \G^{\prime-}_{\delta,\varepsilon}$. 

\medskip\noindent
(c) If $h \not\aleq\theta f$, then $h \not\aleq\theta f_{\delta,\varepsilon}$ for any $0<\delta,\varepsilon\ll1$. 
Thus,  $\G^+\simeq 0 \simeq \G^{\prime-}_{\delta,\varepsilon}$.
\end{proof}

\subsection{A vanishing result}

Let $\V$ be a complex affine line with coordinate $z$. 
Let $\PP=\V\cup\{\infty\}$ be its projective compactification.
In this subsection, with respect to the notations in \eqref{eq:IStX},
we consider $X=\PP$ and $a\in\{0,\infty\}$.
Let us take $z_0=z$ and $z_\infty=z^{-1}$ as local coordinate at $0$ and $\infty$, respectively.
We shall write $S_\infty\V$ instead of $S_\infty\PP$.  

Let $U$ be an open sectorial neighborhood of $\theta\in S_a\V$, and
let $\varphi\colon U \to \R$ be a real analytic and globally subanalytic function,
that is, a real analytic function whose graph is subanalytic in $\PP\times\cR$.
We write $\varphi'$ for short, instead of $\partial_z\varphi(z,\overline z)$.

\begin{lemma}\label{lem:phi'rR}
Let $\theta_\circ\in S_a\V$, and $f\in\shp^{(0,+\infty)}_{S_a\V,\theta_\circ}$.
Let $\varphi\colon U \to \R$ be a real analytic and globally subanalytic
function, with $U$ a sectorial open neighborhood of $\theta_\circ$.
Assume one of the following conditions:
\begin{itemize}
\item[(a)]
$a=\infty$, $\ord_\infty(f)< 1$, and $|\varphi'|\geq r$ at $\theta_\circ$ for some $r>0$; 
\item[(b)]
$a=\infty$, $\ord_\infty(f)>1$, and $|\varphi'|\leq R$ at $\theta_\circ$ for some $R>0$;
\item[(c)]
$a=0$, and $|\varphi'|\leq R$ at $\theta_\circ$ for some $R>0$.
\end{itemize}
Then, for a generic $\theta$ near $\theta_\circ$, one has
$\G_{(a,\theta,f)}(\ex_{U|\PP}^\varphi) \simeq 0$.
\end{lemma}

We need a preliminary result.

Write $z=s\zeta$, with $s=|z|$ and $\zeta\in\C$ with $|\zeta|=1$.

If $a=\infty$, let us identify $S_\infty\V$ with $\{\zeta\in\C \semicolon |\zeta|=1\}$,
so that $\R_{>0}\times S_\infty\V$ is an open subset of $\widetilde\PP_\infty$.
After shrinking $U\dotowns\theta_\circ$, for $z\in U$ one has
\begin{equation}
\label{eq:phipsimu}
\varphi(z,\overline z)
= s^\mu \widetilde\varphi(s^{-1/p},\zeta), 
\end{equation}
where $p\in\Z_{\geq 1}$, $\mu\in\frac1p\Z_{>0}$, and
$\widetilde\varphi$ is a real valued real analytic function in a connected neighborhood of $(0,\theta_\circ)\in\R\times S_\infty\V$.
If $\widetilde\varphi_0(\zeta)\defeq \widetilde\varphi(0,\zeta)\not\equiv 0$, we say that $\varphi$ has pole order $\mu$.

If $a=0$, we identify instead $S_0\V$ with $\{\zeta\in\C \semicolon |\zeta|=1\}$.
In this case, we have $\varphi(z,\overline z) = s^{-\mu} \widetilde\varphi(s^{1/p},\zeta)$.
If $\widetilde\varphi_0(\zeta)\not\equiv 0$, we say that $\varphi$ has pole order $\mu$.

\begin{sublemma}\label{sublem:phi'rR}
Let $\theta\in S_a\V$, and
let $\varphi$ be a real analytic and globally subanalytic function at
$\theta$, with pole order $\mu\in\Q$.
\begin{itemize}
\item[(i)]
If $a=\infty$ and $|\varphi'|\geq r$ at $\theta$ for some $r>0$, then $\mu\geq 1$.
\item[(ii)]
If $a=\infty$ and $|\varphi'|\leq R$ at $\theta$ for some $R>0$, then $\mu\leq 1$.
\item[(ii)]
If $a=0$ and $|\varphi'|\leq R$ at $\theta$ for some $R>0$, then $\mu\leq 0$.
\end{itemize}
\end{sublemma}

\begin{proof}
Since the proofs are similar, let us only consider the case $a=\infty$.

One has $\varphi' = \partial_z\varphi = \frac1{2\zeta}\partial_s\varphi + \frac1{2s}\partial_\zeta\varphi$, which is the decomposition into real and imaginary parts.
Then, by \eqref{eq:phipsimu},
\[
|\varphi'|^2 = \left|\tfrac\mu{2\zeta} s^{\mu-1}\widetilde\varphi - \tfrac1{2
p}s^{\mu-1-\frac1p}\partial_1\widetilde\varphi\right|^2 + \left|\tfrac12 s^{\mu-1}
\partial_\zeta\widetilde\varphi \right|^2,
\]
where $\partial_1\widetilde\varphi(u,\zeta)$ denotes the derivative with respect to the first variable.

Thus $|\varphi'| = \mathrm{O}(s^{\mu-1})$ at $\theta$, and the statement follows.
\end{proof}

\begin{proof}[Proof of Lemma~\ref{lem:phi'rR}]
As we are considering the case $a=\infty$, we have to prove the statement under either of the assumptions (a) or (b).

If $\varphi\equiv0$ at $\theta_\circ$, the statement hold by Lemma~\ref{lem:GEf12}.

Otherwise, with notations \eqref{eq:phipsimu}, for a generic $\theta$ near $\theta_\circ$, we may assume that $\widetilde\varphi_0(\theta)\neq0$ and 
\[
\varphi(z,\overline z) = \widetilde\varphi_0(\zeta)s^\mu\bl 1+\mathrm{O}(s^{-1/p})\br \quad
\text{for }s\to+\infty.
\]

After replacing $p$ by one of its multiples, we can assume $\lambda\in\frac1p\Z_{>0}$ and
$f(z)=c z^\lambda\bl 1+\mathrm{O}(z^{-1/p})\br$ for $z\to\infty$, with $c\in\C^\times$.
Then, $\Re f(z) = s^\lambda \widetilde\psi(s^{-1/p},\zeta)$, 
where $\widetilde\psi$ is a real valued real analytic function in a connected neighborhood of $(0,\theta_\circ)\in\R\times
S_\infty\V$ with $\widetilde\psi_0(\zeta)\defeq\widetilde\psi(0,\zeta)\not\equiv0$ at $\theta_\circ$. It follows that, for a generic $\theta$ near $\theta_\circ$, one has $\widetilde\psi_0(\theta)\neq0$ and 
\[
\Re f(z) = \widetilde\psi_0(\zeta)s^\lambda\bl 1+\mathrm{O}(s^{-1/p})\br \quad
\text{for }s\to+\infty.
\]

We will use arguments similar to those in the proof of
Lemma~\ref{lem:GEf12}. 
There is a distinguished triangle in $\BDC(\field)$
\[
\G^- \to \G^+ \to \G_{(\infty,\theta,f)}(\ex_{U|\PP}^\varphi) \to[+1],
\]
where we set
\begin{align*}
\G^+_c &\defeq \nuhom[a]\bl\ex_{U|\PP}^{\Re f+c} ,
\ex^{\varphi}_{U|\PP}\br_\theta, \\
\G^-_{\delta,\varepsilon} &\defeq \nuhom[a]\bl\ex_{U|\PP}^{\Re f
-\delta|z|^{\varepsilon}}, \ex^{\varphi}_{U|\PP}\br_\theta,
\end{align*}
and
\[
\G^+ \defeq \ilim[c\to+\infty] \G^+_c, \qquad
\G^- \defeq \ilim[\delta,\varepsilon\to0+] \G^-_{\delta,\varepsilon}.
\]

Note that $z\to\infty$ in $U$ is equivalent to $s\to+\infty$.

\medskip\noindent(a)
We have $0<\lambda<1\leq\mu$, where the last inequality follows from Sublemma~\ref{sublem:phi'rR}.

\medskip\noindent (a-1)
Assume that $\widetilde\varphi_0(\theta)>0$. Then, after shrinking $U$, for some $C>0$ one has
\[
\begin{cases}
\Re(\varphi-f-c) \geq C s^\mu, \\
\Re(\varphi-f+\delta |z|^{\varepsilon}) \geq C s^\mu,
\end{cases}
\]
when $s\to+\infty$.
Thus, both $\G^-_{\delta,\varepsilon}$ and $\G^+_c$ satisfy hypothesis (iii) of
Lemma~\ref{lem:Ephi12}. Hence $\G^-\simeq 0 \simeq \G^+$.

\medskip\noindent (a-2)
Assume that $\widetilde\varphi_0(\theta)<0$. Then, after shrinking $U$, for some $C>0$ one has
\[
\begin{cases}
\Re(\varphi-f-c) \leq
-C s^\mu, \\
\Re(\varphi-f+\delta |z|^{\varepsilon}) \leq
-C s^\mu,
\end{cases}
\]
when $s\to \infty$.
Thus, hypothesis (ii) of Lemma~\ref{lem:Ephi12} is satisfied for
$\G^-_{\delta,\varepsilon}$, for $\G^+_c$.
Hence $\G^-\isoto \G^+ \simeq \field$.

\medskip\noindent(b)
In this case we have $\mu\leq 1<\lambda$, where the first inequality
follows from Sublemma~\ref{sublem:phi'rR}.

\medskip\noindent (b-1)
Assume that $\widetilde\psi_0(\theta)<0$. Then, after shrinking $U$, for some $C>0$ one has
\[
\begin{cases}
\Re(\varphi-f-c) \geq C s^\lambda, \\
\Re(\varphi-f+\delta |z|^{\varepsilon}) \geq C s^\lambda,
\end{cases}
\]
when $s\to +\infty$.
Thus, both $\G^-_{\delta,\varepsilon}$ and $\G^+_c$ satisfy hypothesis (iii) of
Lemma~\ref{lem:Ephi12}. Hence $\G^-\simeq 0 \simeq \G^+$.

\medskip\noindent (b-2)
Assume that $\widetilde\psi_0(\theta)>0$. Then, after shrinking $U$, for some $C>0$ one has
\[
\begin{cases}
\Re(\varphi-f-c) \leq -C s^\lambda, \\
\Re(\varphi-f+\delta |z|^{\varepsilon}) \leq -C s^\lambda,
\end{cases}
\]
when $s\to +\infty$.
Thus, hypothesis (ii) of Lemma~\ref{lem:Ephi12} is satisfied for
$\G^-_{\delta,\varepsilon}$, for $\G^+_c$.
Hence $\G^-\isoto \G^+ \simeq \field$.
\end{proof}

\section{Stationary phase formula}\label{sec:statement}
After recalling in some details the Legendre transform,
we state here the stationary phase formula in terms of enhanced ind-sheaves.

\subsection{Fourier-Laplace and enhanced Fourier-Sato  transforms}

Let $\V$ be a one-dimensional complex vector space, with coordinate $z$, 
and let $\W$ be its dual, with dual coordinate $w$. 
Let $\PP=\V\cup\{\infty\}$ and $\bb=\W\cup\{\infty\}$ be the associated projective lines,
and consider the bordered spaces $\V_\infty = (\V,\PP)$, $\W_\infty = (\W,\bb)$.
Consider the morphisms
\[
\xymatrix{\V_\infty & \V_\infty\times \W_\infty \ar[l]_-p \ar[r]^-q & \W_\infty}
\]
induced by the projections.

The Fourier-Laplace transform for $\D$-modules is defined as follows.
For $\shm\in\BDC(\D_{\V_\infty})$ and $\shn\in\BDC(\D_{\W_\infty})$, set
\begin{align*}
\lap \shm &= \doim q(\she^{-zw}_{\V\times\W|\V_\infty\times\W_\infty}\dtens\dopb p \shm)
&&\in\BDC(\D_{\W_\infty}), \\
\lapa \shn &= \doim p(\she^{zw}_{\V\times\W|\V_\infty\times\W_\infty}\dtens\dopb q \shn)
&&\in\BDC(\D_{\V_\infty}).
\end{align*}
Note that $\BDC_\hol(\D_{\V_\infty})$ is equivalent to the bounded derived category of algebraic $\D_\V$-modules with holonomic cohomologies. Then (see \cite{KL85}), the above functors are compatible with the Fourier transform at the level of the Weyl algebra, given by the isomorphism $ \C[z,\partial_z] \simeq \C[w,\partial_w]$, $z\leftrightarrow-\partial_w$, $\partial_z\leftrightarrow w$.
In particular, $\Lap$ and $\Lapa$ are quasi-inverse of each other, and interchange $\Mod_\hol(\D_{\V_\infty})$ and $\Mod_\hol(\D_{\W_\infty})$.

The Fourier-Sato transform for enhanced sheaves was introduced and studied in \cite{Tam08} (see also \cite{DAg14,KS16L}). It extends to enhanced ind-sheaves as follows.
For $K\in\dereb_+(\ifield_{\V_\infty})$ and $P\in\dereb_+(\ifield_{\W_\infty})$, set
\begin{align*}
\lap K &= \Eeeim q(\ex^{-\Re zw}_{\V\times\W|\V_\infty\times\W_\infty}[1]\ctens\Eopb p K)
&&\in\dereb_+(\ifield_{\W_\infty}), \\
\lapa P &= \Eeeim p(\ex^{\Re zw}_{\V\times\W|\V_\infty\times\W_\infty}[1]\ctens\Eopb q P)
&&\in\dereb_+(\ifield_{\V_\infty}).
\end{align*}
The functors $\Lap$ and $\Lapa$ are quasi-inverse of each other and, since $p$ and $q$ are semiproper, they interchange $\dereb_\Rc(\ifield_{\V_\infty})$ and $\dereb_\Rc(\ifield_{\W_\infty})$.

Note also that $\Lap$ and $\Lapa$ interchange $\dereb_+(\field_\V)$ and $\dereb_+(\field_{\W})$, as well as $\dereb_\Rc(\field_{\V_\infty})$ and $\dereb_\Rc(\field_{\W_\infty})$.

Recall that the Riemann-Hilbert correspondence of \cite{DK16} provides a fully faithful embedding
\[
\xymatrix{
\solE[\V_\infty]\colon\BDC_\hol(\D_{\V_\infty}) \ar@{ >->}[r] & 
\dereb_\Rc(\ind\C_{\V_\infty}).
}
\]
Note that, if $K\simeq\field^\enh\ctens K$, then
\[
\lap K \simeq \Eeeim q(\Ex^{-\Re zw}_{\V\times\W|\V_\infty\times\W_\infty}[1]\ctens\Eopb p K),
\]
and similarly for $\Lapa$.
Since, for $\field=\C$,
\[
\solE[\V\times\W](\she^{\pm\Re zw}_{\V\times\W|\V_\infty\times\W_\infty}) 
\simeq \Ex^{\pm\Re zw}_{\V\times\W|\V_\infty\times\W_\infty},
\]
the next proposition immediately follows from the functoriality of $\solE[]$.
This was first observed in \cite{KS16L}, where the case of non holonomic $\D$-modules was also discussed.

\begin{proposition}\label{pro:Fou}
For $\shm\in\Mod_\hol(\D_{\V_\infty})$ and $\shn\in\Mod_\hol(\D_{\W_\infty})$ one has
\[
\solE[\W](\lap\shm) \simeq \lap\solE[\V](\shm), \quad
\solE[\V](\lapa\shn) \simeq \lapa\solE[\W](\shn).
\]
\end{proposition}

The next lemma will be of use later.

For $a\in\V$, let $\tau_a\colon\V_\infty\to\V_\infty$ be the morphism induced by the translation $\tau_a(z)=z+a$.

\begin{lemma}\label{lem:Foutrans}
For $K\in\dere_+(\ifield_{\V_\infty})$ one has
\[
\lap(\Eopb{{\tau_a}} K) \simeq \ex^{\Re aw}_{\W|\W_\infty}\ctens \lap K.
\]
\end{lemma}

\begin{proof}
Note that $\Eopb{\tau_a} K \simeq \Eoim{(\tau_{-a})} K \simeq \epsilon(\field_{\{z=-a\}})\cconv K$, where
the convolution functor is defined by
$K_1\cconv K_2 = \Eeeim m (K_1 \cetens K_2)$, for $m(z_1,z_2)=z_1+z_2$.
Note also that $\lap \epsilon(\field_{\{z=-a\}})\simeq\ex^{\Re aw}_{\W|\W_\infty}$.
Then one has
\begin{align*}
\lap(\Eopb{\tau_a} K) 
&\simeq \lap\bl \epsilon(\field_{\{z=-a\}})\cconv K \br \\
&\simeq \lap \epsilon(\field_{\{z=-a\}})\ctens \lap K  \\
&\simeq \ex^{\Re aw}_{\W|\W_\infty}\ctens \lap K.
\end{align*}
\end{proof}

\subsection{The case of enhanced sheaves}\label{sse:PhiL}
Consider the projections
\[
\V\times\R \from[\overline p] \V\times\W\times\R^2 \to[\overline q] \W\times\R,
\]
given by $\overline p(z,w,t,s)=(z,t)$, $\overline q(z,w,t,s)=(w,s)$.
For $F\in\BDC(\field_{\V\times\R})$ and $G\in\BDC(\field_{\W\times\R})$, set
\begin{align*}
\Phi_\Lap(F) &\defeq \reim{\overline q}( \field_{\{s-t-\Re zw \geq 0\}}[1] \tens \opb{\overline p}F )
&&\in\BDC(\field_{\W\times\R}), \\
\Phi_{\Lapa}(G) &\defeq \reim{\overline p}( \field_{\{t-s+\Re zw \geq 0\}}[1] \tens \opb{\overline q}G)
&&\in\BDC(\field_{\V\times\R}) .
\end{align*}

\begin{lemma}\label{lem:TfouFou}
For $F\in\BDC(\field_{\V\times\R})$ and $G\in\BDC(\field_{\W\times\R})$, one has
\[
\lap(\quot F) \simeq \quot\Phi_\Lap(F), \quad
\lapa(\quot G) \simeq \quot\Phi_{\Lapa}(G).
\]
\end{lemma}

\begin{proof}
Since the proofs are similar, let us only consider the first isomorphism.
Consider the maps $\R^2\to\R$ given by $p'(t,s) = t$, $q'(t,s) = s$, $\sigma'(t,s) = s-t$, and use the same notations for the associated maps
\[
\V\times\W\times\R^2\to[p',q',\sigma']\V\times\W\times\R.
\]
Recall that we set $p_\R = p\times\id_\R$ and $q_\R = q\times\id_\R$.
Then, one has
\begin{align*}
\lap(\quot F) 
&= \Eeim{q}(\ex^{-\Re zw}_{\V\times\W}[1]\ctens\Eopb{p}\quot F) \\
&= \quot\reim{{q_\R}}(\field_{\{t-\Re zw\geq 0\}}[1]\ctens\opb{p_\R} F) \\
&\underset{(*)}\simeq \quot \reim{{q_\R}}\reim{q'} ( \opb{\sigma'}\field_{\{t-\Re zw \geq 0\}}[1] \tens \opb{p'}\opb{p_\R}F ) \\
&\simeq \quot \reim{\overline q} ( \field_{\{s-t-\Re zw \geq 0\}}[1] \tens \opb{\overline p}F ),
\end{align*}
where $(*)$ follows from \cite[Lemma 4.1.4]{DK16}.
\end{proof}

\subsection{Microsupport and enhanced Fourier-Sato transform}\label{sse:SSenL}

Consider the symplectic coordinates 
\[
(z,t,w,s;z^*,t^*,w^*,s^*)\in T^*(\V\times\R\times\W\times\R).
\] 
Recalling the definitions of $\Phi_\Lap$ and $\Phi_{\Lapa}$ from \S\ref{sse:PhiL}, consider the Lagrangians
\begin{align*}
\Lambda_\Lap &= \{s^*>0\}\cap SS( \field_{\{s-t-\Re zw \geq 0\}}[1] ),\\
\Lambda_{\Lapa} &= \{t^*>0\}\cap SS( \field_{\{t-s+\Re zw \geq 0\}}[1] ).
\end{align*}

Let us concentrate first on $\Lambda_\Lap$.
With the notations of \S\ref{sse:SS}, one has
\[
\Lambda_\Lap = \{ s-t-\Re zw = 0,\ z^*= wt^*,\ w^* = zt^*,\ s^*=-t^*,\ s^*>0 \}.
\]
Let $\Lambda^a_\Lap$ be the image of $\Lambda_\Lap$ by the endomorphism of $T^*(\V\times\R\times\W\times\R)$ changing the sign of $z^*$ and $t^*$. 
Then $\Lambda^a_\Lap$ is the graph of the
homogeneous symplectic transformation
\begin{align*}
\overline\chi\colon T^*_{\{t^*>0\}}(\V\times\R) &\to T^*_{\{s^*>0\}}(\W\times\R) \\
(z,t;z^*,t^*) &\mapsto (z^*/t^*, t+\Re zz^*/t^* ; -zt^* , t^*).
\end{align*}

Recall the notations in \S\ref{sse:SS}.
Then $\overline\chi$ induces by $\gamma$ the contact transformation
\[
\chi\colon (T^*\V)\times \R \to (T^*\W)\times\R ,\quad ((z;z^*),t) \mapsto
((z^*;-z),t+\Re zz^*)
\]
underlying the Legendre transform. In turn, this induces by $\rho$
the symplectic transformation
\[
\chi_\rho\colon T^*\V \to T^*\W ,\quad (z,z^*) \mapsto
(z^*,-z)
\]
considered in \eqref{eq:introchi}.

Similar considerations hold for  $\Lambda_{\Lapa}$, interchanging the roles of $\V$ and $\W$, and replacing $\overline\chi$, $\chi$, and $\chi_\rho$ with their respective inverses.
Note the latter are explicitly given by
\begin{align*}
\chi^{-1}&\colon ((w;w^*),t) \mapsto ((-w^*;w),t+\Re ww^*), \\
\chi_\rho^{-1}&\colon (w,w^*) \mapsto (-w^*,w).
\end{align*}

\begin{proposition}[{\cite[Theorem 3.6]{Tam08}}]\label{thm:Tam}
For $F\in\dere_+(\field_\V)$ and $G\in\dere_+(\field_{\W})$, one has
\begin{align*}
\SSE(\lap F) &= \chi\bl \SSE(F) \br, &
\SSE(\lapa G) &= \chi^{-1}\bl \SSE(G) \br, \\ 
\SSEo(\lap F) &= \chi_\rho\bl \SSEo(F) \br, &
\SSEo(\lapa G) &= \chi_\rho^{-1}\bl \SSEo(G) \br.
\end{align*}
\end{proposition}

\subsection{Legendre transform}\label{sse:Legendre}

Recall the notion of Puiseux germ from Definition~\ref{def:Puis}. Recall that we write $S_\infty\V$ instead of $S_\infty\PP$.  

Let $a\in\PP$, $\theta\in S_a\V$ and $f\in\shp_{S_a\V,\theta}$.

\begin{definition}\label{def:admissible}
Let us say that the Puiseux germ $(a,\theta,f)$ is:
\begin{itemize}
\item[(i)] \emph{unbounded} if $\ord_a(f)>0$;
\item[(ii)] \emph{linear} if $a=\infty$ and $f(z)-bz\in\shp_{S_\infty\V}^{\leq 0}$ for some $b\in\W$;
\item[(iii)] \emph{admissible} if it is unbounded and \emph{not} linear.	
\end{itemize}
\end{definition}

Let $\tau\in\C$ be the coordinate, and consider the complex version
\[
\chi_\C\colon (T^*\V)\times \C \to (T^*\W)\times\C ,\quad (z,w,\tau) \mapsto
(w,-z,\tau+zw),
\]
of the contact transformation $\chi$ from \S\ref{sse:SSenL}. 
We have a commutative diagram
\begin{equation}
\label{eq:chiCR}
\xymatrix{
(T^*\V)\times \C \ar[r]^-{\chi_\C} \ar[d]^\Re & (T^*\W)\times\C \ar[d]^\Re \\
(T^*\V)\times \R \ar[r]^-{\chi} & (T^*\W)\times\R ,}
\end{equation}
where the map $\Re$ is induced by $\C\owns\tau\mapsto \Re\tau\in\R$.

To the Puiseux germ $(a,\theta,f)$ on $\V$, one associates the germ of Lagrangian\footnote{On a contact manifold, Lagrangians are also called Legendrians in the literature.} submanifold $\Lambda_\C^{(a,\theta,f)} \subset (T^*\V)\times \C$ defined by
\begin{equation}
\label{eq:LambdaC}
\Lambda_\C^{(a,\theta,f)} = \{(z,w,\tau)\semicolon \tau=-f(z),\ w=f'(z) \},
\end{equation}
for $z$ near $\theta$.
Note in particular that, with notations as in \S\ref{sse:Airy}, one has
\[
\rho_\C\bl \Lambda_\C^{(a,\theta,f)} \br = C_{(a,\theta,f)},
\]
where $\rho_\C\colon (T^*\V)\times \C \to T^*\V$ is the projection.

\begin{lemmadefinition}\label{lemdef}
If $(a,\theta,f)$ is an admissible Puiseux germ on $\V$, then there exists a unique admissible Puiseux germ $(b,\eta,g)$ on $\W$, such that 
\begin{equation}\label{eq:Lfg}
\chi_\C(\Lambda_\C^{(a,\theta,f)}) = \Lambda_\C^{(b,\eta,g)}.
\end{equation}
We set $\Lap(a,\theta,f) \defeq (b,\eta,g)$, and call it the \emph{Legendre transform} of $(a,\theta,f)$.
\end{lemmadefinition}

\begin{proof}
Since $(a,\theta,f)$ is admissible, $w=f'(z)$ is invertible for $z$ near $\theta$.
Let $b\in\bb$ and $\eta\in S_b\W$ be such that $w=f'(z)\to\eta$ when $z\to\theta$.
(See Lemma~\ref{lem:Leg} below for an explicit computation of $b$ and $\eta$.)
Let $z=\psi(w)$ be the inverse of $w=f'(z)$ for $z$ near $\theta$, and denote by $g(w)$ a primitive of $-\psi(w)$.
Thus $g$ satisfies the relations
\begin{equation}\label{eq:wfzg}
\begin{cases}
w = f'(z), \\
z = -g'(w),
\end{cases}
\end{equation}
which are equivalent to $d(zw-f(z)+g(w))=0$.
Choosing the only primitive $g$ which satisfies
\begin{equation}\label{eq:wfzg'}
zw-f(z)+g(w)=0
\end{equation}
provides the solution to \eqref{eq:Lfg}. In other words, we obtained
\[
g(w) = f(z) - z f'(z),
\quad \text{for }w=f'(z). 	
\]
We shall prove the admissibility of $(b,\eta,g)$ in Lemma~\ref{lem:Leg}, by explicit calculation of $\ord_b(g)$.
\end{proof}

With notations as in the above statement, we give in Lemma~\ref{lem:Leg} below an explicit computation of $b$, $\eta$, and of the pole order of the difference between $g$ and its linear part.

\begin{convention}\label{con:coo}
\begin{itemize}
\item[(i)]
We fix $z_a=z-a$ as  local coordinate at $a\in\V$, and
$z_\infty=z^{-1}$ as  local coordinate at $\infty$.
\item[(ii)]
For $a\in\PP$, $\theta\in S_a\V$, and $\lambda\in\Q$,
we fix a determination of $z_a^{-\lambda}$ at $\theta$. 
\end{itemize}
\end{convention}

\begin{notation}
For $f\in\shp^\lambda_{S_a\V,\theta}$, denote by $\sigma_\lambda(f)\in\C^\times$
the coefficient of $z_a^{-\lambda}$ in its Puiseux expansion (which depends on
the choices in Convention~\ref{con:coo}).
\end{notation}

Set
$\V^\times=\V\setminus\{0\}$. For $\al\in\V^\times$, we write
\[
a+\al0 \defeq \lim\limits_{t\to0+}\tj_a(a+t\al)\in S_a\V, \quad
\alpha\infty \defeq \lim\limits_{t\to+\infty}\tj_\infty(t\alpha) \in S_\infty \V.
\]
We will write for short $(a+\al0,f)$ and $(\alpha\infty,f)$ instead of $(a,a+\al0,f)$ and $(\infty,\alpha\infty,f)$, respectively.

We regard a Puiseux germ $(a,\theta,f)$ as a point of the \'etal\'e space 
\[
\etale{\shp_{S_a\V}} = \DUnion_{\theta\in S_a\V}\shp_{S_a\V,\theta},
\]
endowed with its natural topology.

\begin{lemma}\label{lem:Leg}
Let $a\in\V$ and $b\in\W$.
At the level of \'etal\'e spaces, the Legendre transform gives the 
homeomorphisms
\[
\Lap\colon
\begin{cases}
\etale{\shp_{S_a\V}^{(0,+\infty)}} \isoto \etale{-aw+\shp^{(0,1)}_{S_\infty\W}},
\\[1.5ex]
\etale{bz+\shp_{S_\infty\V}^{(0,1)}} \isoto \etale{\shp^{(0,+\infty)}_{S_b\W}},
\\[1.5ex]
\etale{\shp_{S_\infty\V}^{(1,+\infty)}} \isoto
\etale{\shp^{(1,+\infty)}_{S_\infty\W}}.
\end{cases}
\]
More precisely, one has:
\begin{itemize}
\item[(i)]
for $\lambda\in\Q_{>0}$,
\[
\Lap\colon \etale{\shp_{S_a\V}^\lambda} \isoto
\etale{-aw+\shp^{\lambda/(\lambda+1)}_{S_\infty\W}},\quad
(a+\alpha0,f)\mapsto(\beta\infty,g),
\] 
with $\beta=-\lambda\alpha^{-\lambda-1}\sigma_\lambda(f)$.
\item[(ii)]
for $\lambda\in\Q\cap(0,1)$,
\[
\Lap\colon \etale{bz+\shp_{S_\infty\V}^\lambda} \isoto
\etale{\shp^{\lambda/(1-\lambda)}_{S_b\W}},\quad
(\alpha\infty,f)\mapsto(b+\beta0,g),
\] 
with $\beta=\lambda\alpha^{\lambda-1}\sigma_\lambda(f-bz)$.
\item[(iii)]
for $\lambda\in\Q_{>1}$,
\[
\Lap\colon \etale{\shp_{S_\infty\V}^\lambda} \isoto
\etale{\shp^{\lambda/(\lambda-1)}_{S_\infty\W}},\quad
(\alpha\infty,f)\mapsto(\beta\infty,g),
\] 
with $\beta=\lambda\alpha^{\lambda-1}\sigma_\lambda(f)$.
\end{itemize}
\end{lemma}

\begin{proof}
The inverse $\Lapa$ of $\Lap$ is obtained by replacing $\chi_\C$ with $\chi_\C^{-1}$ in Definition~\ref{lemdef}.

Concerning $\Lap$, recall that \eqref{eq:Lfg} is equivalent to \eqref{eq:wfzg'} and implies
\eqref{eq:wfzg}.
We will prove the statement using the relations \eqref{eq:wfzg}.

\medskip\noindent (i)
Since $f\in\shp_{S_a\V,a+\alpha0}^\lambda$, we can write $f(z) = c z_a^{-\lambda} + \sum_{\mu<\lambda} c_\mu z_a^{-\mu}$,
where $c=\sigma_\lambda(f)\in\C^\times$, $c_\mu\in\C$,
$\lambda,\mu\in\frac1\ram\Z$ for some $\ram\in\Z_{\geq 1}$, and we fixed a
determination of $z_a^{1/\ram}$ at $a+\alpha0$ compatible with that of $z_a^\lambda$. Hence, $w=f'(z)$ reads
\begin{equation}\label{eq:wf'1}
w = -\lambda c z_a^{-\lambda-1} - \sum_{\mu<\lambda} \mu c_\mu
z_a^{-\mu-1}.
\end{equation}
For $z_a=t\alpha$, with $t\in\R$ and $t\sim 0+$, we get $w =
t^{-\lambda-1}(-\lambda c \alpha^{-\lambda-1}+\mathrm{O}(t^\varepsilon))$ for some
$\varepsilon>0$. As $\lambda>0$, $w$ is near $\eta=\beta\infty$ with $\beta =
-\lambda c \alpha^{-\lambda-1}$.

The relation $z=-g'(w)$ reads $z_a = -(g(w)+aw)'$.
Note that \eqref{eq:wf'1} implies $(-w/\lambda c)^{-1/(\lambda+1)} =
z_a(1+\sum_{n>0}d_n z_a^{n/\ram})^{-1/(\lambda+1)} = z_a(1+\sum_{n>0}e_n
z_a^{n/\ram})$. Hence the Puiseux expansion of $(g(w)+aw)'=-z_a$ has maximum
exponent $-1/(\lambda+1)$ in $w$. Thus the Puiseux expansion of $g(w)+aw$ has
maximum exponent $\lambda/(\lambda+1)$ in $w$, that is, pole order
$\lambda/(\lambda+1)$ in $w_\infty=w^{-1}$.

\medskip\noindent (ii)
Since $f\in bz+\shp_{S_\infty\V,\alpha\infty}^\lambda$,
we can write $f(z)-bz = c z^\lambda + \sum_{\mu<\lambda} c_\mu z^\mu$.
Hence the relation $w= f'(z)$, which is equivalent to $w_b=(f(z)-bz)'$, reads
\begin{equation}\label{eq:wf'2}
w_b = \lambda c z^{\lambda-1} + \sum_{\mu<\lambda} \mu c_\mu z^{\mu-1}.
\end{equation}
For $z=t\alpha$ with $t\in\R$ and $t\sim +\infty$, we get $w_b = t^{\lambda-1}(\lambda c \alpha^{\lambda-1}+\mathrm{O}(t^{-\varepsilon}))$ for some
$\varepsilon>0$. Since $0<\lambda<1$, this implies that $w$ is near $\eta=b+\beta0$ with $\beta =
\lambda c \alpha^{\lambda-1}$.

Note that \eqref{eq:wf'2} implies $(w_b/\lambda c)^{1/(\lambda-1)} =
z(1+\sum_{n<0}d_n z^{n/\ram})^{1/(\lambda-1)} = z(1+\sum_{n<0}e_n z^{n/\ram})$. Hence $g'(w)=-z$ has pole order $1/(1-\lambda)$
at $w=b$. Thus $g(w)$ has pole order $\lambda/(1-\lambda)$ at $w=b$.

\medskip\noindent (iii)
Since $f\in \shp_{S_\infty\V,\alpha\infty}^\lambda$,
we can write $f(z) = c z^\lambda + \sum_{\mu<\lambda} c_\mu z^\mu$. Hence $w=f'(z)$ reads
\begin{equation}\label{eq:wf'3}
w = \lambda c z^{\lambda-1} + \sum_{\mu<\lambda} \mu c_\mu z^{\mu-1}.
\end{equation}

For $z=t\alpha$, with $t\in\R$ and $t\sim +\infty$, we get $w =
t^{\lambda-1}(\lambda c \alpha^{\lambda-1}+\mathrm{O}(t^{-\varepsilon}))$ for some
$\varepsilon>0$. Since $\lambda>1$, $w$ is near $\eta=\beta\infty$ with $\beta =
\lambda c \alpha^{\lambda-1}$.

Note that \eqref{eq:wf'3} implies $(w/\lambda c)^{1/(\lambda-1)} =
z(1+\sum_{n<0}d_n z^{n/\ram})^{1/(\lambda-1)} =
z(1+\sum_{n<0}e_n z^{n/\ram})$. Hence the Puiseux expansion of $g'(w)=-z$ has
maximum exponent $1/(\lambda-1)$ in $w$. Thus the Puiseux expansion of $g(w)$
has maximum exponent $\lambda/(\lambda-1)$ in $w$, that is, pole order
$\lambda/(\lambda-1)$ in $w_\infty=w^{-1}$.
\end{proof}

Let us show that the Legendre transform is compatible with the equivalence $\asim\theta$.

\begin{lemma}\label{lem:Lell}
For $i=1,2$, let $(a,\theta,f_i)$ be an admissible Puiseux germ on $\V$, and set $(b_i,\eta_i,g_i) \defeq \Lap (a,\theta,f_i)$.
Assume that $f_1 \asim\theta f_2$. Then
$b_1=b_2$, $\eta_1=\eta_2$ and $g_1 \asim{{\eta_1}} g_2$.
\end{lemma}

\begin{proof}
There are three possible situations:
\begin{itemize}
\item[(i)] $a\in\V$ and $f_1,f_2\in \shp_{S_a\V,\theta}^{(0,+\infty)}$,
\item[(ii)] $f_1,f_2\in bz+\shp_{S_\infty\V,\theta}^{(0,1)}$ for some $b\in\W$,
\item[(iii)] $f_1,f_2\in \shp_{S_\infty\V,\theta}^{(1,+\infty)}$.
\end{itemize}
Since the arguments are similar, let us only discuss case (i).

Set $\lambda_i\defeq\ord_a(f_i)\in\Q_{>0}$ for $i=1,2$.
Recall that the assumption $f_1 \asim\theta f_2$ is equivalent to $f_1-f_2\in\shp_{S_a\V,\theta}^{\le0}$.
Then, we must have $\lambda_1 = \lambda_2 \eqdef \lambda$ and $\sigma_\lambda(f_1) = \sigma_\lambda(f_2) \eqdef c$.
Hence, it follows from Lemma~\ref{lem:Leg}~(i) that $b_1=b_2=\infty$,
$\eta_1=\eta_2\eqdef\eta$, and $g_1,g_2\in-aw+\shp^{(0,1)}_{S_\infty\W,\eta}$.
We are thus left to prove that $g_1-g_2\in\shp_{S_\infty\W,\eta}^{\le0}$.

Let $\ram,q\in\Z_{\geq1}$ be such that $\lambda=q/\ram$ and $f_i\in\C\{z_a^{1/\ram}\}[z_a^{-1}]$.
Write $f_i(z) = c
z_a^{-q/p} + \sum_{n>0} c_n^{(i)} z_a^{(n-q)/\ram}$, and note that
$f_1-f_2\in\shp_{S_a\V,\theta}^{\le0}$ implies $c_n^{(1)}=c_n^{(2)}$ for $n<q$.
Recall that $z^{(i)}=g'_i(w)$ is equivalent to
\begin{align*}
w = f_i'(z^{(i)})
&= -\lambda c (z_a^{(i)})^{-\frac qp-1} - \sum_{n>0} \tfrac {q-n}\ram c_n^{(i)} (z_a^{(i)})^{\frac{n-q}\ram-1}  \\
&= -\lambda c (z_a^{(i)})^{-\frac{\ram +q}p} \bl  1+ \sum_{n>0} \tfrac{q-n}{qc} c_n^{(i)} (z_a^{(i)})^{\frac n\ram}\br .
\end{align*}
This gives
\begin{align*}
(-w/\lambda c)^{-\frac \ram{\ram+q}} &=
 z_a^{(i)}\bl  1+ \sum_{n>0} \tfrac{q-n}{qc} c_n^{(i)} (z_a^{(i)})^{\frac n\ram}\br^{-\frac \ram{\ram+q}} \\
&= z_a^{(i)}\bl 1+\sum_{n>0}d_n^{(i)} (z_a^{(i)})^{\frac n\ram}\br,
\end{align*}
where the coefficients $d_n^{(i)}$ only depend on $c_{n'}^{(i)}$
for $n'\leq n$. In particular, $d_n^{(1)}=d_n^{(2)}$ for $n<q$.
Inverting the series, we get 
\begin{align*}
g_i'(w)&=a+z_a^{(i)} \\
&=a+(-w/\lambda
c)^{-\frac \ram{\ram+q}}\bl 1+\sum_{n>0}e_n^{(i)} (w^{-\frac \ram{\ram+q}})^{\frac n\ram}\br, 
\end{align*}
where the
coefficients $e_n^{(i)}$ only depend on $d_{n'}^{(i)}$ for $n'\leq n$.
In particular, $e_n^{(1)}=e_n^{(2)}$ for $n<q$.
Noticing that $-\tfrac\ram{\ram+q}-\tfrac\ram{\ram+q}\tfrac n\ram = -1+\tfrac {q-n}{\ram+q}$.
It follows that $g_1'-g_2'\in
w^{-1}\C\{w^{-\frac1{\ram+q}}\}$. This implies
$g_1-g_2\in \C\{w^{-\frac1{\ram+q}}\}$.
The statement follows, recalling that $w_\infty=w^{-1}$.
\end{proof}

\subsection{Statement of the main result}
We can now state our main result.

\begin{theorem}[Stationary phase formula]\label{thm:EFou}
Let $(a,\theta,f)$ be an admissible Puiseux germ on $\V$, 
and let $(b,\eta,g) = \Lap(a,\theta,f)$.
Let $K\in\dereb_\Rc(\ifield_{\V_\infty})$ have normal form at $a$.
Then, for generic $\eta$, one has
\[
\G_{(b,\eta,g)}(\lap K) \simeq \G_{(a,\theta,f)}(K).
\]
\end{theorem}

Note that this implies the stationary phase formula of Theorem~\ref{thm:introstatio}.
In fact, if $\field=\C$ and $K=\solE[\V_\infty](\shm)$ for $\shm$ holonomic, then
$K$  has normal form at $a$ and $\lap K \simeq \solE[\W_\infty](\lap\shm)$.
Thus, by Propositions~\ref{pro:GrK} and \ref{pro:Fou}, the isomorphism in the statement
reads
\[
\bigl(\Gr_{g} \Psi_b(\lap K)\bigr)_\eta \simeq
\bigl(\Gr_{f} \Psi_a(K)\bigr)_\theta.
\]

\begin{proof}[Proof of Theorem~\ref{thm:EFou}]
At the level of Puiseux germs, let us consider the three possible cases:
\begin{itemize}
\item[(a)] $a\in\V$, $b=\infty$, $f\in\shp_{S_a\V}^{(0,+\infty)}$, and $g=-a w+\tilde g$ with $\tilde g\in\shp^{(0,1)}_{S_\infty\W}$,
\item[(b)] $a=\infty$, $b\in\V$, $f=bz+\tilde f$ with $\tilde f\in\shp_{S_\infty\V}^{(0,1)}$, and $g \in \shp^{(0,+\infty)}_{S_b\W}$,
\item[(c)] $a=\infty$, $b=\infty$, $f\in\shp_{S_\infty\V}^{(1,+\infty)}$, and $g \in \shp^{(1,+\infty)}_{S_\infty\W}$.
\end{itemize}

In case (a), assume that $a\neq 0$. 

Recall that $\tau_a\colon\V\to\V$ denotes the translation $\tau_a(z)=z+a$.
Consider the isomorphism
\[
\tau_a^*\colon S_a\V \isoto S_0\V, \quad
a+\alpha 0 \mapsto \alpha0.
\]
Note that
\[
\G_{(a,\theta,f)}(K) \simeq \G_{(0,\tau_a^* \theta,\tau_a^* f)}\bl\Eopb{\tau_a} K\br.
\]
By \eqref{eq:wfzg'}, one has $(\infty,\eta,\tilde g) = \Lap(0,\tau_a^* \theta,\tau_a^* f)$.
Moreover, Lemma~\ref{lem:Foutrans} implies
\[
\G_{(\infty,\eta,g)}\bigl(\lap K\bigr) 
\simeq 
\G_{(\infty,\eta,\tilde g)}\bigl(\lap(\Eopb{\tau_a} K)\bigr).
\]
We can thus reduce to the case $a=0$.

Similarly, in case (b) we can reduce to $b=0$.

Summarizing, it is enough to consider the cases:
\begin{itemize}
\item[(a$'$)] $a=0$, $b=\infty$, $f\in\shp_{S_a\V}^{(0,+\infty)}$, and $g\in\shp^{(0,1)}_{S_\infty\W}$,
\item[(b$'$)] $a=\infty$, $b=0$, $f\in\shp_{S_\infty\V}^{(0,1)}$, and $g \in \shp^{(0,+\infty)}_{S_b\W}$,
\item[(c)] $a=\infty$, $b=\infty$, $f\in\shp_{S_\infty\V}^{(1,+\infty)}$, and $g \in \shp^{(1,+\infty)}_{S_\infty\W}$.
\end{itemize}

Since the arguments are similar, let us only consider the case (a$'$). 

We have to prove the isomorphism
\[
\G_{(\infty,\eta,g)}(\lap K) \simeq \G_{(0,\theta,f)}(K).
\]
Since $K$ has normal form at $0$, by Proposition~\ref{pro:GrK} (i) it is equivalent to prove
\[
\G_{(\infty,\eta,g)}(\lap K) \simeq \field^{\overline N(f)},
\]
where $\overline N$ is the multiplicity class of $K$.

We will proceed by d\'evissage in $K$, in the category
$\dereb_+(\ifield_{\V_\infty})$.

\medskip\noindent(1)
For $r>0$, consider the distinguished triangle
\[
\opb\pi\field_{\{|z|<r\}}\tens K 
\to K \to
\opb\pi\field_{\{|z|\geq r\}}\tens K
\to[+1].
\]
Proposition~\ref{pro:Gzetabdd}~(i) below implies
\[
\G_{(\infty,\eta,g)}(\lap(\opb\pi\field_{\{|z|\geq r\}}\tens
K)) \simeq 0.
\]
We thus reduce to show that, for some $r>0$, one has
\[
\G_{(\infty,\eta,g)}\bigl(\lap(\opb\pi\field_{\{|z|<r\}}\tens
K)\bigr) \simeq  \field^{\overline N(f)}.
\]

\medskip\noindent(2)
Consider the distinguished triangle
\[
\opb\pi\field_{\{0<|z|<r\}} \tens K \to 
\opb\pi\field_{\{|z|<r\}} \tens K \to 
\opb\pi\field_{\{0\}} \tens K \to[+1].
\]
Since $\opb\pi\field_{\{0\}} \tens K$ is a finite direct sum of copies of $e(\field_{\{0\}}[n])$ for $n\in\Z$, it follows that $\lap(\opb\pi\field_{\{0\}} \tens K)$  is a finite direct sum of copies of $\field_{\W}^\enh[n+1]$.
Since $\ord_\infty(g)>0$, Lemma~\ref{lem:GEf12} gives
\[
\G_{(\infty,\eta,g)}(\field_{\W}^\enh) \simeq 0.
\]
We are thus left to show that, for some $r>0$,
\[
\G_{(\infty,\eta,g)}\bigl(\lap(\opb\pi\field_{\{0<|z|<r\}}\tens
K)\bigr) \simeq  \field^{\overline N(f)}.
\]

\medskip\noindent(3) For $r>0$ small enough, we can assume that
\[
\opb\pi\field_{\{0<|z|<r\}}\tens K \simeq
\field_\V^\enh\ctens F
\]
where $F\in\dere^0_\Rc(\field_{\V_\infty})$ satisfies $F\simeq \opb\pi\field_{\{0<|z|<r\}}\tens F$ and has a normal form at $0$ with multiplicity of class $\overline N$. 
Since $\lap(\field_\V^\enh\ctens F) \simeq 
\field_{\W}^\enh \ctens \lap F$, by Lemma~\ref{lem:Gzetabdd} we are reduced to show
\[
\G_{(\infty,\eta,g)}\bigl(\lap F\bigr) \simeq  \field^{\overline N(f)}.
\]
Then, we conclude by Proposition~\ref{pro:GlapF}~(i) below.
\end{proof}

\section{Microlocal arguments for the proof}\label{sec:micro}

We collect here some results used in the proof of the stationary phase formula, that we obtain using techniques from the microlocal study of sheaves of \cite{KS90}.

\subsection{Notations}\label{sse:Lnot}
Recall the identification $(T^*\V)^\R\times\R \simeq (T^*\V^\R)\times\R$ from \S\ref{sse:SS}.
We will write $(T^*\V)\times\R$ instead of $(T^*\V)^\R\times\R$, for short.

Let $((z;z^*),t)$ be the contact coordinates in $(T^*\V)\times\R$.
For $U\subset\V$ an open subset, and $\varphi\colon U\to\R$ a smooth function, we set
\[
\Lambda^{\varphi}_{U|\V} \defeq
\{((z;z^*),t)\semicolon z\in U,\ z^*=2\partial_z\varphi(z),\ t=-\varphi(z) \},
\]
a Lagrangian subset of $(T^*\V)\times\R$. Note that
\[
\Lambda^\varphi_{U|\V} = \SSE(\ex^\varphi_{U|\V}\bigr|_U).
\]
Note also that, if $f\in\O_X(U)$, then 
\[
\Lambda^{\Re f}_{U|\V} \defeq
\{((z;z^*),t)\semicolon z\in U,\ z^*=f'(z),\ t=-\Re f(z) \}.
\]

\subsection{Enhanced Fourier-Sato transform as quantization}
Recall the notation \eqref{eq:LambdaC}, the notations in \S\ref{sse:Lnot}, and the diagram \eqref{eq:chiCR}. 

For $(a,\theta,f)$ a Puiseux germ on $\V$, one has
\[
\Lambda^{\Re f}_{V|\V} = \Re(\Lambda_\C^{(a,\theta,f)}) ,
\]
for some $V\dotowns\theta$.
Since $\chi_\C$ underlies the Legendre transform, one has
\begin{equation}
\label{eq:Lambdafg}
\chi(\Lambda^{\Re f}_{V|\V}) = \Lambda^{\Re g}_{W|\W},
\end{equation}
where we set $(b,\eta,g)=\Lap(a,\theta,f)$, and where $W$ is a sectorial neighborhood of $\eta$.

Consider the enhanced sheaf
$\ex^{\Re f}_{V|\V} = \quot\field_{\{z\in V\semicolon t+\Re f(z)\geq 0\}}$.
Note that $t+\Re f(z) = 0$ is a smooth hypersurface of $V\times\R$.
Moreover,  one has
\[
SS^\enh(\ex^{\Re f}_{V|\V}\bigr|_{\opb\pi V}) = \Lambda^{\Re f}_{V|\V}.
\]
For $z_0\in V$, let $p_0=((z_0;f'(z_0)),-\Re f(z_0))\in \Lambda^{\Re f}_{V|\V}$.
Then, in the terminology of \cite[\S7.5]{KS90}, $\ex^{\Re f}_{V|\V}$ is of type $\field$ with shift $1/2$ at $p_0$ along $\Lambda^{\Re f}_{V|\V}$.\footnote{In fact, for $\overline p_0\in\opb\gamma\bl\Lambda^{\Re f}_{V|\V}\br$ with $\gamma(\overline p_0) = p_0$, \cite[\S7.5]{KS90} would say that $\field_{\{z\in V\semicolon t+\Re f(z)\geq 0\}}$ is of type $\field$ with shift $1/2$ at $\overline p_0$ along  $\opb\gamma\bl\Lambda^{\Re f}_{V|\V}\br$.}

Set $q_0=\chi(p_0)$.
In the terminology of \cite[\S7.2]{KS90}, the functor $\Lap$ is a quantization of
the contact transformation $\chi$.\footnote{In fact, \cite[\S7.2]{KS90} would refer to $\Phi_\Lap$ as a quantization of $\overline\chi$.}

\begin{proposition}
\label{pro:simp}
With the above notations, $\lap\ex^{\Re f}_{V|\V}$ is of type $\field$ with shift $1/2$ at $q_0$ along $\Lambda^{\Re g}_{W|\W}$.
\end{proposition}

\begin{proof}
With notations as in \S\ref{sse:SSenL}, consider the commutative diagram
\[
\xymatrix{
T^*_{\{t^*>0\}}(\V\times\R) \ar[r]^{\overline\chi} \ar[d]^\gamma & 
T^*_{\{s^*>0\}}(\W\times\R) \ar[d]^\gamma \\
(T^*\V)\times \R \ar[r]^\chi & (T^*\W)\times\R.
}
\]
Set $\overline p_0=(z_0,-\Re f(z_0);f'(z_0),1)$,
$\overline q_0=\overline \chi(\overline p_0)$.
Note that $\gamma(\overline p_0) = p_0$, $\gamma(\overline q_0) = q_0$,  and $\overline p_0\in\overline \Lambda^f$ for $\overline \Lambda^f \defeq \opb\gamma\Lambda^{\Re f}_{V|\V}$.

Set $\lambda_f =  T_{\overline p_0}\overline\Lambda^f$,
$\lambda_0 = T_{\overline p_0}\opb o( o(\overline p_0))$, and
$\mu_0 = T_{\overline q_0}\opb o( o(\overline q_0))$,
where $ o$ denotes the projection to the zero section.
Set $\lambda_1 = \overline\chi'(\overline p_0)^{-1}(\mu_0)$.

Denote by $\tau$ the Maslov inertia index, for which we refer to \cite[\S A.3]{KS90}.
Then \cite[Proposition 7.5.6]{KS90} states that $\lap\ex^{\Re f}_{V|\V}$ is of type $\field$ with shift $\frac12 - \frac12\tau(\lambda_0,\lambda_f,\lambda_1)$ at $q_0$ along $\Lambda^{\Re g}_{W|\W}$.
We are left to show that $\tau$ vanishes.

Setting $d_0=f'(z_0)$ and $e_0=f''(z_0)$, one has
\[
\overline\chi'(\overline p_0) =
{\tiny
\begin{pmatrix}
0 & 0 & 1 & -d_0 \\
d_0 & 1 & z_0 & -\Re(z_0d_0) \\
-1 & 0 & 0 & -z_0 \\
0& 0& 0& 1
\end{pmatrix}
}, \quad
\overline\chi'(\overline p_0)^{-1} =
{\tiny
\begin{pmatrix}
0 & 0 & -1 & -z_0 \\
-z_0 & 1 & d_0 & \Re(z_0d_0) \\
1 & 0 & 0 & d_0 \\
0& 0& 0& 1
\end{pmatrix}
},
\]
and
\begin{align*}
\lambda_0 &= \{(0,0;\omega,r) \semicolon \omega\in\C, r\in\R \}, \\ 
\lambda_f &= \{(\omega,-\Re (d_0 \omega);d_0 r + e_0 \omega,r) \semicolon \omega\in\C, r\in\R \}, \\ 
\lambda_1 &= \{(\omega,-\Re(d_0 \omega);d_0 r,r) \semicolon \omega\in\C, r\in\R \}.
\end{align*}
The Euler vector field at $\overline p_0$, given by
\[
\rho =  \{(0,0;d_0 r,r) \semicolon r\in\R \},
\]
is included in all above Lagrangians. Hence \cite[Theorem A.3.2 (v)]{KS90} gives
\[
\tau(\lambda_0,\lambda_f,\lambda_1) = \tau(\lambda_0/\rho,\lambda_f/\rho,\lambda_1/\rho),
\]
where the Maslov index on the right hand side is computed in $\rho^\bot/\rho$.
With the identification $\rho^\bot/\rho\simeq T_{(z_0,d_0)}T^*\V$, one has
\begin{align*}
\lambda_0/\rho &= \{(0;\omega)\semicolon \omega\in\C\}, \\
\lambda_f/\rho &= \{(\omega;e_0\omega)\semicolon \omega\in\C\}, \\
\lambda_1/\rho &= \{(\omega;0)\semicolon \omega\in\C\}.
\end{align*}
(Note that, with notations as in \S\ref{sse:Airy}, $\lambda_f/\rho = T_{(z_0,d_0)}C_{(a,\theta,f)}$.)
As these are complex Lagrangians, by \cite[Exercise A.7]{KS90} one has
\[
\tau(\lambda_0/\rho,\lambda_f/\rho,\lambda_1/\rho) = 0.
\]
\end{proof}

\subsection{Another vanishing result}

\begin{proposition}\label{pro:Gzetabdd} 
For $b\in\{0,\infty\}$, let $(b,\eta,g)$ be an admissible Puiseux germ on $\W$. Let $K\in\dereb_\Rc(\ifield_{\V_\infty})$. 
Assume one of the following conditions:
\begin{itemize}
\item[(i)]
$b=\infty$, $\ord_\infty(g)< 1$, and 
$\supp(K)\subset\{|z|\geq R\}$
for some $R>0$;
\item[(ii)]
$b=\infty$, $\ord_\infty(g)>1$, and 
$\supp(K)\subset\{|z|\leq r\}$
for some $r>0$;
\item[(iii)]
$b=0$, and 
$\supp(K)\subset\{|z|\leq r\}$ for some $r>0$.
\end{itemize}
Then, for a generic $\eta$, one has $\G_{(b,\eta,g)}(\lap K) \simeq 0$.
\end{proposition}

\begin{proof}
Since the arguments are similar, let us only discuss (i).

Let $F\in\dereb_\Rc(\field_{\V_\infty})$ satisfy
$K\simeq\field_\V^\enh\ctens F$ and $\supp(F)\subset\{|z|\geq R\}$.
Since $\lap K \simeq \field_{\W}^\enh\ctens \lap F$, by Lemma~\ref{lem:Gzetabdd} (iii) we have to show
\begin{equation}
\label{eq:GlapF0}
\G_{(\infty,\eta,g)}(\lap F) \simeq 0.
\end{equation}

Since $F\in\dereb_\Rc(\field_{\V_\infty})$, one has $\lap F\in\dereb_\Rc(\field_{\W_\infty})$.
By Lemma~\ref{lem:RcStruct}, for a generic $\eta$, there exists a sectorial open neighborhood $W$ of $\eta$ such that
\[
\opb\pi\field_W\tens\lap F\simeq \bl\DSum_{i\in I} \ex_{W|\W_\infty}^{\varphi_i}
[d_i]\br \dsum \bl\DSum_{j\in J} \ex_{W|\W_\infty}^{\range{\varphi_j^+}{\varphi_j^-}} [d_j]\br ,
\]
where $I$ and $J$ are finite sets, $d_i,d_j\in\Z$, and
$\varphi_i,\varphi_j^+,\varphi_j^-\colon W\to\R$ are real analytic and globally
subanalytic functions. 
Hence
\begin{equation}
\label{eq:SS10}
\SSE\bl(\lap F)|_{\opb\pi W}\br 
= 
\bl\Union_{i\in I}\Lambda_{W|\W}^{\varphi_i}\br \cup
\Union_{j\in J} (\Lambda_{W|\W}^{\varphi_j^+} \cup \Lambda_{W|\W}^{\varphi_j^-}).
\end{equation}

On the other hand, since $\supp(F)\subset\{|z|\geq R\}$,
\begin{align}
\label{eq:SS20}
\SSE(\lap F) 
&=\chi(\SSE(F)) \\ \notag
&\subset \chi(\{((z;z^*),t)\semicolon |z|\geq R\}) \\ \notag
&= \{((w;w^*),t)\semicolon |w^*|\geq R \}.
\end{align}

Comparing \eqref{eq:SS10} and \eqref{eq:SS20}, it follows that, for $\phi=\varphi_i,\varphi_j^+,\varphi_j^-$,
\begin{align*}
\Lambda_{W|\W}^\phi 
&= \{((w;w^*),t)\semicolon w\in W,\ w^*=\phi'(w),\ t=-\phi(w) \} \\
&\subset \{((w;w^*),t)\semicolon |w^*|\geq R \}.
\end{align*}
Hence $|\varphi_i'|,|\varphi_j^{+\prime}|,|\varphi_j^{-\prime}|\geq R$. Then
\eqref{eq:GlapF0} follows from Lemma~\ref{lem:phi'rR}.
\end{proof}

\subsection{A microlocal approach to multiplicity test}

\begin{proposition}\label{pro:GlapF}
Let $F\in\dere_\Rc(\field_{\V_\infty})$ have normal form at $a\in\{0,\infty\}$
with multiplicity $N$. 
Let $(a,\theta,f)$ be an admissible Puiseux germ on $\V$. 
Set $(b,\eta,g) = \Lap(a,\theta,f)$.
\begin{itemize}
\item[(i)]
Assume $a=0$, so that $b=\infty$ and $\ord_\infty(g) < 1$. Then, for $r>0$ small enough, and for a generic $\eta$, one has
\[
\G_{(\infty,\eta,g)}\bigl(\lap(\opb\pi\field_{\{0<|z|<r\}}\tens
F)\bigr) \simeq \field^{\overline N(f)}.
\]
\item[(ii)]
Assume $a=\infty$ and $\ord_\infty(f) \neq 1$, so that $b\in\{0,\infty\}$.
Then, for $R>0$ big enough, and for a generic $\eta$, one has
\[
\G_{(b,\eta,g)}\bigl(\lap(\opb\pi\field_{\{|z|>R\}}\tens
F)\bigr) \simeq \field^{\overline N(f)}.
\]
\end{itemize}
\end{proposition}

\begin{proof}
Since the proofs are similar, let us only discuss (i).

If $\overline N(f)>0$, by replacing $f$ with $\tilde f$ such that $[\tilde f]=[f]$ and $N(\tilde f)>0$, we may assume from the beginning that $\overline N(f) = N(f)$.

We will use some arguments from the proof of Proposition~\ref{pro:Gzetabdd}.

Set for short $F'=\opb\pi\field_{\{0<|z|<r\}}\tens F$.
Since $F'\in\dereb_\Rc(\field_{\V_\infty})$, one has $\lap F'\in\dereb_\Rc(\field_{\W_\infty})$.
By Lemma~\ref{lem:RcStruct}, for a generic $\eta$, there exists a sectorial open
neighborhood $W$ of $\eta$ such that
\begin{equation}
\label{eq:WlapF'}
\opb\pi\field_W\tens\lap F' \simeq \bl\DSum_{i\in I} \ex_{W|\W_\infty}^{\varphi_i}
[d_i]\br \dsum \bl\DSum_{j\in J} \ex_{W|\W_\infty}^{\range{\varphi_j^+}{\varphi_j^-}} [d_j]\br ,
\end{equation}
where $I$ and $J$ are finite sets, $d_i,d_j\in\Z$, and
$\varphi_i,\varphi_j^+,\varphi_j^-\colon W\to\R$ are real analytic and globally
subanalytic functions with $\varphi^-_j(w) < \varphi^+_j(w)$ for any $w\in W$ and $j\in J$.

Clearly, one has $\G_{(\infty,\eta,g)}\bigl(\lap F'
\bigr) \simeq \G_{(\infty,\eta,g)}\bigl(\opb\pi\field_W\tens\lap F'
\bigr)$. Hence, we are left to prove
\begin{equation}
\label{eq:GlapN}
\G_{(\infty,\eta,g)}\bigl(\opb\pi\field_W\tens\lap F' \bigr) \simeq \field^{N(f)}.
\end{equation}

By \eqref{eq:WlapF'}, one has
\begin{equation}
\label{eq:SS1}
\SSE(\lap F'|_{\opb\pi W}) 
= 
\bl\Union_{i\in I}\Lambda_{W|\W}^{\varphi_i}\br \cup
\Union_{j\in J} (\Lambda_{W|\W}^{\varphi_j^+} \cup \Lambda_{W|\W}^{\varphi_j^-}).
\end{equation}

On the other hand, since $F$ has normal form at $0$, for $r>0$ small enough one has
\[
\SSE(F\vert_{\opb\pi\{0<|z|<r\}}) \subset \Union_{\tilde\theta\in S_0\V,\ \tilde f\in N^{>0}_{\tilde \theta}}\Lambda_{V_{\tilde\theta}|\V}^{\Re\tilde f}.
\]
Using the fact that $\SSE(F')$ is Lagrangian, one deduces
\begin{align*}
\SSE(F') &= \SSE(\opb\pi\field_{\{0<|z|<r\}}\tens F) \\
&\subset \{ |z|=r \}
\union \bl\Union_{c\in \Sigma} \{ z=0,\ t=c \}\br \union \overline{\SSE(F\vert_{\opb\pi\{0<|z|<r\}})},
\end{align*}
where $\Sigma\subset\R$ is a finite set. 
By Proposition~\ref{thm:Tam} and \eqref{eq:Lambdafg}, one then has
\begin{align}
\label{eq:SS2}
\SSE(\lap F') &= \chi(\SSE(F')) \\ \notag
&\subset \{|w^*|=r \}
\union \bl\Union_{c\in \Sigma} \Lambda_{\W|\W}^{c}\br \union \overline{\bl\Union_{\tilde\theta\in S_0\V,\ \tilde f\in N^{>0}_{\tilde \theta}}\Lambda_{W_{\tilde\eta}|\W}^{\Re\tilde g}\br},
\end{align}
where $\tilde g$ and $\tilde\eta$ are defined by $(\infty,\tilde\eta,\tilde g) = \Lap(0,\tilde\theta,\tilde f)$, and $W_\eta\dotowns\tilde\eta$.

Comparing \eqref{eq:SS1} and \eqref{eq:SS2}, for $\phi=\varphi_i,\varphi_j^+,\varphi_j^-$ there are three possibilities:

\medskip\noindent(a)
$\Lambda_{W|\W}^\phi\subset\{|w^*|=r \}$, so that $|\phi'|=r$. In this case, by Lemma~\ref{lem:phi'rR} one has
\begin{equation}
\label{eq:Gg}
\begin{cases}
\G_{(\infty,\eta,g)}(\ex_{W|\W_\infty}^{\phi}) \simeq 0, \\
\G_{(\infty,\eta,g)}(\ex_{W|\W_\infty}^{\range{\phi}{\varphi_j^-}}) \simeq \G_{(\infty,\eta,g)}(\ex_{W|\W_\infty}^{\varphi_j^-})[-1], \\
\G_{(\infty,\eta,g)}(\ex_{W|\W_\infty}^{\range{\varphi_j^+}{\phi}}) \simeq \G_{(\infty,\eta,g)}(\ex_{W|\W_\infty}^{\varphi_j^+}).
\end{cases}
\end{equation}

\medskip\noindent(b)
$\Lambda_{W|\W}^\phi\subset\Lambda_{W|\W}^c$ for some $c\in\Sigma$. In this case, $\phi=c$. 
Since $\ord_\infty(g)>0$, then $g\not\asim \eta c$ and \eqref{eq:Gg} holds by Lemma~\ref{lem:GEf12}.

\medskip\noindent(c)
$\Lambda_{W|\W}^\phi\subset\Lambda_{W_{\tilde\eta}|\W}^{\Re\tilde g}$, where $(\infty,\tilde\eta,\tilde g) = \Lap(0,\tilde\theta,\tilde f)$ for some $\tilde\theta\in S_a\V$ and $\tilde f\in N_{\tilde\theta}^{>0}$. 
Up to shrinking $W$ and rotating $\tilde\theta$, we can assume $W=W_{\tilde\eta}$ and $\tilde\eta=\eta$.
Then $\phi=\Re\tilde g$.

\medskip\noindent(c-1)
If $\tilde g \not\asim \eta g$, then \eqref{eq:Gg} holds by Lemma~\ref{lem:GEf12}.

\medskip\noindent(c-2)
If $\tilde g \asim \eta g$, then $\theta=\tilde\theta$ and $\tilde f \asim \theta f$ by Lemma~\ref{lem:Lell}.
Then $f= \tilde f$ since $f$ is well situated with respect to $N$. Hence,  $g=\tilde g$.

By \eqref{eq:WlapF'} and the considerations (a)--(c) above,  we have
\[
\G_{(\infty,\eta,g)}\bigl( \opb\pi\field_W\tens\lap F' \bigr)
\simeq
\G_{(\infty,\eta,g)}\bigl( \DSum_{k\in I_g\cup J^+_g \cup J_g^-} \ex_{W|\W_\infty}^{\Re g}[d'_k] \bigr),
\]
where $I_g = \{ i\in I \semicolon \varphi_i=\Re g\}$, $J^\pm_g = \{ j\in J \semicolon \varphi^\pm_j=\Re g\}$, and
\[
d'_k=
\begin{cases}
d_k &\text{if }k\in I_g\cup J^+_g, \\
d_k-1 &\text{if }k\in J^-_g.
\end{cases}
\]
Note that $J^+_g \cap J_g^-=\emptyset$ since $\varphi^-_j<\varphi^+_j$ for any $j$.

In order to show \eqref{eq:GlapN}, we are thus left to prove
\begin{equation}
\label{eq:d'N}
\begin{cases}
d'_k=0 \text{ for any } k, \\
N(f) = \# I_g + \#(J^+_g \cup J_g^-).
\end{cases}
\end{equation}

Take $w_0$ near $\eta$ in $W$, and set $q_0=(w_0,g'(w_0),-\Re g(w_0))\in\Lambda_{W|\W}^{\Re g}$.
Setting $p_0\defeq\chi^{-1}(q_0)$, one has $p_0\in\chi^{-1}\bl\Lambda_{W|\W}^{\Re g}\br=\Lambda_{V_\theta|\V}^{\Re f}$.

Since $F'$ has normal form at $0$ with multiplicity $N$, we can decompose it
as in Definition~\ref{def:eform}. Thus $F'$ is of type $\field^{N(f)}$ with shift $1/2$ at $p_0$ along $\Lambda_{V_\theta|\V}^{\Re f}$. It follows from Proposition~\ref{pro:simp} that $\lap F'$ is of type $\field^{N(f)}$ with shift $1/2$ at $q_0$ along $\Lambda_{W|\W}^{\Re g}$.

On the other hand, according to the possibilities (a)--(c) above, \eqref{eq:WlapF'}
implies that $\lap F'$ is of type $T$
with shift $1/2$ at $q_0$ along $\Lambda_{W|\W}^{\Re g}$, where
\[
T = \DSum_{k\in I_g\cup J^+_g \cup J_g^-}\field{}[d'_k].
\]

Then, one has $T \simeq \field^{N(f)}$. This implies \eqref{eq:d'N}.

\end{proof}

\end{document}